\documentclass{amsart}

\usepackage{amssymb}
\usepackage{amsmath}
\usepackage{amsthm}
\usepackage{amscd}
\long\def\comment#1\endcomment{}

\usepackage{color}

\def\mcgb{\MCG (S)_e^\omega}

\def\fo{\phi_\omega}
\def\sm{\seq\mu}
\def\sn{\seq\nu}
\def\sa{\seq\alpha}
\def\sb{\seq\beta}
\def\sc{\seq\gamma}

\def\delg{\seq\Delta_g}
\def\delh{\seq\Delta_h}
\def\bs{\mathbf{S}}
\def\bu{\mathbf{U}}
\def\lu{\mathcal{M}^\omega \mathbf{U}}
\def\plu{\pi_{\mathcal{M}^\omega \mathbf{U}}}
\def\Fix{\mathrm{Fix}}
\def\rt{\mathrm{retr}}
\def\dlu{\mathrm{dist}_{\mathbf{U}}}
\def\td{\widetilde{\mathrm{dist}}}
\def\tdlu{\widetilde{\mathrm{dist}}_{\mathbf{U}}}
\def\tdls{\widetilde{\mathrm{dist}}_{\mathbf{S}}}
\def\bv{\mathbf{V}}
\def\buv{_{\bu , \bv}}
\def\lv{\mathcal{M}^\omega \mathbf{V}}
\def\plv{\pi_{\mathcal{M}^\omega \mathbf{V}}}

\def\tdlv{\widetilde{\mathrm{dist}}_{\mathbf{V}}}
\def\by{\mathbf{Y}}
\def\bz{\mathbf{Z}}

\def\MCG{\mathcal{MCG}}
\def\MM{\mathcal {M}}
\def\ultrau{\mathcal{U}}
\def\upss{\Pi \ultrau /\omega}

\def\UM{\MM(S)_b^{\omega}}
\newcommand{\coneM}{\mathrm{Con}^\omega(\MM (S); (d_n))}
\newcommand{\coneMx}{\mathrm{Con}^\omega(\MM (S); (x_n), (d_n))}
\newcommand{\coneMmu}{\mathrm{Con}^\omega(\MM (S); (\mu_n^0), (d_n))}

\def\AM{\mathcal{AM}}
\def\GM{\mathcal {GM}}
\def\stab{{\rm stab}}
\def\CC{{\mathcal C}}

\def\QQ{{\mathcal Q}}

\newcommand{\Cutp}{{\mathrm{Cutp}}\, }

\newcommand{\tdist}{\widetilde{\mathrm{dist}}}

\def\diam{{\rm diam}}

\def\ulim{\lim_\omega}

\newcommand {\Notat}{\noindent {\it{Notation}}:} 
\newcommand{\dist}{{\mathrm{dist}}}
\newcommand{\dcs}{\dist_{\CC (S)}}
\newcommand{\dcy}{\dist_{\CC (Y)}}
\newcommand{\pcs}{\pi_{\CC (S)}}

\newcommand{\dmm}{\dist_{\MM(S)}}
\newcommand{\dam}{\dist_{\coneMx}}

\newcommand {\q}{\mathfrak q} 
\newcommand {\g}{\mathfrak g} 
\newcommand {\pgot}{\mathfrak p}
\newcommand {\cf}{\mathfrak c}
\newcommand {\ft}{\mathfrak t}
\newcommand {\fct}{\mathfrak T}
\newcommand{\Stab}{{\mathrm{Stab}}}
\newcommand{\la}{\langle}
\newcommand{\ra}{\rangle}

\newcommand{\fh}{\mathfrak{h}}

\newcommand{\uas}{{$\omega$-almost surely}}
\newcommand{\uass}{$\omega$-almost surely }

\newcommand{\oae}{$\omega$-almost every }

\newcommand {\N}{\mathbb{N}} 
\newcommand {\R}{\mathbb{R}} 

\newcommand{\card}{{\mathrm{card}}\, }

\newcommand{\pp}{{\mathcal P}}

\newcommand {\free}{\mathbb{F}} 

\def\calp{\mathcal{P}}   
\def\calt{\mathcal{T}}   
\def\calc{\mathcal{C}}   

\def\ck{\mathcal{K}}   
\def\yy{\mathcal{Y}}   

\def\Z{{\mathbb Z}}

\newcommand\ulimn{\ulim}

\catcode`\@=11
\long\def\@savemarbox#1#2{\global\setbox#1\vtop{\hsize\marginparwidth
  \@parboxrestore\tiny\raggedright #2}}

\newcommand\notpitchfork{\not\pitchfork}

\newcommand{\tsh}[1]{\left\{\kern-.9ex\left\{#1\right\}\kern-.9ex\right\}}
\newcommand{\Tsh}[2]{\tsh{#2}_{#1}}

\newcommand {\me}{\medskip}

\newtheorem{theorem}{Theorem}[section]
\newtheorem{proposition}[theorem]{Proposition}
\newtheorem{corollary}[theorem]{Corollary}

\newtheorem{question}[theorem]{Question}
\newtheorem{lemma}[theorem]{Lemma}

\theoremstyle{definition}
\newtheorem{defn}[theorem]{Definition}
\theoremstyle{definition}
\newtheorem{notation}[theorem]{Notation}

\theoremstyle{remark}
\newtheorem{remark}[theorem]{Remark}
\newtheorem{remarks}[theorem]{Remarks}

\newtheorem{cvn}[theorem]{Convention}
\def\square{\hfill${\vcenter{\vbox{\hrule height.4pt \hbox{\vrule
width.4pt
height7pt \kern7pt \vrule width.4pt} \hrule height.4pt}}}$}

\def\<{\langle}
\def\>{\rangle}

\newcommand\seq[1]{\mathchoice{\mbox{\boldmath$#1$}}{\mbox{\boldmath$#1$}}{\mbox{\boldmath$\scriptstyle#1$}}{\mbox{\boldmath$\scriptstyle#1$}}}

\newcommand\seqrep[1]{\lm_\omega\left(#1_{n}\right)}

\newcommand\subseq[1]{\mbox{\boldmath$\scriptstyle#1$}}

\newcommand\ultra[1]{( #1_{n})^{\omega}}

\newcommand{\Con}{{\mathrm{Con}}}
\newcommand{\lm}{{\lim}}
\newcommand {\iv}{^{-1}}
\newcommand{\lio}[1]{\lm_\omega\left(#1\right)}
\newcommand{\ko}[1]{\Con^\omega(#1)}

\newcommand{\co}{\colon\thinspace}

\newcommand{\muk}{\mu^{(k)}}
\newcommand{\nuk}{\nu^{(k)}}

\def\fg{{\mathfrak g}}
\def\fu{{\mathfrak U}}
\def\fp{{\mathfrak P}}
\newcommand{\base}{\operatorname{base}}

\def\boundary{\partial}

\begin{document}

\title[Median structures and homomorphisms into mapping class groups]{Median structures on asymptotic cones and homomorphisms into mapping class groups}
\author{Jason Behrstock}\thanks{The research of the first author was supported in part by
NSF grant DMS-0812513.}
\address{Lehman College,
City University of New York,
U.S.A.}
\email{jason.behrstock@lehman.cuny.edu}
\author{Cornelia Dru\c{t}u}\thanks{The research of the second author was supported in part by
the ANR project ``Groupe de recherche de G\'eom\'etrie et Probabilit\'es dans
les Groupes''.}
\address{Mathematical Institute,
24-29 St Giles,
Oxford OX1 3LB,
United Kingdom.}
\email{drutu@maths.ox.ac.uk}
\author{Mark Sapir}\thanks{The research of the third author was supported in part by NSF grants DMS-0455881 and DMS-0700811 and a BSF (US-Israeli) grant 2004010.}
\address{Department of Mathematics,
Vanderbilt University,
Nashville, TN 37240, U.S.A.}
\email{m.sapir@vanderbilt.edu}
\date{ }

\begin{abstract}
The main goal of this paper is a detailed study of asymptotic cones of
the mapping class groups. In particular, we prove that every
asymptotic cone of a mapping class group has a bi-Lipschitz
equivariant embedding into a product of real trees, sending limits of
hierarchy paths onto geodesics, and with image a
median subspace. One of the applications is that a group with
Kazhdan's property (T) can have only finitely many pairwise
non-conjugate homomorphisms into a mapping class group. We also give a
new proof of the rank conjecture of Brock and Farb (previously
proved by Behrstock and Minsky, and independently by Hamenstaedt).

\end{abstract}

\maketitle

\tableofcontents

\section{Introduction}

Mapping class groups of surfaces (denoted in this paper by
$\MCG(S)$, where $S$ always denotes a compact connected orientable
surface) are very interesting
geometric objects, whose geometry is not yet completely understood.
Aspects of their geometry are especially striking when compared
with lattices
in semi-simple Lie groups.  Mapping class groups are known to share some properties with lattices in rank 1 and some others with lattices in higher rank
semi-simple Lie groups.  For instance the intersection
pattern of quasi-flats in $\MCG(S)$ is reminiscent of intersection patterns
of quasi-flats in uniform lattices of higher rank semi-simple groups
\cite{BehrstockKleinerMinskyMosher}.  On the other hand, the
pseudo-Anosov elements in $\MCG(S)$ are `rank 1 elements', i.e. the
cyclic subgroups generated by them are quasi-geodesics satisfying the
Morse property (\cite{FarbLubotzkyMinsky}, \cite{Behrstock:thesis},
\cite{DrutuMozesSapir}).

Non-uniform lattices in rank one semi-simple groups are (strongly)
relatively hyperbolic \cite{Farb:RelHyp}.  As mapping class groups act
by isometries on their complex of curves, and the latter are
hyperbolic \cite{MasurMinsky:complex1}, mapping class groups are
weakly relatively hyperbolic with respect to finitely many stabilizers of multicurves.  On the other hand, mapping class groups are not strongly
relatively hyperbolic with respect to any collection of subgroups (\cite{AASh:RelHypMCG}, \cite{BehrstockDrutuMosher:thick1}), they are not even metrically relatively hyperbolic with respect to any collection of subsets \cite{BehrstockDrutuMosher:thick1}.  Still,
$\MCG(S)$ share further properties with relatively hyperbolic groups.
A subgroup of $\MCG(S)$ either contains a pseudo-Anosov element or it is
parabolic, that is it stabilizes a (multi-)curve in $S$
\cite{Ivanov:subgroups}.  A similar property is one of the main `rank
1' properties of relatively hyperbolic groups
\cite{DrutuSapir:TreeGraded}.

Another form of rank 1 phenomenon, which is also a weaker version of relative hyperbolicity, is existence of cut-points in the asymptotic cones. In an asymptotic cone of a relatively hyperbolic group every point is a cut-point. In general, a set of cut-points $\calc$ in a geodesic metric space determines a \emph{tree-graded structure} on that space \cite[Lemma 2.30]{DrutuSapir:TreeGraded}, that is it determines a collection of proper geodesic subspaces, called {\em pieces}, such that every two pieces intersect in at most one point (contained in $\calc$) and every simple loop is contained in one piece. One can consider as pieces maximal path connected subsets with no cut-points in $\calc$, and singletons. When taking the quotient of a tree-graded space $\free$ with respect to the closure of the equivalence relation `two points are in the same piece' (this corresponds, roughly, to shrinking all pieces to points) one obtains a real tree $T_\free$ \cite[Section 2.3]{DrutuSapir:splitting}.

It was proved in \cite{Behrstock:asymptotic} that in the asymptotic cones of mapping class groups every point is a cut-point, consequently such asymptotic cones are tree-graded. The pieces of the corresponding tree-graded structure, that is the maximal path connected subsets with no cut-points, are described in \cite{BehrstockKleinerMinskyMosher} (see Theorem \ref{classifypieces} in this paper). Further information about the pieces  is contained in Proposition \ref{unique}. The canonical projection of an asymptotic cone $\AM (S)$ of a mapping class group onto the corresponding asymptotic cone $\mathcal{AC} (S)$ of the complex of curves (which is a real tree, since the complex of curves is hyperbolic) is a composition of  the projection of $\AM (S)$ seen as a tree-graded space onto the quotient tree described above, which we denote by $T_S$, and a projection of $T_S$ onto $\mathcal{AC} (S)$. The second projection has large pre-images of singletons, and is therefore very far from being injective (see Remark \ref{trees}).

Asymptotic cones were used to prove quasi-isometric rigidity of lattices in higher rank semi-simple Lie groups (\cite{KleinerLeeb:buildings}, \cite{Drutu:HrankRigidity}; see \cite{EskinFarb}, \cite{Schwartz:Diophantine}, \cite{Eskin:HigherRank} for proofs without the use of asymptotic cones) and of relatively hyperbolic groups (\cite{DrutuSapir:TreeGraded}, \cite{Drutu:RelHyp}). Unsurprisingly therefore asymptotic cones of mapping class groups play a central part in the proof of the quasi-isometric rigidity of mapping class groups (\cite{BehrstockKleinerMinskyMosher}, \cite{Hamenstadt:qirigidity}), as well as in the proof of the Brock-Farb rank conjecture that the rank of every quasi-flat in $\MCG(S)$ does not exceed $\xi (S)=3g+p-3$, where $g$ is the genus of the surface $S$ and $p$ is the number of punctures (\cite{BehrstockMinsky:rankconj}, \cite{Hamenstadt:qirigidity}). Many useful results about the structure of asymptotic cones of mapping class groups can be found in \cite{BehrstockKleinerMinskyMosher,Hamenstadt:qirigidity,BehrstockMinsky:rankconj}.

In this paper, we continue the study of asymptotic cones of mapping class groups, and show that the natural metric on every asymptotic cone of $\MCG(S)$ can be deformed in an equivariant and bi-Lipschitz way, so that the new metric space is inside an $\ell_1$-product of $\R$-trees, it is a median space and the limits of hierarchy paths become geodesics. To this end, the projection of the mapping class group onto mapping class groups of subsurfaces (see Section \ref{projmy}) is used to define the projection of an asymptotic cone $\AM (S)$ onto limits $\MM (\bu )$ of sequences of  mapping class groups of subsurfaces $\bu = (U_n)$. A limit $\MM (\bu )$ is isometric to an asymptotic cone of $\MCG(Y)$ with $Y$ a fixed subsurface, the latter is a tree-graded space, hence $\MM (\bu )$ has a projection onto a real tree $T_\bu$ obtained by shrinking pieces to points as described previously. These projections allow us to construct an embedding of $\AM (S)$ into an $\ell_1$-product of $\R$-trees.

\begin{theorem}[Theorem \ref{tilde}, Theorem \ref{thmedian}]\label{prodtrees}
The map $\psi\co \AM (S) \to \prod_{\bu} T_\bu $
whose components
are the canonical projections of $\AM (S)$ onto $T_\bu $ is a
bi-Lipschitz map, when $\prod_{\bu} T_\bu $ is endowed with the $\ell^1$-metric. Its image $\psi (\AM (S) )$ is a median space. Moreover $\psi$ maps limits of hierarchy paths onto geodesics in $\prod_{\bu} T_\bu $.
\end{theorem}

\begin{remark}
The subspaces of a median space that are strongly convex (i.e. for
every two points in the subspace all geodesics connecting these points are
in the subspace) are automatically median. Convex subspaces (i.e.
containing one geodesic for every pair of points) are not necessarily
median. The image of the asymptotic cone in Theorem 1.1 is convex but not
strongly convex. Indeed in the asymptotic cone every point is a cut-point
[Beh06], and the cone itself is not a tree. On the other hand, a strongly
convex subspace $Y$ with a (global) cut-point in a product of trees
(possibly infinite uncountable) $\prod_{i\in I} T_i$ must be a factor tree
$T_{j}\times \prod_{i\neq j} \{ a_i\}$. Indeed, let $x$ be a cut-point of
$Y$ and $y,z$ two points in two distinct connected components of
$Y\setminus \{ x\}$. If there existed at least two indices $i,j\in I$ such
that $y_i\neq z_i$ and $y_j \neq z_j$ then $y,z$ could be connected by a
geodesic in  $Y\setminus \{ x\}$, a contradiction. It follows that there
exists a unique $i\in I$ such that $\{ y_i, x_i, z_i\}$ are pairwise
distinct. By repeating the argument for every two points in every two
components of $Y\setminus \{ x\}$ we obtain the result.
\end{remark}

The bi-Lipschitz equivalence between the limit metric on $\AM (S)$ and the pull-back of the $\ell^1$-metric on $\prod_{\bu } T_\bu $ yields a distance formula in the asymptotic cone $\AM (S)$,
similar to the  Masur--Minsky distance formula for the marking complex \cite{MasurMinsky:complex2}.

The embedding $\psi$ allows us to give in Section \ref{sec:dimcone} an
alternative proof of the Brock-Farb conjecture, which essentially
follows the ideas outlined in \cite{Behrstock:thesis}.  We prove that
the covering dimension of any locally compact subset of any asymptotic cone $\AM (S)$ does not
exceed $\xi (S)$ (Theorem \ref{rank}). This is done by showing that for every compact subset $K$ of $\AM (S)$ and every $\epsilon >0$ there exists an $\epsilon$-\emph{map}
$f:K \to X$ (i.e. a
continuous map with diameter of $f^{-1}(x)$ at most $\epsilon$ for
every $x\in X$) from $K$ to a product of finitely many $\R$-trees $X$ such that $f(K)$ is of dimension at most $\xi (S)$.
This, by a standard statement from dimension theory, implies that the dimension of $\AM (S)$ is at most $\xi (S)$.

One of the typical `rank 1' properties of groups is the following result essentially due to Bestvina \cite{Bestvina:degener} and Paulin \cite{Paulin:arbres}: if a group $A$ has infinitely many injective homomorphisms $\phi_1,\phi_2,...$ into a hyperbolic group $G$ which are pairwise non-conjugate in $G$, then $A$ splits over a virtually abelian subgroup. The reason for this is that $A$ acts without global fixed point on an asymptotic cone of $G$ (which is an $\R$-tree) by the natural action: \begin{equation}\label{act} a\cdot (x_i)=(\phi_i(a)x_i).\end{equation}
Similar statements hold for relatively hyperbolic groups (see \cite{OhshikaPotyag}, \cite{DelzantPotyag:kleinian}, \cite{Groves:limit}, \cite{Groves:Hopf}, \cite{Groves:MakaninR}, \cite{DrutuSapir:splitting}, \cite{BelegradekSz}).

It is easy to see that this statement does not hold for mapping class
groups.  Indeed, consider the right angled Artin group $B$
corresponding to a finite graph $\Gamma$ ($B$ is generated by the set
$X$ of vertices of $\Gamma$ subject to commutativity relations: two
generators commute if and only if the corresponding vertices are
adjacent in $\Gamma$).  There clearly exists a surface $S$ and a
collection of curves $X_S$ in one-to-one correspondence with $X$ such
that two curves $\alpha,\beta$ from $X_S$ are disjoint if and only if
the corresponding vertices in $X$ are adjacent (see \cite{CP}).  Let $d_\alpha$,
$\alpha\in X$, be the Dehn twist corresponding to the curve $\alpha$.
Then every map $X\to \MCG(S)$ such that $\alpha\mapsto
d_\alpha^{k_\alpha}$ for some integer $k_\alpha$ extends to a
homomorphism $B\to \MCG(S)$.  For different choices of the $k_\alpha$ one obtains homomorphisms that are pairwise non-conjugate in
$\MCG(S)$ (many of these homomorphisms are injective by \cite{CP}). This set of homomorphisms can be further increased by changing the collection of curves $X_S$ (while preserving the intersection patterns), or by considering more complex subsurfaces of $S$ than just simple closed curves (equivalently, annuli around those curves), and replacing Dehn twists with pseudo-Anosovs on those subsurfaces. Thus $B$ has many pairwise non-conjugate homomorphisms into $\MCG(S)$. But the group $B$ does not necessarily split over any
``nice" (e.g. abelian, small, etc.)  subgroup.

Nevertheless, if a group $A$ has infinitely many pairwise non-conjugate
homomorphisms into a mapping class group $\MCG(S)$, then it acts
as in (\ref{act}) without global fixed point on an asymptotic cone $\AM (S)$ of $\MCG(S)$.  Since $\AM (S)$
is a tree-graded space, we apply the theory of actions of groups on
tree-graded spaces from \cite{DrutuSapir:splitting}.  We prove
(Corollary \ref{ii}) that in this case either $A$ is virtually
abelian, or it splits over a virtually abelian subgroup, or the action
(\ref{act}) fixes a piece of  $\AM (S)$ set-wise.

We also prove, in Section \ref{sect:T}, that the action
(\ref{act}) of $A$ has unbounded orbits, via an inductive argument on the complexity of the surface $S$. The main ingredient is a careful analysis of the sets of fixed points of pure elements in $A$, comprising a proof of the fact that a pure element in $A$ with bounded orbits has fixed points (Lemmas \ref{middlepA} and \ref{middlered}), a complete description of the sets of fixed points of pure elements in $A$ (Lemmas \ref{fixpA} and \ref{fixred}), and an argument showing that distinct pure elements in $A$ with no common fixed point generate a group with infinite orbits (Lemma \ref{2redgen} and Proposition \ref{nredgen}). The last argument follows the same outline as the similar result holding for isometries of an $\R$-tree, although it is much more complex.

The above and the fact that a group with property (T) cannot act on a
median space with unbounded orbits (see Section \ref{sect:T} for
references)
allow us to apply our results in the case when $A$ has property (T).

\begin{theorem}[Corollary \ref{cor:thmT}]\label{introT}
A finitely generated group with property (T) has at most finitely many pairwise non-conjugate homomorphisms into a mapping class group.
\end{theorem}

D. Groves announced the same result (personal communication).

A result similar to Theorem \ref{introT} is that given a one-ended group $A$, there are finitely many pairwise non-conjugate injective homomorphisms $A\to \MCG(S)$ such that every non-trivial element in $A$ has a pseudo-Anosov image (\cite{Bowditch:atoroidal}, \cite{DahmaniFujiwara}, \cite{Bowditch:oneended}). Both this result and Theorem \ref{introT} should be seen as evidence that there are few subgroups (if any) with these properties in the mapping class group of a surface. We recall that the mapping class group itself does not have property (T) \cite{Andersen:mcg-T}.

Via Theorem~\ref{introT}, an affirmative answer to the
following natural question would yield a new proof of Andersen's
result that $\MCG(S)$ does not have property (T):

\begin{question}
Given a surface $S$, does there exists a surface $S'$ and an infinite set of pairwise nonconjugate homomorphisms from $\MCG(S)$ into $\MCG(S')$?
\end{question}

See \cite{AramayonaLeiningerSouto} for some interesting
constructions of nontrivial homomorphisms from the mapping class
group of one surface into the mapping class group of another, but
note that their constructions only yield finitely many conjugacy
classes of homomorphisms for fixed $S$ and $S'$.

\medskip

\noindent {\textbf{Organization of the paper.}}
In Section \ref{sec:back} we recall results on asymptotic cones,
complexes of curves, and mapping class groups, while in Section
\ref{stg} we recall properties of tree-graded metric spaces and prove
new results on groups of isometries of such spaces.  In Section
\ref{sacmcg} we prove Theorem \ref{prodtrees}, and we give a new proof
that the dimension of any locally compact subset of any asymptotic cone of $\MCG(S)$ is at most $\xi
(S)$. This provides a new proof of the Brock-Farb
conjecture.
 In Section \ref{sec:actions} we describe further the asymptotic
cones of $\MCG(S)$ and deduce that for groups not virtually abelian
nor splitting over a virtually abelian subgroup sequences of pairwise
non-conjugate homomorphism into $\MCG(S)$ induce an action on the
asymptotic cone fixing a piece (Corollary~\ref{ii}).  In Section \ref{sect:T} we study in more detail actions on asymptotic cones of $\MCG(S)$ induced by sequences of pairwise
non-conjugate homomorphisms, focusing on the relationship between bounded orbits and existence either of fixed points or of fixed multicurves. We prove that if a group $G$ is not virtually cyclic and has infinitely
many pairwise non-conjugate homomorphisms into $\MCG(S)$, then $G$ acts on the
asymptotic cone of $\MCG(S)$ viewed as a median space with unbounded orbits. This allows us to prove
Theorem \ref{introT}.

\medskip

\noindent {\textbf{Acknowledgement.}} We are grateful to Yair Minsky
and Lee Mosher for helpful conversations. We also thank the referee
for a very careful reading and many useful comments.

\section{Background}\label{sec:back}

\subsection{Asymptotic cones}\label{ac}

A \textit{non-principal ultrafilter} $\omega$ over a countable set
$I$ is a finitely additive measure on the class $\mathcal{P}(I)$
of subsets of $I$, such that each subset has measure either $0$ or
$1$ and all finite sets have measure 0. Since we only use
non-principal ultrafilters, the word non-principal will be omitted
in what follows.

If a statement $S(i)$ is satisfied for all $i$ in a set $J$ with
$\omega (J)=1$, then we say that
$S(i)$ holds {\em $\omega$--a.s.}

Given a sequence of sets $(X_n)_{n\in I}$ and an ultrafilter
$\omega$, the {\em ultraproduct corresponding to $\omega$}, $\Pi
X_n/\omega$, consists of equivalence classes of sequences
$(x_{n})_{n\in I}$, with $x_n\in X_n$, where two sequences
$(x_{n})$ and
$(y_{n})$ are identified if $x_n=y_n$ $\omega$--a.s. The equivalence
class of a sequence $x=(x_{n})$ in $\Pi X_n/\omega$ is denoted
either by $x^{\omega}$ or by $\ultra x$. In particular, if all
$X_n$ are equal to the same $X$, the ultraproduct is called the
{\em ultrapower} of $X$ and it is denoted  by $\Pi X/\omega$.

If $G_n$, $n\ge 1$, are groups then $\Pi G_n/\omega$ is again a group
with the multiplication law
$\ultra x\ultra y =(x_{n}y_{n})^{\omega}$.

If $\Re$ is a relation on $X$, then one can define a
relation $\Re_\omega$ on $\Pi X/\omega$ by setting $\ultra x
\Re_\omega \ultra y$ if and only if $x_n\, \Re\, y_n$ \uas.

\begin{lemma}[\cite{universalalg}, Lemma 6.5]
    \label{D}

Let $\omega$ be an ultrafilter over $I$ and let $(X_n)_{n\in I}$ be a
sequence of sets which $\omega$--a.s.\ have cardinality at
most $N$. Then the ultraproduct $\Pi X_i/\omega$
has cardinality at most~$N$.
\end{lemma}

For every sequence of points $(x_n)_{n\in I}$ in a topological space
$X$, its \emph{$\omega$-limit},
denoted by $\ulim x_n$, is a point $x$ in $X$ such that every
neighborhood $U$ of $x$ contains $x_n$
for \oae~$n$. If a metric space $X$ is Hausdorff and  $(x_n)$ is a
sequence in $X$, then when the
$\omega$--limit of  $(x_n)$ exists, it is unique. In a
compact metric space every sequence
has an $\omega$--limit \cite[$\S $ I.9.1]{Bourbaki}.

\begin{defn}[ultraproduct of metric spaces]

Let $(X_n,\dist_n)$, $n\in I$, be a sequence of metric spaces and
let $\omega$ be an ultrafilter over $I$. Consider the ultraproduct
$\Pi X_n/\omega$. For every two points $x^\omega=\ultra x,
y^\omega=\ultra y$ in $\Pi X_n/\omega$ let
$$D(x^\omega ,y^\omega )=\lm_\omega \dist_{n}(x_n,y_n)\, .$$

The function $D\colon \Pi
X_n/\omega\times \Pi X_n/\omega\to [0,\infty]$ is a pseudometric, that is, it satisfies all
the properties of a metric except it can take infinite values and need not satisfy $D(x^\omega,y^\omega)=0 \Rightarrow x^\omega=y^\omega$. We may make the function $D$ take only finite values by restricting it to a subset of the ultrapower in the following way.

Consider an {\em observation point} $e=\ultra e$ in $\Pi X_n/\omega$
and define $\Pi_e X_n/\omega$ to
be the subset of $\Pi X_n/\omega$ consisting of elements which are
finite distance from $e$ with
respect to $D$. Note that transitivity of  $D$ implies that the distance is finite between any pair of points in
  $\Pi_e X_n/\omega$.
\end{defn}

Note that if $X$ is a group $G$ endowed with a word metric then
$\Pi_1 G/\omega$ is a subgroup of the
ultrapower of $G$.

\begin{defn}[ultralimit of metric spaces]

The {\em $\omega$-limit  of the metric spaces
$(X_n,\dist_n)$ relative to the observation point $e$} is the
metric space obtained from $\Pi_e X_n/\omega$ by identifying all
pairs of points $x^\omega,y^\omega$ with $D(x^\omega,y^\omega)=0$; this space is denoted
$\ulim(X_n, e)$. When there is no need to specify the ultrafilter $\omega$ then we also call
 $\ulim(X_n, e)$ {\em ultralimit  of $(X_n,\dist_n)$ relative to $e$.}

The equivalence class of a
sequence $x=(x_n)$ in $\ulim(X_n, e)$ is denoted either by $\ulim{x_n}$ or by $\seq x$.
\end{defn}

Note that if $e,e'\in \Pi X_n/\omega$ and $D(e,e')<\infty$ then
$\ulim(X_n,e)=\ulim(X_n,e')$.

\begin{defn}
 For a sequence $A=(A_n)_{n\in I}$ of subsets $A_n \subset X_n$, we write $\seq A$ or $\ulim{A_n}$  to denote the subset of  $\ulim(X_n, e)$
    consisting of all the elements $\ulim{x_n}$ such that $\ulim
    x_n\in A_n$.

    Notice that if $\lim_\omega
    \dist_n (e_n,A_n)=\infty $ then the set $\lio{A_n}$ is
empty.
\end{defn}

Any ultralimit of metric spaces is a complete metric space
\cite{DriesWilkie}. The same proof gives that $\seq A=\seqrep A$ is
always a closed subset of the ultralimit $\ulim(X_n, e)$.

\begin{defn}[asymptotic cone] Let $(X,\dist)$ be a metric
space, $\omega$ be an ultrafilter over a countable set $I$, $e=\ultra e$
be an observation point. Consider a sequence of numbers
$d=(d_n)_{n\in I}$ called {\em scaling constants} satisfying
$\lm_\omega d_n=\infty$.

The space $\ulim\left( X,\frac{1}{d_n}\dist, e \right)$ is called an
{\em asymptotic cone of $X$.} It is denoted by $\ko{X;e,d}$.
\end{defn}

\begin{remark}\label{grascone}
 Let $G$ be a finitely generated group endowed with a
word metric.

\noindent (1) The group $\Pi_1 G/\omega$ acts on $\ko{G;1,d}$
transitively by isometries:
$$\ultra g \lio{x_n}=\lio{g_nx_n}.$$
\noindent (2) Given an arbitrary sequence of observation points
$e$,
the group $e^\omega (\Pi_1 G/\omega) (e^\omega)\iv$ acts
transitively by isometries on the asymptotic cone $\ko{G;e,d}$. In
particular, every asymptotic cone of $G$ is homogeneous.
\end{remark}

\me

\begin{cvn}
By the above remark, when we consider
an asymptotic cone of a finitely generated group, it is no loss of
generality to assume that the observation
point is $(1)^\omega$. We shall do this unless explicitly
stated otherwise.
\end{cvn}

\subsection{The complex of curves}
\label{Section:CC}

Throughout, $S=S_{g,p}$ will denote a compact connected orientable
surface of genus $g$ and with $p$
boundary components. Subsurfaces $Y\subset S$ will always be
considered to be essential (i.e., with non-trivial fundamental group
which injects into the fundamental group of
$S$), also they will not be assumed
to be proper unless explicitly stated. We will often measure the
\emph{complexity of a surface} by
$\xi(S_{g,p})=3g+p-3$; this complexity is additive under disjoint
union. Surfaces and curves are always
considered up to homotopy unless explicitly stated otherwise; we
refer to a pair of curves (surfaces,
etc) intersecting if they have non-trivial intersection
independent of the choice of representatives.

The ($1$-skeleton of the) \emph{complex of curves} of a surface $S$,
denoted by $\CC(S)$, is defined as
follows. The set of vertices of $\CC(S)$, denoted by $\CC_0(S)$, is
the set of homotopy classes of
essential non-peripheral simple closed curves on $S$.
When $\xi(S)>1$, a collection of $n+1$ vertices span an $n$--simplex if the
corresponding curves can be realized
(by representatives of the homotopy classes)
disjointly on $S$. A simplicial complex is quasi-isometric to its
1-skeleton, so when it is convenient we abuse notation and use
the term  complex of curves to refer to its 1-skeleton.

A {\em multicurve} on $S$ are the homotopy classes of
curves associated to a simplex in $\CC(S)$.

If $\xi(S)=1$ then two vertices are connected by an edge if they can
be realized so that they intersect
in the {minimal possible number of points} on the surface $S$ (i.e.
$1$ if $S=S_{1,1}$ and $2$ if
$S=S_{0,4}$). If $\xi(S)=0$ then $S=S_{1,0}$ or $S=S_{0,3}$.  In the
first case we do the same as for
$\xi(S)=1$ and in the second case the curve complex is empty since this
surface doesn't support any essential
simple closed curve. The complex is also empty if $\xi(S)\le -2$.

Finally if $\xi(S)=-1$, then $S$ is an annulus. We only consider the case when the annulus is
a subsurface of a surface $S'$. In this case we define $\CC(S)$ by
looking in the annular cover
$\tilde{S'}$ in which $S$ lifts homeomorphically.
We use the
compactification of the hyperbolic plane
as the closed unit disk to obtain a closed annulus $\hat{S'}$.

We define the vertices of $\CC(S)$ to be
the homotopy classes of arcs connecting the two boundary components
of $\hat{S'}$, where the homotopies
are required to fix the endpoints. We define a pair of vertices to be
connected by an edge if they have
representatives which can be realized with disjoint interior.
Metrizing edges to be isometric to the unit interval, this
space is quasi-isometric to $\Z$.

A fundamental result on the curve complex is the following:

\begin{theorem}[Masur--Minsky; \cite{MasurMinsky:complex1}]
    For any surface $S$, the complex of curves of $S$ is an
    infinite-diameter $\delta$-hyperbolic space (as long as it is
    non-empty), with $\delta$ depending only on $\xi(S)$.
\end{theorem}

There is a particular family of geodesics in a complex of curves
$\CC(S)$ called \emph{tight geodesics}.  We use Bowditch's notion of
tight geodesic, as defined in \cite[$\S 1$]{Bowditch:tightgeod}.
Although $\CC (S)$ is a
locally infinite complex, some finiteness phenomena appear when
restricting to the family of tight geodesics
(\cite{Bowditch:tightgeod}, \cite{MasurMinsky:complex2}).

\subsection{Projection on the complex of curves of a subsurface}

We shall need the natural projection of a curve (multicurve) $\gamma$
onto an essential subsurface $Y\subseteq S$ defined in
\cite{MasurMinsky:complex2}.  By definition, the projection
$\pi_{\CC(Y)}(\gamma)$ is a possibly empty subset of $\CC(Y)$.

The definition of this projection is from
\cite[$\S$~2]{MasurMinsky:complex2} and is given below.
Roughly speaking, the projection consists of all closed curves of the
intersections of $\gamma$ with $Y$ together with all the arcs of
$\gamma\cap Y$ combined with parts of the boundary of $Y$ (to form
essential non-peripheral closed curves).

\begin{defn}\label{projcy} Fix an essential subsurface $Y\subset S$.
Given an element $\gamma\in\CC(S)$ we define the projection
$\pi_{\CC(Y)}(\gamma)\in 2^{\CC(Y)}$ as follows.
\begin{enumerate}
    \item If either $\gamma\cap Y=\emptyset$, or $Y$ is an annulus and
    $\gamma $ is its core curve (i.e. $\gamma$ is the unique homotopy
    class of essential simple closed curve in $Y$), then we define
    $\pi_{\CC(Y)}(\gamma)=\emptyset$.

    \item If $Y$ is an annulus
    which transversally intersects $\gamma$, then we define
    $\pi_{\CC(Y)}(\gamma)$ to be the union of the elements of $\CC(Y)$ defined by every
    lifting of $\gamma$ to an annular cover
    as in the definition of the curve complex of an annulus.

    \item In all remaining cases, consider the arcs and simple closed
    curves obtained by intersecting $Y$ and $\gamma$.  We define
    $\pi_{\CC(Y)}(\gamma)$ to be the set of vertices in $\CC(Y)$
    consisting of:
\begin{itemize}
    \item the essential non-peripheral
    simple closed curves in the
    collection above;
    \item the essential non-peripheral simple closed curves obtained by taking arcs from the above
    collection union a subarc of $\partial Y$ (i.e. those curves which
    can be obtained by resolving the arcs in $\gamma\cap Y$ to
    curves).
\end{itemize}
\end{enumerate}
\end{defn}

One of the basic properties is the following result of Masur--Minsky
(see \cite[Lemma~2.3]{MasurMinsky:complex2} for the original proof, in
\cite{Minsky:ELC1} the bound was corrected from 2 to 3).

\begin{lemma}[\cite{MasurMinsky:complex2}] \label{r1}
If $\Delta$ is a multicurve in $S$ and $Y$ is a subsurface of $S$ intersecting nontrivially every homotopy
class of curve composing $\Delta$, then
$$\diam_{\CC(Y)}(\pi_{\CC(Y)}(\Delta))\leq 3\, .$$
\end{lemma}

\begin{notation}\label{Ndistproj}
{Given $\Delta , \Delta'$ a pair of multicurves in  $\CC(S)$, for
brevity we
    often write $\dist_{\CC(Y)}(\Delta , \Delta')$ instead of
    $\dist_{\CC(Y)}(\pi_{Y}(\Delta),\pi_{Y}(\Delta'))$.}
\end{notation}

\subsection{Mapping class groups}\label{sec:mcg}

The \emph{mapping class group}, $\MCG (S)$, of a surface $S$ of finite
type is the quotient of the group of homeomorphisms of $S$ by the
subgroup of homeomorphisms isotopic to the identity.  Since the mapping class group of a
surface of finite type is finitely generated \cite{Birman:Braids}, we may consider a word
metric on the group --- this metric is unique up to bi-Lipschitz equivalence.
Note that $\MCG (S)$ acts on $\CC(S)$ by simplicial automorphisms (in
particular by isometries) and with finite quotient,
and that the
family of tight geodesics is invariant with respect to this action.

Recall that according to the Nielsen-Thurston classification, any
element $g\in \MCG (S)$ satisfies one
of the following three properties, where the first two are not mutually exclusive:
\begin{enumerate}
    \item $g$ has finite order;
    \item there exists a multicurve $\Delta$ in $\CC (S)$ invariant
by $g$ (in this case $g$ is called
    \textit{reducible});
    \item $g$ is pseudo-Anosov.
\end{enumerate}

We call an element $g\in \MCG(S)$ \emph{pure} if there exists a
multicurve $\Delta$ (possibly empty) component-wise invariant by $g$ and such that
$g$ does not permute the connected components of $S\setminus \Delta$,
and it induces on each component of $S\setminus \Delta$ and on each
annulus with core curve in $\Delta$ either a pseudo-Anosov or the
identity map (we use the convention that a Dehn twist on an annulus
is considered to be pseudo-Anosov).
In particular every pseudo-Anosov is pure.

\begin{theorem}(\cite[Corollary 1.8]{Ivanov:subgroups}, \cite[Theorem
7.1.E]{Ivanov:mcg})\label{thpure}
Consider the homomorphism from $\MCG (S)$ to the finite group $
\mathrm{Aut} (H_1 (S, \Z /k\Z ))$
defined by the action of diffeomorphisms on homology.

If $k\geq 3$ then the kernel $\MCG_k (S)$ of the homomorphism is
composed only of pure elements, in
particular it is torsion free.
\end{theorem}

\medskip

We now show that versions of some of the previous results hold for
the ultrapower $\Pi\MCG (S)/\omega$
of a mapping class group. The elements in $\MCG (S)^\omega$ can also be
classified into finite order,
reducible and pseudo-Anosov elements, according to whether their components satisfy that property \uas.

Similarly, one may define pure elements in $\MCG (S)^\omega$. Note that non-trivial pure
elements both in $\MCG (S)$ and in
its ultrapower are of infinite order.

Theorem \ref{thpure} implies the following statements.
\begin{lemma}\label{fs}
\begin{itemize}
    \item[(1)] The ultrapower $\MCG (S)^\omega$ contains a finite
index normal subgroup $\MCG (S)_p^\omega$
    which consists only of pure elements.
    \item[(2)] The orders of finite subgroups in the ultrapower $\MCG
(S)^\omega$ are boun\-ded by a constant
$N=N(S)$.
\end{itemize}
\end{lemma}

\proof (1) The homomorphism in Theorem \ref{thpure} for $k\geq 3$
induces a homomorphism from $\MCG
(S)^\omega$ to a finite group whose kernel consists only of pure
elements.

(2) Since any finite subgroup of  $\MCG (S)^\omega$ has trivial
intersection with the group of pure elements $\MCG (S)_p^\omega$,
it follows that it injects into the quotient group, hence its
cardinality is at most the index of $\MCG
(S)_p^\omega$.
\endproof

\subsection{The marking complex}

For most of the sequel, we do not work with the mapping class group
directly, but rather with a particular quasi-isometric model which is
a graph called
the \emph{marking complex}, $\MM(S)$,  which is defined as follows.

The vertices of the marking graph are called \emph{markings}. Each
marking $\mu\in\MM(S)$ consists of the following pair of data:
\begin{itemize}
    \item \emph{base curves}: a multicurve consisting of $3g+p-3$
    components, i.e. a maximal simplex in $\CC(S)$. This collection
    is denoted $\base(\mu)$.

    \item \emph{transversal curves}: to each curve
    $\gamma\in\base(\mu)$ is associated an essential curve in the
    complex of curves of the annulus with core curve $\gamma$ with a
    certain compatibility condition.
    More precisely, letting $T$ denote the complexity $1$
    component of
    $S\setminus \bigcup_{\alpha\in\base{\mu}, \alpha \neq
\gamma}\alpha$, the
    transversal curve to $\gamma$ is any curve $t(\gamma)\in\CC(T)$
with
    $\dist_{\CC(T)}(\gamma,t(\gamma))=1$; since
$t(\gamma)\cap\gamma\neq\emptyset$,
    the curve $t(\gamma)$ is a representative of a point in the curve
complex
    of the annulus about $\gamma$, i.e.
    $t(\gamma)\in\CC(\gamma)$.
\end{itemize}

We define two vertices $\mu,\nu\in\MM(S)$ to be connected by an edge
if either of the two conditions hold:

\begin{enumerate}
    \item \emph{Twists}: $\mu$ and $\nu$ differ by a Dehn twist along
    one of the base curves.  That is, $\base(\mu)=\base(\nu)$ and all
    their transversal curves agree except for about one element
    $\gamma\in \base(\mu)=\base(\nu)$ where $t_{\mu}(\gamma)$ is
    obtained from $t_{\nu}(\gamma)$ by twisting once
    about the curve $\gamma$.

    \item \emph{Flips}: The base curves and transversal curves of
    $\mu$ and $\nu$ agree except for one pair $(\gamma,
    t(\gamma))\in\mu$ for which the corresponding pair consists of the
    same pair but with the roles of base and transversal reversed.
    Note that the second condition to be a marking requires that each
    transversal curve intersects exactly one base curve, but the Flip
    move may violate this condition.  It is shown in
    \cite[Lemma 2.4]{MasurMinsky:complex2}, that there is a
    finite set of natural ways to resolve this issue, yielding a
    finite (in fact uniformly bounded) number of flip moves which can
    be obtained by flipping the pair $(\gamma, t(\gamma))\in\mu$; an
    edge connects each of these possible flips to $\mu$.
\end{enumerate}

The following result is due to Masur--Minsky
\cite{MasurMinsky:complex2}.

\begin{theorem}\label{MM:marking} The graph $\MM(S)$ is locally finite
and the mapping class group acts cocompactly and properly
discontinuously on it.  In particular the mapping class group of $S$
endowed with a word metric is quasi-isometric to $\MM(S)$ endowed with the
simplicial distance, denoted by $\dist_{\MM(S)}$.
\end{theorem}

\begin{notation}\label{Ndistm}
In what follows we sometimes denote $\dist_{\MM (S)}$ by $\dist_\MM$, when there is no
possibility of confusion.
\end{notation}

The subsurface projections introduced in Section~\ref{Section:CC}
allow one to consider     the
projection of a marking on $S$ to the curve complex of a subsurface
$Y\subseteq S$. Given a marking
$\mu\in\MM(S)$ we define $\pi_{\CC(S)}(\mu)$ to be $\base\mu$. More
generally, given a subsurface
$Y\subset S$, we define $\pi_{\CC(Y)}(\mu)=\pi_{\CC(Y)}(\base(\mu))$,
if $Y$ is not an annulus about an
element of $\base(\mu)$; if $Y$ is an annulus about an element
$\gamma\in\base(\mu)$, then we define
$\pi_{\CC(Y)}(\mu)=t({\gamma})$, the transversal curve to $\gamma$.

\begin{notation}
For two markings $\mu,\nu\in\MM(S)$ we often use the following
standard simplification of notation:
$$\dcy(\mu,\nu)=\dcy(\pi_{\CC(Y)}(\mu),\pi_{\CC(Y)}(\nu)).$$
\end{notation}

\begin{remark}\label{r2}
By Lemma \ref{r1}, for every marking $\mu$ and every subsurface
$Y\subseteq S$, the diameter of the
projection of $\mu$ into $\CC(Y)$ is at most $3$. This implies that
the difference between
$\dcy(\mu,\nu)$ as defined above and the Hausdorff distance between
$\pi_{\CC(Y)}(\mu)$ and
$\pi_{\CC(Y)}(\nu)$ in $\CC (Y)$ is at most six.
\end{remark}

\subsubsection*{Hierarchies}\label{shierar}

In the marking complex, there is an important
family of quasi-geodesics called \emph{hierarchy paths} which
have several useful geometric properties.
The concept of
hierarchy was first developed by Masur--Minsky in
\cite{MasurMinsky:complex2}, which the reader may consult for
further details. For a survey see \cite{Min06:ICMAddress}.
We recall below the properties of hierarchies
that we shall use in the sequel.

Given two subsets $A,B\subset \R$,
a map $f\co A\to B$ is said to be
\emph{coarsely increasing} if there exists a constant $D$ such that
for each $a,b$ in $A$ satisfying $a+D<b$, we have that $f(a)\leq f(b)$. Similarly, we define
\emph{coarsely decreasing} and \emph{coarsely monotonic} maps. We say a
map between quasi-geodesics is coarsely monotonic if it defines a
coarsely monotonic map between suitable nets in their domain.

We say a quasi-geodesic $\fg$ in $\MM(S)$ is
\emph{$\CC(U)$--monotonic} for some subsurface $U\subset S$ if one can
associate a geodesic $\ft_U$ in $\CC(U)$ which
{\em shadows} $\fg$ in the sense that $\ft_U$ is a path from
a vertex of
$\pi_{U}(\base(\mu))$ to a vertex of $\pi_{U}(\base(\nu))$
and there is a coarsely
monotonic map $v\co\fg\to\ft_U$ such that $v(\rho)$
is a vertex in $\pi_{U}(\base(\rho))$ for every vertex $\rho\in\fg$.

Any pair of points $\mu,\nu\in\MM(S)$ are connected by at least one
hierarchy path.  Hierarchy paths are quasi-geodesics with uniform
constants depending only on the surface $S$. One of the important
properties of hierarchy paths is that
they are $\CC(U)$--monotonic for every $U\subseteq S$ and moreover the
geodesic onto which they project is a tight geodesic.

The following is an immediate consequence of  Lemma 6.2 in \cite{MasurMinsky:complex2}.

\begin{lemma}\label{MM2:LLL}
There exists a constant $M=M(S)$ such that, if $Y$ is an essential
proper
subsurface of $S$
and $\mu,\nu$ are
two markings in $\MM (S)$ satisfying $\dist_{\CC(Y)}(\mu,\nu) > M$,
then any hierarchy path $\fg$ connecting $\mu$ to $\nu$ contains a
marking $\rho$ such that the multicurve $\base(\rho)$ includes the
multicurve $\boundary Y$.  Furthermore, there exists a vertex $v$ in
the geodesic $\ft_\fg$ shadowed by $\fg$ for which $v\in\base(\rho)$,
and hence satisfying $Y\subseteq S\setminus v$.
\end{lemma}

\begin{defn}\label{klarged}
Given a constant $K\ge M(S)$, where $M(S)$ is the constant from Lemma
\ref{MM2:LLL}, and a pair of
markings $\mu, \nu$, the subsurfaces $Y\subseteq S$ for which
$\dist_{\CC(Y)}(\mu, \nu)>K$ are called
the \emph{$K$-large domains} for the pair $(\mu, \nu)$.
We say that a
hierarchy path {\em contains a domain $Y\subseteq S$} if $Y$ is a
$M(S)$--large domain between some pair of points on the hierarchy path.
Note that for every such domain, the hierarchy contains
a marking whose base contains $\partial Y$.
\end{defn}

The following useful lemma is one of the basic ingredients in the
structure of hierarchy paths; it is an immediate consequence of
\cite[Theorem~4.7]{MasurMinsky:complex2} and the fact that a chain of
nested subsurfaces has length at most $\xi(S)$.

\begin{lemma}\label{atmost2xi} Let $\mu,\nu\in \MM(S)$ and let $\xi$
be the complexity of
$S$. Then for every $M(S)$-large domain $Y$ in a hierarchy path
$[\mu,\nu]$ there exist at most $2\xi$
domains that are $M(S)$-large and that contain $Y$, in that path.
\end{lemma}

A pair of subsurfaces $Y$,$Z$ are said to \emph{overlap} if $Y\cap
Z\neq\emptyset$ and neither of the two
subsurfaces is a subsurface of the other; when $Y$ and $Z$ overlap we
write $Y\pitchfork Z$, when they do not overlap
we write $Y\notpitchfork Z$. The following useful
theorem was proven by Behrstock in \cite{Behrstock:asymptotic}.

\begin{theorem}{\bf (Projection estimates \cite{Behrstock:asymptotic}).}\label{projest}
  There exists a constant $D$ depending only on the topological type
of a surface $S$
   such that for any two overlapping
     subsurfaces $Y$ and $Z$ in $S$,
     with  $\xi(Y)\neq 0 \neq \xi(Z)$, and for any $\mu\in\MM(S)$:
     $$\min\{ \dist_{\CC(Y)}(\partial Z, \mu), \dist_{\CC(Z)}(\partial Y,
     \mu)\}
     \leq D \, .$$
\end{theorem}

\begin{cvn}\label{mlargerd}
{In what follows we assume that the constant $M=M(S)$ from Lemma
\ref{MM2:LLL} is larger than the
constant $D$ from Theorem \ref{projest}.}
\end{cvn}

\begin{notation} Let $a>1$, $b,x,y$ be
positive real numbers. We write $ x\leq_{a,b} y$ if $$ x\le ay+b.$$

We write $x \approx_{a,b} y$ if and only if $ x\leq_{a,b} y$ and
$y\leq_{a,b} x$.
\end{notation}

\begin{notation} Let $K, N>0$ be real numbers. We define $\Tsh K{N}$
to be $N$ if $N\geq K$ and $0$ otherwise.
\end{notation}

The following result is fundamental in studying the metric geometry of
the marking complex.  It provides a way to compute distance in the
marking complex from the distances in the curve complexes of the large
domains.

\begin{theorem}[Masur--Minsky;
\cite{MasurMinsky:complex2}]\label{distanceformula}
    If $\mu,\nu\in \MM(S)$,  then there exists a constant $K(S)$,
    depending only on $S$, such that for each
    $K>K(S)$ there exists $a\geq 1$ and $b\geq 0$ for which:
    \begin{equation}\label{fdistformula}
     \dist_{\MM(S)}(\mu,\nu) \approx_{a,b} \sum_{Y\subseteq S}
    \Tsh K{\dist_{\CC(Y)}(\pi_{Y}(\mu),\pi_{Y}(\nu))}.
\end{equation}
\end{theorem}

We now define an important collection of subsets of the marking
complex.  In what follows by \emph{topological type of a multicurve}
$\Gamma$ we mean the topological type of the pair $(S, \Gamma)$.

\begin{notation}
Let $\Delta$ be a simplex in $\CC(S)$. We define
$\QQ(\Delta)$
to be the set of elements of $\MM(S)$ whose bases contain $\Delta$.
\end{notation}

\begin{remark}\label{rmk:qqproduct}
As noted in \cite{BehrstockMinsky:rankconj} a consequence of the
distance formula is that the space
$\QQ(\Delta)$ is quasi-isometric
to a coset of a stabilizer in $\MCG (S)$ of a multicurve with the same
topological type as $\Delta$. To see this, fix a collection
$\Gamma_1,...,\Gamma_n$ of
multicurves where each topological type
of multicurve is represented
exactly once in this list.  Given any multicurve $\Delta$, fix an
element $f\in\MCG$ for which $f(\Gamma_i)=\Delta$, for the appropriate
$1\leq i\leq n$.  Now, up to a bounded Hausdorff distance, we have an
identification of $\QQ(\Delta)$ with $f \stab(\Gamma_i)$ given
by the natural quasi-isometry between $\MM(S)$ and $\MCG(S)$.
\end{remark}

\subsection{Marking projections}

\subsubsection{Projection on the marking complex of a
subsurface.}\label{projmy}\quad Given any subsurface $Z\subset S$, we
define a projection $\pi_{\MM(Z)}\co\MM(S)\to 2^{\MM(Z)}$, which sends
elements of $\MM(S)$ to subsets of $\MM(Z)$.  Given any $\mu\in\MM(S)$
we build a marking on $Z$ in the following way.  Choose an element
$\gamma_{1}\in\pi_{Z}(\mu)$, and then recursively choose $\gamma_{n}$
from $\pi_{Z\setminus \cup_{i<n}\gamma_{i}}(\mu)$, for each $n\leq
\xi(Z)$.  Now take these $\gamma_{i}$ to be the base curves of a
marking on $Z$.  For each $\gamma_{i}$ we define its transversal
$t(\gamma_{i})$ to be an element of $\pi_{\gamma_{i}}(\mu)$.  This
process yields a marking, see \cite{Behrstock:asymptotic} for
details.

Arbitrary choices were made in this construction, but it is proven in
\cite{Behrstock:asymptotic} that there is a uniform constant depending
only on $\xi(S)$, so that given any $Z\subset S$ and any $\mu$ any two
choices in building $\pi_{\MM(Z)}(\mu)$ lead to elements of $\MM(Z)$
whose distance is bounded by this uniform constant. Thus, in the
sequel the
choices made in the construction will be irrelevant.

\begin{remark}\label{rnested}
Given two nested subsurfaces $Y\subset Z\subset S$ the projection of
an arbitrary marking $\mu$ onto $C(Y)$ is at uniformly bounded
distance from the projection of $\pi_{\MM (Z)}(\mu )$ onto $C(Y)$.
This follows from the fact that in the choice of $\pi_{\MM (Z)}(\mu)$
one can start with a base curve in $S$ which intersects $Y$ and hence
also determines up to diameter 3 the projection of $\mu$ to $C(Y)$.

A similar argument implies that $\pi_{\MM (Y)}(\mu )$ is at uniformly
bounded distance from $\pi_{\MM (Y)}\left( \pi_{\MM (Z)}(\mu )
\right)$.
\end{remark}

An easy consequence of the distance formula in Theorem
\ref{distanceformula} is the following.

\begin{corollary}\label{distsubsurf}
There exist $A\geq 1$ and $B\geq 0$ depending only on $S$ such that
for any subsurface $Z\subset S$ and
any two markings $\mu ,\nu \in \MM (S)$ the following holds:
$$
\dist_{\MM (Z)} \left( \pi_{\MM(Z)} (\mu ) \, ,\, \pi_{\MM(Z)} (\nu )
\right) \leq_{A,B} \dist_{\MM
(S)} (\mu , \nu )\, .
$$
\end{corollary}

\subsubsection{Projection on a set $\QQ
(\Delta)$.}\label{sect:proj}\quad Given a marking $\mu$ and a
multicurve $\Delta$, the projection $\pi_{\MM(S\setminus\Delta)}(\mu)$
can be defined as in Section \ref{projmy}.  This allows one to construct a
point $\mu'\in \QQ(\Delta)$ which, up to a uniformly bounded error,
is closest to $\mu$.  See
\cite{BehrstockMinsky:rankconj} for details.  The marking $\mu'$ is
obtained by taking
the union of the (possibly partial collection of) base curves $\Delta$
with transversal curves given by $\pi_{\Delta}(\mu)$ together with
the base curves and transversals given by
$\pi_{\MM(S\setminus\Delta)}(\mu)$.  Note that the construction of
$\mu'$ requires, for each subsurface $W$ determined by the multicurve
$\Delta$, the construction of a projection $\pi_{\MM(W)}(\mu )$.  As
explained in Section \ref{projmy} each $\pi_{\MM(W)}(\mu )$ is
determined up to uniformly bounded distance in $\MM (W)$, thus $\mu'$
is well defined up to a uniformly bounded ambiguity depending only on
the topological type of $S$.

\section{Tree-graded metric spaces}\label{stg}

\subsection{Preliminaries}\label{section:treegraded}

A subset $A$ in a geodesic metric
  space $X$ is called \emph{geodesic} if every two points in $A$
  can be joined by a geodesic contained in $A$.

\begin{defn}(Dru\c tu-Sapir \cite{DrutuSapir:TreeGraded})\label{deftgr}
Let $\free$ be a complete geodesic metric space and let $\pp$ be a
collection of closed geodesic proper subsets, called {\it{pieces}}, covering $\free$.
We say that the space $\free$ is {\em tree-graded with
respect to }$\pp$ if the following two properties are satisfied:
\begin{enumerate}

\item[($T_1$)] Every two different pieces have at most one point in
common.

\item[($T_2$)] Every simple non-trivial geodesic triangle in $\free$
is contained
in one piece.
\end{enumerate}

When there is no risk of confusion as to the set $\pp$, we simply
say that $\free$ is \emph{tree-graded}.
\end{defn}

Note that one can drop the requirement that pieces cover $X$ because
one can always add to the collection of pieces $\pp$ all the 1-element
subsets of $X$.  In some important cases, as for asymptotic cones of
metrically relatively hyperbolic spaces \cite{Drutu:RelHyp}, the
pieces of the natural tree-graded structure do not cover
$X$.

We discuss in what follows some of the properties of tree-graded
spaces that we shall need further on.

\begin{proposition}[\cite{DrutuSapir:TreeGraded}, Proposition
2.17]\label{pt2}
    Property ($T_2$) can be replaced by the following property:

\begin{quotation}
($T_2'$)\quad for every topological arc $\cf:[0,d]\to \free$ and
$t\in [0,d]$, let $\cf[t-a,t+b]$ be a maximal sub-arc of $\cf$
containing $\cf (t)$ and contained in one piece. Then every other
topological arc with the same endpoints as $\cf$ must contain the
points $\cf (t-a)$ and $\cf (t+b)$.
\end{quotation}
\end{proposition}

\begin{cvn}
In what follows when speaking about cut-points we always mean
global cut-points.
\end{cvn}

Any complete geodesic metric space with a global
cut-point provides an example of a tree-graded metric space, as the following result points out.

\begin{lemma}[\cite{DrutuSapir:TreeGraded}, Lemma 2.30]\label{cutting}
Let $X$ be a complete geodesic metric space containing at least
two points and let $\calc$ be a non-empty set of cut-points
in~$X$.

The set $\calp$ of all maximal path connected subsets that are
either singletons or such that none of their cut-points belongs to
$\calc$ is a set of pieces for a tree-graded structure on $X$.

Moreover the intersection of any two distinct pieces from $\calp$
is either empty or a point from $\calc$.
\end{lemma}

\begin{lemma}[\cite{DrutuSapir:TreeGraded}, Section 2.1]\label{ptx}
Let $x$ be an arbitrary point in $\free$ and let $T_x$ be the set
of points $y\in \free$ which can be joined to $x$ by a topological
arc intersecting every piece in at most one point.

The subset $T_x$ is a real tree and a closed subset of $\free$,
and every topological arc joining two points in $T_x$ is
    contained in $T_x$. Moreover, for every $y\in T_x$, $T_y=T_x$.
\end{lemma}

\begin{defn}\label{ttrees}
A subset $T_x$ as in Lemma \ref{ptx} is called a \emph{transversal
tree} in $\free$.

A geodesic segment, ray or line contained in a transversal tree is
called a \emph{transversal geodesic}.
\end{defn}

Throughout the rest of the section, $(\free, \pp)$ is a
tree-graded space.

The following statement is an immediate consequence of
\cite[Corollary 2.11]{DrutuSapir:TreeGraded}.

\begin{lemma}\label{ab}
Let $A$ and $B$ be two pieces in $\pp$. There exists a unique pair
of points $a\in A$ and $b\in B$ such that any topological arc
joining $A$ and $B$ contains $a$ and $b$. In particular $\dist
(A,B)=\dist (a,b)$.
\end{lemma}

For every tree-graded space there can be defined a canonical
$\R$--tree quotient
\cite{DrutuSapir:splitting}.

\Notat\quad Let  $x,y$ be two arbitrary points in $\free$. We define
$\widetilde\dist(x,y)$ to be
$\dist(x,y)$ minus the sum of lengths of non-trivial sub-arcs which
appear as intersections of one
(any) geodesic $[x,y]$ with pieces.

The function $\widetilde\dist(x,y)$ is well defined (independent of
the choice of a geodesic $[x,y]$),
symmetric, and it satisfies the triangle inequality \cite{DrutuSapir:splitting}.

The relation $\approx$ defined by
\begin{equation}\label{approx}
    x\approx y \hbox{ if and only if
} \widetilde\dist(x,y)=0\, ,
\end{equation} is
a closed equivalence relation.

\begin{lemma}(\cite{DrutuSapir:splitting})\label{quotientT}
\begin{enumerate}
    \item The quotient $T=\free/{\approx}$ is an $\R$--tree with
respect to the metric induced by
 $\widetilde\dist$.

    \item Every geodesic in $\free$ projects onto a geodesic in $T$.
Conversely, for every non-trivial geodesic $\g$ in $T$ there exists a
non-trivial geodesic $\pgot$
 in $\free$ such that its projection on $T$ is $\g$.
    \item If $x\ne y$ are in the same transversal tree of $\free$
then $\dist(x,y)={\widetilde\dist}(x,y)$. In
particular, $x\not\approx y$. Thus every transversal tree projects
into $T$ isometrically.
\end{enumerate}
\end{lemma}

Following \cite[Definitions 2.6 and 2.9]{DrutuSapir:splitting}, given
a topological arc $\g$ in $\free$, we define the {\textit{set of
cut-points on}} $\g$, which we denote by $\Cutp(\g)$, as the  subset of $\g$ which is complementary to  the union of all the interiors of
sub-arcs appearing as intersections of $\g$ with pieces. Given two
points $x,y$ in $\free$, we define the {\textit{set of cut-points
separating $x$ and $y$}}, which we denote by $\Cutp\{x,y\}$, as the
set of cut-points of some (any) topological arc joining $x$ and $y$.
Note that if $g$ is an isometry of $\free$ which permutes pieces in
a given tree-grading, then  $g(\Cutp\{x,y\})= \Cutp\{gx,gy\}$

\subsection{Isometries of tree-graded spaces}\label{section:treegradedisoms}

For all the results on tree-graded metric spaces that we use in what follows we refer to
\cite{DrutuSapir:TreeGraded}, mainly to Section 2 in that paper.

\begin{lemma}\label{midpiece}
Let $x,y$ be two distinct points, and assume that $\Cutp \{ x, y\}$ does not contain a
   point at equal distance from $x$ and $y$. Let $a$ be the farthest from $x$ point in $\Cutp \{ x, y\}$ with $\dist (x,a ) \leq \frac{\dist (x,y)}{2}$, and let $b$ be the farthest from $y$ point in $\Cutp \{ x, y\}$ with $\dist (y,b ) \leq \frac{\dist (x,y)}{2}$.  Then there exists a unique piece $P$ containing $\{a,b\}$, and $P$ contains all points at equal distance from $x$ and $y$.
\end{lemma}

\proof Since $\Cutp \{ x, y\}$ does not contain a
   point at equal distance from $x$ and $y$ it follows that $a\neq b$. The choice of $a,b$ implies that $\Cutp \{a,b\}=\{a,b\}$, whence $\{a,b\}$ is contained in a piece $P$, and $P$ is the unique piece with this property, by property $(T_1)$ of a tree-graded space.

Let $m\in \free$ be such that $\dist (x,m)=\dist (y,m)= \frac{\dist
(x,y)}{2}$.  Then any union of geodesics $[x,m]\cup [m,y]$ is a
geodesic.  By property $(T_2')$ of a tree-graded space, $a\in [x,m],
b\in [m,y]$.  The sub-geodesic $[a,m]\cup [m,b]$ has endpoints in the
piece $P$ and therefore is entirely
contained in $P$, since pieces are convex, i.e., each geodesic with endpoints in a piece is entirely contained in that piece
\cite[Lemma~2.6]{DrutuSapir:TreeGraded}.\endproof

\begin{defn}
Let $x,y$ be two distinct points. If $\Cutp \{ x, y\}$ contains a point at equal distance from $x$ and
$y$ then we call that point the \emph{middle cut-point} of $x,y$.

If $\Cutp \{ x, y\}$ does not contain such a point then we call the piece defined in Lemma \ref{midpiece}
 the \emph{middle cut-piece} of $x,y$.

If $x=y$ then we say that $x,y$ have the middle cut-point $x$.

Let $P,Q$ be two distinct pieces, and let $x\in P$ and $y\in Q$ be the
unique pair of points minimizing the distance Lemma~\ref{ab}.
The \emph{middle cut-point} (or \emph{middle cut-piece}) of $P,Q$ is
the middle cut-point (respectively, the middle cut-piece) of $x,y$.

If $P=Q$ then we say that $P,Q$ have the middle cut-piece $P$.
\end{defn}

\begin{lemma}\label{isompiece}
Let $g$ be an isometry permuting pieces of a tree-graded space $\free$, such that the cyclic group $\la g\ra$ has bounded orbits.
\begin{enumerate}
  \item\label{fixpoint} If $x$ is a point such that $gx\neq x$ then $g$ fixes the middle cut-point or the middle cut-piece of $x,gx$.
  \item\label{fixpiece} If $P$ is a piece such that $gP\neq P$ then $g$ fixes the middle cut-point or the middle cut-piece of $P,gP$.
\end{enumerate}

\end{lemma}

\proof (\ref{fixpoint})\quad  Let $e$ be the farthest from $gx$ point in $\Cutp \{ x,gx\}\cap \Cutp
\{ gx,g^2x \}$ and let $d=\dist (x,gx)>0$.

\medskip

\textbf{(a)} \quad Assume that $x,gx$ have a middle cut-point $m \in \Cutp \{ x,gx \}$. If the
intersection $\Cutp \{ x,gx\} \cap \Cutp \{ gx,g^2x \}$ contains $m$ then $gm=m$ because $g$ takes $m$ is the (unique) point from $\Cutp \{gx, g^2x\}$ at distance $\dist(x,gx)/2$ from $gx$. We argue by
contradiction and assume that $gm \neq m$, whence $\Cutp \{ x,gx \}\cap \Cutp \{ gx,g^2x\}$ does not
contain $m$. Then $\dist (gx , e ) = \frac{d}{2} - \epsilon$ for some $\epsilon >0$.

Assume that $\pgot=[x,e]\sqcup [e,g^2x]$ is a topological arc. If $e\in \Cutp \pgot $ then $e\in \Cutp
\{ x,g^2x \}$. It follows that $\dist (x,g^2 x )=\dist(x,e)+\dist(e,g^2x)= 2
(\frac{d}{2}+\epsilon)= d+2 \epsilon$. An induction argument will then give that $\Cutp \{x,g^nx\}$ contains $e,ge,..., g^{n-2}e$, hence that $\dist (x,g^n x )= d+2\epsilon(n-1)$. This contradicts the hypothesis that the orbits of $\la g\ra $ are bounded.

If $e\not\in \Cutp \pgot $ then there exists $a\in \Cutp \{x,e\}$ and $b\in \Cutp \{e,g^2x\}$ such that
$a,e$ are the endpoints of a non-trivial intersection of any geodesic $[x,gx]$ with a piece $P$, and $e,b$ are the
endpoints of a non-trivial intersection of any geodesic $[gx, g^2 x]$ with the same piece $P$. Note that $a\neq b$ otherwise the choice of $e$ would be contradicted. Also, since $m\in \Cutp \{x, e \}$ and $m\neq e$ it follows that $m \in \Cutp \{x, a \}$. Similarly, $gm \in \Cutp \{ g^2 x , b \}$. It follows that $\dist (x,a)$ and $\dist (g^2 x, b)$ are at least $\frac{d}{2}$. Since
$[x,a]\sqcup [a,b]\sqcup [b,g^2x]$ is a geodesic (by \cite[Lemma 2.28]{DrutuSapir:TreeGraded}) it follows that $\dist(x,g^2x)\geq \dist (x,a) + \dist (a,b) + \dist (b,g^2 x)\geq d+\dist (a,b)$. An
induction argument gives that $[x,a]\sqcup \bigsqcup_{k=0}^{n-2} \left([g^k a , g^k b]\sqcup [g^k b,
g^{k+1}a]\right) \sqcup [g^{n-1}a, g^n x]$ is a geodesic. This implies that $\dist(x,g^nx)\geq
d+(n-1) \dist (a,b)$, contradicting the hypothesis that $\la g\ra $ has bounded orbits. Note that the argument in this paragraph applies as soon as we find the points $x,e,p$ as above.

Assume that $\pgot=[x,e]\sqcup [e,g^2x]$ is not a topological arc. Then $[x,e]\cap [e,g^2x]$ contains a
point $y\neq e$. According to the choice of $e$, $y$ is either not in $\Cutp \{ x, e \}$ or not in $\Cutp \{ e,g^2x \}$, and $\Cutp \{ e,y \}$ must be $\{ e,y \}$. It follows that $y,e$ are
 in the same piece $P$. If we consider the endpoints of the (non-trivial) intersections of any geodesics $[x,e]$ and $[e,g^2x]$
with the piece $P$, points $a,e$ and $e,b$ respectively, then we are in the setup described in the previous case and we can apply the same argument.

\medskip

\textbf{(b)} \quad Assume that $x,gx$ have a middle cut-piece $Q$. Then there exist two points $i,o$ in
$\Cutp \{ x,gx \}$, the entrance and respectively the exit point of any geodesic $[x,gx]$ in the piece
$Q$, such that the midpoint of $[x,gx]$ is in the interior of sub-geodesic $[i,o]$  of $[x,gx]$ (for
any choice of the geodesic $[x,gx]$).

If $e\in \Cutp \{ x,i\}$ then $g$ stabilizes $\Cutp \{e, g e\}$ and, since  $g$ is an isometry, $g$
fixes the middle cut-piece $Q$ of $e,ge$. Assume on the contrary that $gQ \neq Q$. Then $e\in \Cutp \{
o, gx\}$. If $\pgot=[x,e]\sqcup [e,g^2x]$ is a topological arc and $e\in \Cutp \pgot $ then as in (a)
we may conclude that $\la g\ra $ has an unbounded orbit. If either $e\not\in \Cutp \pgot $ or
$\pgot=[x,e]\sqcup [e,g^2x]$ is not a topological arc then as in (a) we may conclude that there exist
$a\in \Cutp \{ x,e \}$ and $b\in \Cutp \{ e, g^2 x \}$ such that $a,b,e$ are pairwise distinct and
contained in the same piece $P$.

Assume that $e=o$. Then $a=i$ and $P=Q$. By hypothesis $P\neq gQ$. It follows that any geodesic $[b,g^2x]$ intersects $gQ$, whence $ga , ge \in \Cutp \{b,g^2x\}$. Since $[x,a] \sqcup [a,b ] \sqcup [b,g^2 x]$ is a geodesic (by \cite[Lemma 2.28]{DrutuSapir:TreeGraded}) we have that $\dist (x, g^2 x) = \dist (x,a) + \dist (a,b) + \dist (b, g^2 x) \geq \dist (x,a) + \dist (a,b) + \dist ( ga , g^2 x) = d + \dist (a,b)$. An inductive argument gives that for any $n\geq 1$, the union of geodesics $[x,a]\sqcup \bigsqcup_{k=0}^{n-2} \left([g^k a , g^k b]\sqcup [g^k b,
g^{k+1}a]\right) \sqcup [g^{n-1}a, g^n x]$ is a geodesic, hence $\dist(x,g^nx)\geq
d+(n-1) \dist (a,b)$, contradicting the hypothesis that $\la g\ra $ has bounded orbits.

Assume that $e\neq o$. If $e=gi$, hence $b=go$ and $P=gQ$ then an argument as before gives that for every $n\geq 1$,  $\dist(x,g^nx)\geq
d+(n-1) \dist (a,b)$, a contradiction.

If $e\not\in \{ gi, o \}$ then $i,o\in \Cutp \{x,a\}$ and $gi,go\in \Cutp \{b,g^2 x\}$, whence both $\dist (x,a)$ and $\dist (b, g^2 x)$ are larger than $\frac{d}{2}$. It follows that the union $[x,a]\sqcup \bigsqcup_{k=0}^{n-2} \left([g^k a , g^k b]\sqcup [g^k b,
g^{k+1}a]\right) \sqcup [g^{n-1}a, g^n x]$ is a geodesic, therefore $\dist(x,g^nx)$ is at least $
d+(n-1) \dist (a,b)$, contradiction.

\medskip

(\ref{fixpiece}) Let $x\in P$ and $y\in gP$ be the pair of points realizing the distance which exist by Lemma \ref{ab}. Assume that $gx\neq y$. Then for every geodesics $[x,y]$ and $[y,gx]$ it holds that
$[x,y]\sqcup [y,gx]$ is a geodesic. Similarly, $[x,y]\sqcup [y,gx]\sqcup [gx,gy]\sqcup [gy,g^2x]$  is a
geodesic. An easy induction argument gives that $\bigsqcup_{k=0}^{n-1} [g^k x,g^k y]$ is a geodesic. In particular $\dist (x, g^n x)\geq n \dist (y, gx)$ contradicting the hypothesis that $\la g\ra$ has bounded orbits.

It follows that $y=gx$. If $gx=x$ then we are done. If $gx \neq x$ then we apply~(\ref{fixpoint}).\endproof

\begin{lemma}\label{gpisompiece}
Let $g_1,..., g_n$ be isometries of a tree-graded space permuting pieces and generating a group with bounded orbits. Then $g_1,..., g_n$ have a common fixed point or (set-wise) fixed piece.
\end{lemma}

\proof We argue by induction on $n$. For $n=1$ it follows from Lemma \ref{isompiece}. Assume the conclusion holds for $n$ and take $g_1,..., g_{n+1}$ isometries generating a group with bounded orbits and permuting pieces. By the induction hypothesis  $g_1,..., g_n$ fix either a point $x$ or a piece $P$.

Assume they fix a point $x$, and that $g_{n+1}x \neq x$. Assume that $x, g_{n+1}x$ have a middle cut-point $m$. Then $g_{n+1}m=m$ by Lemma \ref{isompiece}, (\ref{fixpoint}). For every $i\in \{1,...,n\}$,
$g_{n+1} g_i x = g_{n+1}x$ therefore $g_{n+1} g_im=m$, hence $g_i m = m$. If $x, g_{n+1}x$ have a
middle cut-piece $Q$ then it is shown similarly that $g_1,..., g_{n+1}$ fix $Q$ set-wise. In the case
when $g_1,..., g_n$ fix a piece $P$, Lemma \ref{isompiece} also allows to prove that  $g_1,...,
g_{n+1}$ fix the middle cut-point or middle cut-piece of $P , g_{n+1} P$.\endproof


\section{Asymptotic cones of mapping class groups}\label{sacmcg}

\subsection{Distance formula in asymptotic cones of mapping class groups}\label{sec:conedistformula}
Fix an arbitrary asymptotic cone ${\mathcal{AM}}(S) =\mathrm{Con}^\omega
(\MM(S) ; (x_n), (d_n))$ of $\MM(S)$.

We fix a point $\nu_0$ in $\MM (S)$ and
define the map $\MCG (S) \to \MM (S)\, ,\, g \mapsto g\nu_0$, which
according to Theorem~\ref{MM:marking} is a quasi-isometry.
There exists a
sequence $g_0=(g_n^0)$ in $\MCG (S)$ such that $x_n = g_n^0 \nu_0$,
which we shall later use to discuss the ultrapower of the mapping
class group.

\medskip

\Notat \quad By Remark \ref{grascone}, the group $g_0^\omega (\Pi_1
\MCG (S)/\omega) (g_0^\omega)\iv$ acts transitively by isometries on
the asymptotic cone $\coneMx$.  We denote this group by $\GM$.

\me

\begin{defn}\label{def:hier}
A path in $\AM$ obtained by taking
an ultralimit of hierarchy paths, is, by a slight abuse of notation,
also called a \emph{hierarchy path}.
\end{defn}

\medskip

It was proved in \cite{Behrstock:asymptotic} that
the asymptotic cone $\AM$
has cut-points and is thus a tree-graded space.
Since $\AM$ is tree-graded, one
can define the collection of pieces
in the tree-graded structure of $\AM$ as the collection of maximal
subsets in $\AM$ without cut-points.
That set of pieces can be described as follows
\cite[$\S$~7]{BehrstockKleinerMinskyMosher}, where the equivalence to the
third item below is an implicit consequence of the proof in
\cite{BehrstockKleinerMinskyMosher} of the equivalence of the first
two items.

\begin{theorem}[Behrstock-Kleiner-Minsky-Mosher
    \cite{BehrstockKleinerMinskyMosher}]\label{classifypieces}
    Fix a pair of points $\seq \mu, \seq \nu\in\AM(S)$. If $\xi(S)\geq
    2$, then the following are equivalent.
    \begin{enumerate}
    \item No point of $\AM(S)$ separates $\seq \mu$ from $\seq \nu$.

    \item There exist points $\seq \mu',\seq \nu'$ arbitrarily close
      to $\seq\mu,\seq\nu$, resp., for which there exists representative
      sequences
      $(\mu'_{n}),(\nu'_{n})$ satisfying
      $$\ulim d_{\CC(S)}(\mu'_{n},\nu'_{n})<\infty.$$

    \item For every hierarchy path $\seq H =\lim_\omega h_{n}$
    connecting $\seq\mu$ and $\seq\nu$ there exists points $\seq
    \mu',\seq \nu'$ on $\seq H$ which are arbitrarily close to
    $\seq\mu,\seq\nu$, resp., and for which there exists
    representative sequences
    $(\mu'_{n}),(\nu'_{n})$ with $\mu'_{n},
    \nu'_{n}$ on $h_n$ satisfying $$\ulim
    d_{\CC(S)}(\mu'_{n},\nu'_{n})<\infty.$$
    \end{enumerate}
\end{theorem}

Thus for every two points $\mu , \nu$ in the same piece of $\AM$ there
exists a sequence of pairs $\seq\muk ,\seq\nuk$ such that:

\noindent$\bullet$ $\epsilon_k = \max \left( \dam \left( \seq\mu ,
\seq\muk \right)\, ,\, \dam \left( \seq\nu , \seq\nuk \right) \right)$
goes to zero;

\noindent$\bullet$ $D^{(k)}=\ulim \dcs(\muk_{n},\nuk_{n})$ is finite
for every $k\in \N$.

\begin{remark}\label{trees}
The projection of the marking complex $\MM (S)$ onto the complex of curves $\CC (S)$ induces a Lipschitz map from $\AM (S)$ onto the asymptotic cone ${\mathcal{AC}}(S) =\mathrm{Con}^\omega
(\CC(S) ; (\pcs x_n), (d_n))$.  By Theorem \ref{classifypieces},
pieces in $\AM(S)$ project onto singletons in ${\mathcal{AC}}(S)$.
Therefore two points $\sm $ and $\sn$ in $\AM(S)$ for which
$\widetilde\dist(\sm,\sn)=0$ (i.e., satisfying
the relation $\sm \approx \sn$; see Lemma~\ref{quotientT}) project onto the same
point in ${\mathcal{AC}}(S)$.  Thus the projection $\AM (S) \to
{\mathcal{AC}}(S)$ induces a projection of $T_S = \AM (S) / \approx$
onto ${\mathcal{AC}}(S)$.  The latter projection is not a bijection.
This can be seen by taking for instance a sequence $(\gamma_n )$ of
geodesics in $\CC (S)$ with one endpoint in $\pcs (x_n)$ and of length
$\sqrt {d_n}$, and considering elements $g_n$ in $\MCG (S)$ obtained
by performing $\lfloor\sqrt {d_n}\rfloor$. Dehn twists around each curve in
$\gamma_n$ consecutively.  The projections of the limit points
$\seqrep x$ and $\lio{ g_n x_n }$ onto ${\mathcal{AC}}(S)$ coincide, while their projections onto $T_S$ are distinct. Indeed the limit of the sequence of paths $h_n$ joining $x_n$ and $g_n x_n$ obtained by consecutive applications to $x_n$ of the Dehn twists is a transversal path. Otherwise, if this path had a non-trivial intersection with a piece then according to Theorem \ref{classifypieces}, (3), there would exist points $\mu_n$ and $\nu_n$ on $h_n$ such that $\dist_{\MM (S)} (\mu_n,\nu_n) \geq \epsilon d_n$ for some $\epsilon >0$ and such that $\dist_{\CC (S)} (\mu_n,\nu_n)$ is bounded by a constant $D>0$ uniform in $n$. But the latter inequality would imply that $\dist_{\MM (S)} (\mu_n,\nu_n) \leq D \sqrt{d_n}\, ,$ contradicting the former inequality.
\end{remark}

\medskip

Let $\ultrau$ be the set of all subsurfaces of $S$ and let  $\Pi \ultrau
/\omega$ be its ultrapower. For
simplicity we denote by $\mathbf{S}$ the element in $\Pi \ultrau/\omega $
given by the constant sequence
$(S)$. We define, for every $\bu=(U_n)^\omega \in \upss$, its
complexity $\xi (\bu )$ to be
$\lim_\omega \xi (U_n)$.

We say that an element ${\mathbf{U}}=(U_n)^\omega$ in $\Pi \ultrau/\omega
$ is a \emph{subsurface}
of another element $\mathbf{Y}=(Y_n)^\omega$, and we denote it by
$\bu \subseteq \mathbf{Y}$ if \uass
$U_n \subseteq Y_n$. An element ${\mathbf{U}}=(U_n)^\omega$ in
$\Pi \ultrau/\omega
$ is said to be a \emph{strict subsurface}
of another element $\mathbf{Y}=(Y_n)^\omega$, denoted
$\bu \subsetneq \mathbf{Y}$, if \uass
$U_n \subsetneq Y_n$; equivalently $\bu \subsetneq \mathbf{Y}$ if and
only if $\bu \subseteq \mathbf{Y}$ and $\xi (\by ) -
\xi (\bu ) \geq 1$.

For every ${\mathbf{U}}=(U_n)^\omega$ in $\Pi \ultrau/\omega $ consider
the ultralimit of the marking
complexes of $U_n$ with their own metric $\lu = \lim_\omega (\MM
(U_n), (1), (d_n))$. Since there
exists a surface $U'$ such that \uass $U_n$ is homeomorphic to $U'$,
the ultralimit $\lu$ is isometric
to the asymptotic cone $\mathrm{Con}^\omega (\MM (U'), (d_n))$.
Consequently $\lu$ is a tree-graded
metric space.

\begin{notation} We denote by $T_{\mathbf{U}}$ the quotient tree $\lu
/\approx$, as constructed in Lemma
\ref{quotientT}. We denote by $\dlu$ the metric on $\lu$. We abuse
notation slightly, by writing
$\tdlu$ to denote both the pseudo-metric on
$\lu$ defined at the end of Section \ref{stg}, and the metric
this induces on $T_\bu$.
\end{notation}

\begin{notation}
We denote by $\QQ(\partial \bu )$ the ultralimit
$\lio{\QQ(\partial U_n)}$ in the asymptotic cone $\AM$, taken with
respect to the basepoint obtained by projecting the base points we
use for $\AM$ projected to $\QQ(\partial \bu )$.
\end{notation}

There exists a natural projection map $\plu$ from $\AM$ to $\lu$
sending any element $\seq \mu =\seqrep
\mu$ to the element of $\lu$ defined by the sequence of projections
of $\mu_n$ onto $\MM (U_n)$. This
projection induces a well-defined projection between asymptotic cones
with the same rescaling constants
by Corollary \ref{distsubsurf}.

\begin{notation}
    For simplicity, we write $\dlu (\seq \mu , \seq \nu)$
    and $\tdlu (\seq \mu , \seq \nu)$ to
    denote the distance, and the pseudo-distance in $\lu$ between the
    images under the projection maps, $\plu
    (\seq \mu )$ and $\plu (\seq \nu)$.

    We denote by $\dist_{C(\bu )} (\seq \mu , \seq \nu)$ the ultralimit
    $\lim_\omega \frac{1}{d_n}
    \dist_{C(U_n)} (\mu_n , \nu_n )$.
\end{notation}

The following is from \cite[Theorem~6.5 and Remark~6.3]{Behrstock:asymptotic}.

\begin{lemma}[\cite{Behrstock:asymptotic}]\label{lem:transv}
Given a point $\seq \mu$ in $\AM$, the transversal tree $T_{\seq \mu}$ as defined in Definition
\ref{ttrees} contains the set
$$
\left\{ \seq \nu \mid \dlu (\seq \mu , \seq \nu)=0\, ,\, \forall \bu
\subsetneq \mathbf{S} \right\}.$$
\end{lemma}

\begin{corollary}\label{Utransv}
For any two distinct points $\seq\mu$ and $\seq\nu$ in $\AM$ there
exists at least one subsurface $\bu$ in $\upss$ such that $\dlu
(\seq\mu , \seq\nu) >0$ and for every strict subsurface
$\mathbf{Y}\subsetneq \bu $, $\dist_{\mathcal{M}^\omega (\mathbf{Y})}
(\seq\mu , \seq\nu)=0$.  In particular $\plu ( \seq\mu )$ and $\plu
(\seq\nu)$ are in the same transversal tree.
\end{corollary}

\proof Indeed, every chain of nested subsurfaces of $S$ contains at most $\xi$ (the complexity of $S$) elements. It implies that the same is true for chains of subsurfaces $\bu$ in $\upss$. In particular, $\upss$ with the inclusion order satisfies the descending chain condition. It remains to apply Lemma \ref{lem:transv}.\endproof

\begin{lemma}\label{cstd}
There exists a constant $t$ depending only on $\xi(S)$ such that
for every $\seq\mu$, $\seq\nu$ in $\AM$ and $\bu \in \upss$ the
following inequality holds
$$
\dist_{C(\bu)} (\seq\mu ,\seq\nu)\leq t\,  \tdlu (\seq\mu ,\seq\nu)\, .
$$
\end{lemma}

\proof The inequality involves only the projections of $\seq\mu$,
$\seq\nu$ onto $\MM^\omega (\bu )$.
Also \uass $U_n$ is homeomorphic to a fixed surface $U$, hence
$\MM^\omega (\bu )$ is isometric to some
asymptotic cone of $\MM (U)$, and it suffices to prove the inequality
for $\bu$ the constant sequence
$\mathbf{S}$.

Let $([\seq\alpha_k,\seq\beta_k])_{k\in K}$ be the set of non-trivial
intersections of a geodesic
$[\seq\mu , \seq\nu ]$ with pieces in $\AM$. Then $\widetilde{\dist}
(\seq\mu , \seq\nu) = \dist
(\seq\mu , \seq\nu) - \sum_{k\in K} \dist
(\seq\alpha_k,\seq\beta_k)$. For any $\epsilon >0$ there
exists a finite subset $J$ in $K$ such that $\sum_{k\in K\setminus J}
\dist (\seq\alpha_k,\seq\beta_k) \leq \epsilon $.  According to
Theorem~\ref{classifypieces} for every $k\in K$ there exist
$\seq\alpha_k' = \lio{ \alpha_{k,n}'}$ and $\seq\beta_k' =
\lio{\beta_{k,n}'}$ for which
$[\seq\alpha_k',\seq\beta_k']\subset [\seq\alpha_k,\seq\beta_k]$ and
such that:
\begin{enumerate}
    \item
$\ulimn\dist_{C(S)}(\alpha_{k,n}',\beta_{k,n}')<\infty$
    \item $\sum_{k\in K} \dist (\seq\alpha_k,\seq\beta_k )- 2\epsilon
    \leq\sum_{k\in J} \dist (\seq\alpha_k',\seq\beta_k').$
\end{enumerate}

The second item above follows since Theorem~\ref{classifypieces}
yields that $\seq\alpha_k',\seq\beta_k'$ can be chosen so that
$\sum_{k\in J} \dist (\seq\alpha_k,\seq\beta_k )$ is arbitrarily close
to $\sum_{k\in J} \dist (\seq\alpha_k',\seq\beta_k')$, and since the
contributions from those entries indexed by $K-J$ are less than
$\epsilon$.

Assume that $J=\{ 1,2,...m\}$ and that the points
$\seq\alpha_1',\seq\beta_1',
\seq\alpha_2',\seq\beta_2',...., \seq\alpha_m',\seq\beta_m'$ appear
on the geodesic $[\seq\mu,
\seq\nu]$ in that order.

By the triangle inequality
\begin{equation}\label{csn}
\dist_{C(S)} (\mu_n , \nu _n) \leq \dist_{C(S)} (\mu_n ,
\alpha_{1,n}') + \dist_{C(S)} (\beta_{m,n}' ,
\nu_n) +
\end{equation}
$$
\sum_{j=1}^m \dist_{C(S)} (\alpha_{j,n}',\beta_{j,n}') +
\sum_{j=1}^{m-1} \dist_{C(S)} (\beta_{j,n}',
\alpha_{j+1,n}')\, .
$$

Above we noted $\ulimn\sum_{j=1}^m
\dist_{C(S)}(\alpha_{j,n}',\beta_{j,n}')<\infty$, thus if we rescale
the above inequality by $\frac{1}{d_{n}}$ and take the ultralimit, we
obtain:
\begin{equation}\label{elast}
    \dist_{C(\mathbf{S})} (\seq\mu , \seq\nu )
    \leq
    \dist_{C(\mathbf{S})} (\seq\mu , \seq\alpha_1') +
    \dist_{C(\mathbf{S})}(\seq \beta_{m}' , \seq \nu )
    + \sum_{j=1}^{m-1}
    \dist_{C(\mathbf{S})} (\seq\beta_{j}', \seq\alpha_{j+1}').
\end{equation}

The distance formula implies that up to some multiplicative constant
$t$, the right hand side of equation~(\ref{elast}) is at most
$\dist_{\mathbf{S}} (\seq\mu , \seq\alpha_1') + \sum_{j=1}^{m-1}
\dist_{\mathbf{S}} (\seq\beta_{j}', \seq\alpha_{j+1}') +
\dist_{\mathbf{S}} (\seq \beta_{m}' , \seq \nu )$, which is equal to
$\dist_{\mathbf{S}} (\seq\mu , \seq\nu ) - \sum_{j=1}^m
\dist_{\mathbf{S}} (\seq\alpha_{j}' , \seq\beta_{j}')$.  Since above
we noted that $\sum_{k\in K} \dist (\seq\alpha_k,\seq\beta_k )-
2\epsilon \leq\sum_{j\in J} \dist (\seq\alpha_k',\seq\beta_k')$, it
follows that $$\dist_{\mathbf{S}} (\seq\mu , \seq\nu ) - \sum_{j=1}^m
\dist_{\mathbf{S}} (\seq\alpha_{j}' , \seq\beta_{j}') \leq
\widetilde{\dist}_{\mathbf{S}} (\seq\mu , \seq\nu ) + 2\epsilon.$$

Thus, we have shown that for every $\epsilon >0$ we have
$\dist_{C(\mathbf{S})} (\seq\mu , \seq\nu ) \leq t
\widetilde{\dist}_{\mathbf{S}} (\seq\mu , \seq\nu ) + 2t\epsilon $.
This completes the proof.
\endproof

\begin{lemma}\label{distone} Let $\mu$ and $\nu$ be two markings in
$\MM(S)$ at
$\CC(S)$-distance $s$ and let $\alpha_1,..., \alpha_{s+1}$ be the $s+1$ consecutive vertices (curves) of a tight geodesic in $\CC(S)$ shadowed by a hierarchy path
$\pgot$ joining $\mu$ and $\nu$.  Then the $s+1$ proper
subsurfaces $S_1, ..., S_{s+1}$ of $S$ defined by $S_{i}=S\setminus \alpha_{i}$ satisfy the inequality
\begin{equation}\label{dist1}
\dist_{\MM(S)}(\mu,\nu)\le
C\sum_i\dist_{\MM(S_i)}(\mu,\nu)+Cs+D\end{equation} for some constants
$C, D$
depending only on $S$.
\end{lemma}

\begin{proof} By Lemma~\ref{MM2:LLL},  if  $Y\subsetneq S$  is a proper subsurface
which yields a term in the distance formula (see
Theorem~\ref{distanceformula}) for
$\dist_{\MM(S)}(\mu,\nu)$, then there exists at least one (and at
most 3) $i\in\{1,...,s+1\}$
for which $Y\cap \alpha_i=\emptyset$. Hence, any such $Y$ occurs in
the distance formula for $\dist_{\MM(S_{i})}(\mu,\nu)$, where
$S_{i}=S\setminus \alpha_{i}$.
Every term which occurs in the distance formula for
$\dist_{\MM(S)}(\mu,\nu)$, except for the $\dist_{\CC(S)}(\mu,\nu)$
term, has a corresponding term (up to bounded multiplicative and
additive errors) in the distance formula for at least one of the
$\dist_{\MM(S_{i})}(\mu,\nu)$.
Since $\dist_{\CC(S)}(\mu,\nu)=s$, up to
the additive and multiplicative bounds occurring in the distance
formula this term in the distance formula for
$\dist_{\MM(S)}(\mu,\nu)$ is bounded above by $s$ up to a
bounded multiplicative and additive error.
This implies inequality~(\ref{dist1}).
\end{proof}

\begin{notation} For any subset $F\subset \upss$ we define the map
$\psi_F\co \AM \to \prod_{\bu \in F} T_\bu $, where for each
$\bu\in F$ the map $\psi_{\bu} \co \AM \to T_\bu$ is the
canonical projection of $\AM$ onto $T_\bu$.  In the
particular case when $F$ is finite equal to $\{ \bu_1,..., \bu_k \}$
we also use the notation $\psi_{\bu_1,.., \bu_k}$.
\end{notation}


\begin{lemma}\label{projtu}
Let $\fh\subset\AM$ denote the ultralimit of a sequence of
uniform quasi-geodesics in
$\MM$. Moreover, assume that the quasi-geodesics in the sequence are
$\CC(U)$--monotonic for every $U\subseteq S$, with constants that are
uniform over the sequence.

Any path $\fh\co [0,a] \to \AM$, as above, projects onto a geodesic
$\g\co [0,b] \to T_\bu$ such that $\fh (0)$ projects onto $\g(0)$ and,
assuming both $\fh$ and $\g$ are  parameterized by arc length,
the map $[0,a]\to [0,b]$ defined by the projection is
non-decreasing.
\end{lemma}

\proof \quad Fix a path $\fh$ satisfying the hypothesis of the lemma.
It suffices to prove that for every $\seq x ,\seq y $ on
$\fh$ and every $\seq\mu$ on $\fh$ between $\seq x$
and $\seq y$, $\psi_\bu (\seq\mu)$ is on the geodesic joining
$\psi_\bu (\seq x)$ to $\psi_\bu (\seq y)$ in $T_\bu$.  If the contrary were to hold then there would exist $\seq\nu
,\seq\rho $ on $\fh$ with $\seq\mu$ between them satisfying $\psi_\bu
(\seq\nu )=\psi_\bu (\seq\rho )\neq \psi_\bu (\seq\mu)$.  Without loss of
generality we may assume that $\seq\nu ,\seq\rho$ are the endpoints of
$\fh$.  We denote by $\fh_1$ and $\fh_2$ the sub-arcs of $\fh$ of
endpoints $\seq\nu ,\seq\mu$ and respectively $\seq\mu ,\seq\rho$.

The projection $\pi_{\MM (\bu )}(\fh )$ is by Corollary \ref{distsubsurf} a continuous path joining
$\pi_{\MM (\bu )}(\seq\nu )$ to $\pi_{\MM (\bu )}(\seq\rho )$ and containing $\pi_{\MM (\bu
)}(\seq\mu)$.

According to \cite[Lemma 2.19]{DrutuSapir:splitting} a geodesic
$\overline{\g}_1$ joining $\pi_{\MM (\bu )}(\seq\nu )$ to $\pi_{\MM
(\bu )}(\seq\mu )$ projects onto the geodesic $[\psi_\bu (\seq\nu
),\psi_\bu (\seq\mu)]$ in $T_\bu$.  Moreover the set $\Cutp
\overline{\g}_1$ of cut-points of $\overline{\g}_1$ in the tree-graded
space $\MM (\bu )$ projects onto $[\psi_\bu (\seq\nu
),\psi_\bu (\seq\mu)]$.  By properties of tree-graded spaces
\cite{DrutuSapir:TreeGraded} the continuous path $\pi_{\MM (\bu
)}(\fh_1 )$ contains $\Cutp \overline{\g}_1$.

Likewise if $\overline{\g}_2$ is a geodesic joining $\pi_{\MM (\bu
)}(\seq\mu )$ to $\pi_{\MM (\bu )}(\seq\rho )$ then the set $\Cutp
\overline{\g}_2$ projects onto $[\psi_\bu (\seq\mu ),\psi_\bu
(\seq\rho )]$, which is the same as the geodesic $[\psi_\bu (\seq\mu
),\psi_\bu (\seq\rho )]$ reversed, and the path $\pi_{\MM (\bu
)}(\fh_2 )$ contains $\Cutp \overline{\g}_2$.  This implies that
$\Cutp \overline{\g}_1= \Cutp \overline{\g}_2$ and that there exists
$\seq\nu'$ on $\fh_1$ and $\seq\rho'$ on $\fh_2$ such that $\pi_{\MM
(\bu )}(\seq\nu')=\pi_{\MM (\bu )}(\seq\rho')$ and $\psi_\bu
(\seq\nu')=\psi_\bu (\seq\rho')\neq \psi_\bu (\seq\mu)$.  Without loss
of generality we assume that $\seq\nu'=\seq\nu$ and
$\seq\rho'=\seq\rho$.

Since $\tdlu (\seq\nu,\seq\mu)>0$ it follows that $\dist_{\MM (\bu )}
(\seq\nu,\seq\mu)>0$.
Since by construction $\mu_n$ is on a path joining $\nu_n$ and
$\rho_n$ satisfying the hypotheses of the lemma, we know that
up to a uniformly bounded additive error we have
$\dist_{\CC(Y)} (\nu_n,\mu_n) \leq \dist_{\CC(Y)} (\nu_n,\rho_n)$
for every $Y\subset U_{n}$. It then follows from the distance formula
that for some positive constant $C$ that
$$\frac{1}{C}\dist_{\MM(\bu)} (\seq\nu,\seq\mu) \leq
\dist_{\MM(\bu)} (\seq\nu,\seq\rho).$$
In particular,
this implies that $\dist_{\MM (\bu )} (\seq\nu,\seq\rho )>0$,
contradicting the fact that $\pi_{\MM
(\bu )}(\seq\nu)=\pi_{\MM (\bu )}(\seq\rho)$.\endproof

The following is an immediate consequence of Lemma~\ref{projtu} since
by construction hierarchy paths satisfy the hypothesis of the lemma.
\begin{corollary}
    Every hierarchy path in $\AM$ projects onto a geodesic in $T_\bu$ for
    every subsurface $\bu$ as in Lemma \ref{projtu}.
\end{corollary}


\begin{notation}
    Let $F,G$ be two finite subsets in the asymptotic cone
    $\AM$, and let $K$ be a fixed
    constant larger than the constant $M(S)$ from Lemma~\ref{MM2:LLL}.

    We denote by $\yy (F, G)$ the set of elements ${\bu}=(U_n)^\omega$ in
    the ultrapower $\upss$ such that
    for any two points $\seq\mu =\lio{\mu_n} \in F$ and
    $\seq\nu=\lio{\nu_n} \in G$, the subsurfaces $U_n$
    are \uass $K$-large domains for the pair $(\mu_n, \nu_n)$, in the
    sense of Definition \ref{klarged}.

    If $F=\{\seq\mu \}$ and $G=\{ \seq\nu \}$ then we simplify the
    notation to $\yy (\seq\mu, \seq\nu)$.
\end{notation}

\begin{lemma}\label{restr}
Let $\seq\mu, \seq\nu$ be two points in $\AM$ and let $\bu
=(U_n)^\omega $ be an element in $\upss$.  If $\tdlu (\seq\mu, \seq\nu
) >0$, then $\lio{\dist_{C(U_n)} (\mu_n ,\nu_n )}=\infty $ (and thus
$\bu \in \yy (\seq\mu, \seq\nu)$).  In particular the following holds.
$$
\sum_{\bu \in \yy(\seq\mu, \seq\nu)}\tdlu (\seq\mu, \seq\nu ) = \sum_{\bu \in \upss }\tdlu (\seq\mu,
\seq\nu )
$$
\end{lemma}

\proof We establish the result by proving the contrapositive; thus we
assume that $\lio{\dist_{C(U_n)} (\mu_n ,\nu_n )}<\infty $.  Theorem
\ref{classifypieces} then implies that $\plu (\seq \mu )$ and $\plu
(\seq\nu)$ are in the same piece, hence $\tdlu (\seq\mu, \seq\nu)=0$.
\endproof

\medskip

We are now ready to prove a distance formula in the asymptotic cones.

\begin{theorem}[distance formula for asymptotic cones]\label{tilde}
    There is a constant $E$, depending only on the
    constant $K$ used to define $\yy
    (\seq\mu,\seq\nu)$, and on the complexity $\xi (S)$, such that
    for every $\seq\mu, \seq\nu$ in $\AM$
\begin{equation}\label{eqdist}
\frac{1}{E}\dist_\AM (\seq\mu,\seq\nu)\! \leq\! \sum_{\bu\in\yy (\seq\mu,
\seq\nu)}\! \tdlu(\seq\mu,\seq\nu) \! \leq \! \sum_{\bu\in\yy (\seq\mu,
\seq\nu)} \!\dlu(\seq\mu,\seq\nu)\! \leq \! E \dist_\AM (\seq\mu,\seq\nu).
\end{equation}

\end{theorem}

\begin{proof} Let us prove by induction on the complexity of $S$ that
\begin{equation}\label{eq88} \sum_{\bu\in \yy (\seq\mu, \seq\nu)}
    \tdlu(\seq\mu,\seq\nu)>
    \frac1E\dist_\AM(\seq\mu,\seq\nu)\end{equation} for some $E>1$.
    Let $\seq\mu, \seq\nu$ be two distinct elements in $\AM$.  If
    $\MM(S)$ is hyperbolic, then $\AM$ is a tree; hence there are no
    non-trivial subsets without cut-points and thus in this case we
    have $\widetilde\dist_{\mathbf{S}}=\dist_{\mathbf{S}}$.  This
    gives the base for the induction.

We may assume that
$\tdls(\seq\mu,\seq\nu)<\frac13\dist_\AM(\seq\mu,\seq\nu)$.  Otherwise
we would have that $\tdls(\seq\mu,\seq\nu) >0$, which implies by Lemma
\ref{restr} that $\mathbf{S}\in \yy (\seq\mu, \seq\nu )$, and we would
be done by choosing $E=3$.

Since $\tdls(\seq\mu,\seq\nu)$ is obtained from
$\dist_\AM(\seq\mu,\seq\nu)$ by removing $\sum_{i\in I } \dist_\AM
(\seq\alpha_i, \seq\beta_i)$, where $[\seq\alpha_i, \seq\beta_i], i\in
I ,$ are all the non-trivial intersections of a geodesic $[\seq\mu ,
\seq\nu]$ with pieces, it follows that there exists $F\subset I$
finite such that $\sum_{i\in F } \dist_\AM (\seq\alpha_i,
\seq\beta_i)\geq \frac12 \dist_\AM(\seq\mu,\seq\nu)$.  For simplicity
assume that $F= \{ 1,2,...,m\}$ and that the intersections
$[\seq\alpha_i, \seq\beta_i]$ appear on $[\seq\mu , \seq\nu]$ in the
increasing order of their index.  According to Proposition~\ref{pt2},
$(T_2')$, the points $\seq\alpha_i, \seq\beta_i$ also appear on any
path joining $\seq\mu , \seq\nu$.  Therefore, without loss
of generality, for the rest of the proof we will assume that $[\seq\mu , \seq\nu]$
is a hierarchy path, and $[\seq\alpha_i, \seq\beta_i]$ are sub-paths
of it (this is a slight abuse of notation since hierarchy paths are
not geodesics).
By Theorem~\ref{classifypieces} for every $i\in F$ there exist $[\seq\alpha_i',
\seq\beta_i']\subset [\seq\alpha_i, \seq\beta_i] $
with the following properties:

\begin{itemize}
    \item there exists a number $s$ such that $\forall i=1,...,m$
    $$ \dist_{C(S)}(\alpha_{i,n}',\beta_{i,n}')<s\; \; \mbox{\uas}
     \, ;$$
    \item $\sum_{i=1}^m \dist_\AM (\seq\alpha_i',\seq\beta_i')>
\frac13\dist_\AM(\seq\mu, \seq\nu)$.
\end{itemize}

Let $l=m(s+1)$.  By Lemma \ref{distone} there exists a sequence of proper
subsurfaces $Y_1(n),...,Y_l(n)$ of the form $Y_j(n) = S \setminus v_j(n)$ with $v_j(n)$ a vertex (curve) on the tight geodesic in $\CC (S)$ shadowed by the hierarchy path $[\mu_n, \nu_n]$,
such that \uas:
\begin{equation}\label{intermed}
\sum_{i=1 }^m \dist_{\MM (S)} (\alpha_{i,n}', \beta_{i,n}')\leq C
\sum_{i=1}^m \sum_{j=1}^l \dist_{\MM
(Y_j(n))}(\alpha_{i,n}', \beta_{i,n}') + C sm +Dm\, .
\end{equation}

Let $\seq Y_{j}$ be the element in $\upss$ given by the sequence of subsurfaces $(Y_{j}(n))$.

Rescaling (\ref{intermed}) by $\frac{1}{d_n}$, passing to the
$\omega$-limit and applying Lemma \ref{D} we deduce that
 \begin{equation}\label{intermed2}
\frac13\dist_\AM(\seq\mu, \seq\nu)\leq C \sum_{i=1 }^m \sum_{j=1}^l
\dist_{\subseq Y_{j}}(\seq\alpha_{i}', \seq\beta_{i}')\, .
\end{equation}

As the complexity of $\seq Y_j$ is smaller than the complexity
 of $S$, according to the induction
 hypothesis the second term in (\ref{intermed2}) is at most
 $$
C E \sum_{i=1}^m \sum_{j=1}^l \sum_{\bu \in \mathcal{Y}
(\subseq\alpha_{i}', \subseq\beta_{i}' ), \bu
\subseteq \subseq Y_{j}} \tdlu (\seq\alpha_{i}', \seq\beta_{i}')\, .
 $$

Lemma \ref{projtu} implies that the non-zero terms in
the latter sum correspond to subsurfaces $\bu \in \mathcal{Y} (\seq\mu
, \seq\nu )$, and that the sum is at most
$$
C E \sum_{j=1}^l \sum_{\bu \in \mathcal{Y} (\seq\mu
, \seq\nu ), \bu
\subseteq \subseq Y_{j}} \tdlu (\seq\mu
, \seq\nu ) \, .
$$

According to Lemmas \ref{restr} and~\ref{MM2:LLL} for every $\bu = (U_n)^\omega \in \mathcal{Y} (\seq\mu
, \seq\nu )$ there exists at least one and at
most 3 vertices (curves) on the tight geodesic in $\CC (S)$
shadowed by the hierarchy path $[\mu_n, \nu_n]$ which
are disjoint from $U_{n}$ \uas.
In particular for every $\bu \in \mathcal{Y} (\seq\mu
, \seq\nu )$ there exist at most three $j\in \{ 1,2,...,l\}$ such that $\bu \subseteq Y_{j}\, $.

Therefore the previous sum is at most $3C E \sum_{\bu\in \yy
(\seq\mu, \seq\nu)} \tdlu(\seq\mu,\seq\nu)$.  We have thus obtained
that $\frac13\dist_\AM(\seq\mu, \seq\nu) \leq 3C E \sum_{\bu\in \yy
(\seq\mu, \seq\nu)} \tdlu(\seq\mu,\seq\nu)$.

The inequality
$$
\sum_{\bu\in\yy (\seq\mu, \seq\nu)} \tdlu(\seq\mu,\seq\nu) \leq
\sum_{\bu\in\yy (\seq\mu, \seq\nu)}
\dlu(\seq\mu,\seq\nu)
$$
immediately follows from the definition of $\tdist$.

In remains to prove the inequality:

\begin{equation}\label{eq11}
\sum_{\bu\in\yy (\seq\mu, \seq\nu)} \dlu(\seq\mu,\seq\nu)<
E\dist(\seq\mu,\seq\nu).
\end{equation}

It suffices to prove (\ref{eq11}) for every possible finite sub-sum of
the left hand side of (\ref{eq11}).  Note that this would imply also
that the set of $\bu\in\yy (\seq\mu, \seq\nu)$ with
$\dlu(\seq\mu,\seq\nu) >0$ is countable, since it implies that the set
of $\bu\in\yy (\seq\mu, \seq\nu)$ with $\dlu(\seq\mu,\seq\nu)
>\frac{1}{k}$ has cardinality at most $k E\dist(\seq\mu,\seq\nu)$.

Let $\bu_1,...,\bu_m$ be elements in $\yy (\seq\mu, \seq\nu)$
represented by sequences $(U_{i,n})$ of large domains of hierarchy
paths connecting $\mu_n$ and $\nu_n$, $i=1,...,m$.

By definition, the sum
\begin{equation}\label{eq14}
\dist_{\bu_1}(\seq\mu,\seq\nu)+... +\dist_{\bu_m}(\seq\mu,\seq\nu)
\end{equation}
is equal to
\begin{equation}\label{eq12}
\lio{\frac{\dist_{\MM (U_{1,n})}(\mu_n,\nu_n)}{d_n}} +...+\lio{
\frac{\dist_{\MM (U_{m,n})}(\mu_n,\nu_n)}{d_n}}
\end{equation}
$$
= \lim_\omega \frac{1}{d_n} \left[\dist_{\MM
(U_{1,n})}(\mu_n,\nu_n)+...+\dist_{\MM
(U_{m,n})}(\mu_n,\nu_n)\right]\, .
$$

According to the distance formula (Theorem \ref{distanceformula}), there exist constants $a,b$ depending
only on $\xi(S)$ so that the following holds:

\begin{equation}\label{eq15}
\dist_{\MM (U_{1,n})}(\mu_n,\nu_n)+...+\dist_{\MM
(U_{m,n})}(\mu_n,\nu_n)\le_{a,b}
\end{equation}
$$
\sum_{V\subseteq U_{1,n}}
\dist_{C(V)}(\mu_n,\nu_n)+...+\sum_{V\subseteq
U_{m,n}}\dist_{C(V)}(\mu_n,\nu_n).
$$

Since each $U_{i,n}$ is a large domain in the hierarchy connecting
$\mu_n$ and  $\nu_n$, and since for each fixed $n$ all of the
$U_{i,n}$ are different $\omega$-a.s.\ we can apply Lemma
\ref{atmost2xi}, and conclude that each summand
occurs in the right hand side of (\ref{eq15}) at most $2\xi$ times
(where $\xi$ is denoting $\xi(S)$). Hence we can bound the right
hand side of (\ref{eq15}) from above by
$$
2\xi \sum_{V\subseteq S}\dist_{C(V)}(\mu_n,\nu_n)\le_{a,b} 2\xi\,
\dist_{\MM(S)}(\mu_n,\nu_n).
$$
Therefore the right hand side in (\ref{eq12}) does not exceed
$$2\xi\, \lio{\frac{1}{d_n}(a\dist_{\MM(S)}(\mu_n,\nu_n)+b)}= 2a \xi\,
\dist_{\AM}(\seq\mu,\seq\nu)\, ,
$$
proving (\ref{eq11}).
\end{proof}

\medskip

\Notat \quad Let $\seq\mu^0$ be a fixed point in $\AM$ and for every
$\bu \in \upss$ let $\seq\mu^0_\bu$ be the image of $\seq\mu^0$ by
canonical projection on $T_\bu$.  In $\prod_{\bu \in \upss } T_\bu$ we
consider the subset $\calt_0'=\left\{ (x_\bu)\in \prod_{\bu \in \upss
} T_\bu \; ;\; x_\bu \neq \seq\mu^0_\bu \mbox{ for countably many }
\bu \in \upss \right\}$, and $\calt_0=\left\{ (x_\bu)\in \calt'_0 \;
;\; \sum_{\bu \in \upss }\tdlu \left(x_\bu, \seq\mu^0_\bu
\right)<\infty \right\}$. We will always consider $\calt_0$
endowed with the $\ell^1$ metric.

\medskip

The following is an immediate consequence of
Theorem~\ref{tilde} and Lemma~\ref{restr}.

\begin{corollary}\label{cor:prodtrees}
Consider the map $\psi\co\AM \to \prod_{\bu \in \upss } T_\bu $
whose components
are the canonical projections of $\AM$ onto $T_\bu $. This map is a
bi-Lipschitz homeomorphism onto its image in $\calt_0$.
\end{corollary}

\begin{proposition}\label{hiergeod}
    Let $\fh\subset\AM$ denote the ultralimit of a sequence of
    quasi-geodesics in
    $\MM$ each of which is $\CC(U)$--monotonic for every $U\subseteq
    S$ with
    the quasi-geodesics and monotonicity constants are all
    uniform over the sequence.
    Then $\psi(\fh)$ is a geodesic in $\calt_{0}$.

    In particular, for any hierarchy path $\fh\subset \AM$,
    its image under $\psi$ is a geodesic in $\calt_0$.
\end{proposition}

The first statement of this proposition is a direct consequence of the
following lemma, which is an easy exercise in elementary topology.
The second statement is a consequence of the first.

\begin{lemma}\label{l1geod}
Let $(X_i, \dist_i)_{i\in I}$ be a collection of metric spaces.  Fix
a point
$x=(x_i)\in\prod_{i\in I} X_i$, and consider the
subsets $$S_0'=\left\{ (y_i)\in \prod_{i\in I} X_i : y_i \neq x_i
\mbox{ for countably many } i\in I \right\}$$ and $S_0=\left\{
(y_i)\in S'_0 : \sum_{i\in I}\dist_i \left(y_i,x_i \right)<\infty
\right\}$ endowed with the $\ell^1$ distance $\dist =\sum_{i\in I}
\dist_i$.

Let $\fh \co [0,a] \to S_0$ be a non-degenerate parameterizations of a
topological arc. For each
$i\in I$ assume that $\fh$ projects onto a geodesic $\fh_{i}\co
[0,a_{i}]\to X_{i}$ such that $\fh(0)$ projects onto $\fh_{i}(0)$ and
the function
$\varphi_i=d_{i}(\fh_{i}(0),\fh_{i}(t))\co[0,a] \to [0,a_{i}]$ is a
non-decreasing function.
Then $\fh [0,a]$ is a geodesic in $(S_0, \dist)$.
\end{lemma}

\me

\Notat \quad We write $\td$ to denote the $\ell^1$ metric on $\calt_0$.
We abuse notation slightly by also using $\td$ to denote both
its restriction to $\psi (\AM)$ and for the
metric on $\AM$ which is the pull-back via $\psi$ of $\td$.  We have
that $\td (\seq\mu , \seq\nu )= \sum_{\bu \in \upss } \tdlu (\seq\mu ,
\seq\nu )$ for every $\seq\mu , \seq\nu \in \AM$, and that $\td$ is
bi-Lipschitz equivalent to $\dist_\AM$, according to
Theorem~\ref{tilde}.

\me

Note that the canonical map $\AM \to \prod_{\bu \in \upss } \CC (\bu
)$, whose components are the
canonical projections of $\AM$ onto $\CC (\bu )$, ultralimit of
complexes of curves, factors through
the above bi-Lipschitz embedding. These maps were studied in
\cite{Behrstock:thesis}, where among other things it was shown that
this canonical map is not a bi-Lipschitz embedding (see also Remark \ref{trees}).

\subsection{Dimension of asymptotic cones of mapping class groups}\label{sec:dimcone}

\begin{lemma}\label{point}
Let $\bu$ and $\bv$ be two elements in $\upss$ such that either $\bu
,\bv$ overlap or $\bu \subsetneq \bv$.  Then $\QQ(\partial \bu )$
projects onto $T_\bv$ in a unique point.
\end{lemma}

\proof Indeed, $\QQ(\partial \bu )$ projects into $\QQ(\plv (\partial
\bu))$, which is contained in one piece of $\lv$, hence it projects
onto one point in $T_\bv$.  \endproof

The following gives an asymptotic analogue of
\cite[Theorem~4.4]{Behrstock:asymptotic}.

\begin{theorem}\label{projection trichotomy}
Consider a pair $\bu , \bv$ in $\upss$.
\begin{enumerate}
    \item If $\bu \cap \bv =\emptyset$ then the image of $\psi_{\bu ,
    \bv }$ is $T_\bu \times T_\bv$.

    \medskip

    \item If $\bu$ and $\bv$ overlap then the image of $\psi_{\bu ,
    \bv }$ is $$ \left( T_\bu \times \{u \} \right) \cup \left(
    \{v\}\times T_\bv \right) \, ,$$ where $u$ is the point in $T_\bv$
    onto which projects $\QQ(\partial \bu )$ and $v$ is the point in
    $T_\bu$ onto which projects $\QQ(\partial \bv )$ (see Lemma
    \ref{point});

    \me

    \item If $\bu \subsetneq \bv$, $u\in T_\bv$ is the point onto
    which projects $\QQ(\partial \bu )$ and $T_\bv \setminus
    \{u\}=\bigsqcup_{i\in I} \calc_i$ is the decomposition into
    connected components then the image of $\psi_{\bu , \bv }$ is $$
(T_\bu \times \{u\})\cup \bigsqcup_{i\in I}(\{t_i\}\times \calc_i)\, ,
$$ where $t_i$ are points in $T_\bu$.
\end{enumerate}
\end{theorem}

\proof Case (1) is obvious.  We prove (2).  Let $\seq \mu$ be a point
in $\AM$ whose projection on $T_\bu$ is different from $v$.  Then
$\tdlu (\seq\mu , \partial \bv) >0$ which implies that $\lim_\omega
\dist_{C(U_n)}(\mu_n, \partial V_n) =+\infty$.  Theorem \ref{projest}
implies that \uass $\dist_{C(V_n)}(\mu_n, \partial U_n)\leq D$.  Hence
$\plv (\seq\mu)$ and  $\QQ(\plv (\partial
\bu ))$ are in the same piece of $\lv$, so $\seq\mu$ projects on $T_\bv$ in $u$.  The set of $\seq
\mu$ in $\AM$ projecting on $T_\bu$ in $v$ contains $\QQ(\partial \bv
)$, hence their projection on $T_\bv$ is surjective.

We now prove (3).  As before the set of $\seq \mu$ projecting on
$T_\bv$ in $u$ contains $\QQ(\partial \bu )$, hence it projects on
$T_\bu$ surjectively.

For every $i\in I$ we choose $\seq\mu_i\in \AM$ whose projection on
$T_\bv$ is in $\calc_i$.  Then every $\seq\mu$ with projection on
$T_\bv$ in $\calc_i$ has the property that any topological arc joining
$\plv (\seq\mu)$ to $\plv(\seq\mu_i)$ does not intersect the piece
containing $\QQ(\partial \bu )$.  Otherwise by property ($T_2'$) of
tree-graded spaces the geodesic joining $\plv (\seq\mu)$ to
$\plv(\seq\mu_i)$ in $\lv$ would intersect the same piece, and since
geodesics in $\lv$ project onto geodesics in $T_\bv$
\cite{DrutuSapir:splitting} the geodesic in $T_\bv$ joining the
projections of $\seq\mu$ and $\seq\mu_i$ would contain $u$.  This
would contradict the fact that both projections are in the same
connected component $\calc_i$.

Take $(\mu_n)$ representatives of $\seq\mu$ and $\left( \mu_n^i \right)$
representatives of $\seq\mu_i$.  The above and Lemma \ref{MM2:LLL}
imply that \uass $\dist_{C(U_n)} (\mu_n , \mu_n^i )\leq M$, hence the
projections of $\seq\mu$ and $\seq\mu_i$ onto $\lu $ are in the same
piece.  Therefore the projections of $\seq\mu$ and $\seq\mu_i$ onto
$T_\bu$ coincide.  Thus all elements in $\calc_i$ project in $T_\bu$
in the same point $t_i$ which is the projection of $\seq
\mu_i$.\endproof

\begin{remark}
Note that in cases (2) and (3) the image of $\psi_{\bu , \bv }$ has
dimension 1.  In case (3) this is due to the Hurewicz-Morita-Nagami
Theorem \cite[Theorem III.6]{Nagata:dimension}.
\end{remark}

We shall need the following classical Dimension Theory result:

\begin{theorem}[\cite{Engelking:dimension}]\label{dim}
Let $K$ be a compact metric space. If for every $\epsilon
>0$ there exists an $\epsilon$-map $f:K\to X$ (i.e. a continuous map with diameter of
$f^{-1}(x)$ at most $\epsilon$ for every $x\in X$) such that $f(K)$ is of dimension at most $n$, then
$K$ has dimension at most $n$.
\end{theorem}

\me

We give another proof of the following theorem:
\begin{theorem}[Dimension Theorem \cite{BehrstockMinsky:rankconj}]\label{rank}
Every locally compact subset of every asymptotic cone of the mapping
class group of a surface $S$ has dimension at most $\xi (S)$.
\end{theorem}

\begin{proof}

Since every subset of an asymptotic cone is itself a metric space, it
is paracompact.  This implies that every locally compact subset of the
asymptotic cone is a free union of $\sigma$-compact subspaces
\cite[Theorem 7.3, p.  241]{Dugundji:topology}.  Thus, it suffices to
prove that every compact subset in $\AM$ has dimension at most $\xi
(S)$.  Let $K$ be such a compact subset.  For simplicity we see it as
a subset of $\psi(\AM)\subset \Pi_{\bv\in\upss}T_\bv$.

Fix $\epsilon>0$.  Let $N$ be a finite $\frac{\epsilon}{4}$--net for
$(K, \td )$, i.e. a finite subset such that $K=\bigcup_{a\in N}
B_{\td} \left(a,\frac{\epsilon}{4}\right)$.  There exists a finite
subset $J_{\epsilon}\subset \upss$ such that for every $a,b\in N$,
$\sum_{\bu \not \in J_\epsilon} \tdlu (a,b) < \frac{\epsilon }{2}$.
Then for every $x,y \in K$, $\sum_{\bu \not \in J_\epsilon} \tdlu
(x,y) < \epsilon $.  In particular this implies that the projection
$\pi_{J_{\epsilon}}\co \Pi_{\bv\in\upss}T_\bv \to \Pi_{\bv\in
J_{\epsilon}}T_\bv$ restricted to $K$ is an $\epsilon$-map.

We now prove that for every finite subset $J\subset \upss$ the
projection $\pi_{J} (K)$ has dimension at most $\xi (S)$, by induction
on the cardinality of $J$.  This will finish the proof, due to Theorem
\ref{dim}.

If the subsurfaces in $J$ are pairwise disjoint then the cardinality
of $J$ is at most $3g+p-3$ and thus the dimension bound follows.  So,
suppose we have a pair of subsurfaces $\bu,\bv$ in $J$ which are not
disjoint: then they are either nested or overlapping.  We deal with
the two cases separately.

    Suppose $\bu,\bv \in J$ overlap.  Then according to Theorem
    \ref{projection trichotomy}, $\psi_{\bu , \bv } (\AM )$ is $
    \left( T_\bu \times \{u \} \right) \cup \left( \{v\}\times T_\bv
    \right) \, ,$ hence we can write $K=K_\bu \cup K_\bv$, where
    $\pi_{\bu , \bv } (K_\bu ) \subset T_\bu \times \{u \}$ and
    $\pi_{\bu , \bv } (K_\bv ) \subset \{v\}\times T_\bv $.  Now
    $\pi_{J} (K_\bu )= \pi_{J\setminus \{\bu \}} (K_\bu )\times \{u \}
    \subset \pi_{J\setminus \{\bu \}} (K )\times \{u \}$, which is of
    dimension at most $\xi (S)$ by induction hypothesis.  Likewise
    $\pi_{J} (K_\bv )= \pi_{J\setminus \{\bv \}} (K_\bv )\times \{v \}
    \subset \pi_{J\setminus \{\bv \}} (K )\times \{v \}$ is of
    dimension at most $\xi (S)$.  It follows that $\pi_{J} (K )$ is of
    dimension at most $\xi (S)$.

     \me

    Assume that $\bu \subsetneq \bv$.  Let $u$ be the point in $T_\bv$
    onto which projects $\QQ(\partial \bu )$ and $T_\bv \setminus
    \{u\}=\bigsqcup_{i\in I} \calc_i$ the decomposition into connected
    components.  By Theorem~\ref{projection trichotomy}, $\psi_{\bu ,
    \bv } (\AM )$ is $ (T_\bu \times \{u\})\cup \bigsqcup_{i\in
    I}(\{t_i\}\times \calc_i)\, , $ where $t_i$ are points in $T_\bu$.
    We prove that $\pi_{J} (K )$ is of dimension at most $\xi (S)$ by
    means of Theorem \ref{dim}.  Let $\delta >0$.  We shall construct
    a $2\delta$-map on $\pi_{J} (K )$ with image of dimension at most
    $\xi (S)$.  Let $N$ be a finite $\delta$-net of $( K , \td )$.
    There exist $i_1,...,i_m$ in $I$ such that $\pi_{\bu , \bv }(N)$
    is contained in ${\mathfrak T}=(T_\bu \times \{u\})\cup
    \bigsqcup_{j=1}^m(\{t_{i_j}\}\times \calc_{i_j})$.  The set
    $(T_\bu \times \{u\})\cup \bigsqcup_{i\in I}(\{t_i\}\times
    \calc_i)$ endowed with the $\ell^1$-metric is a tree and $\fct$ is
    a subtree in it.  We consider the nearest point retraction
    map $$\rt: (T_\bu \times \{u\})\cup \bigsqcup_{i\in
    I}(\{t_i\}\times \calc_i)\to \fct $$ which is moreover a
    contraction.  This defines a contraction $$\rt_{J}: \psi_{J} (\AM
    ) \to \psi_{J\setminus \{\bu , \bv \}} (\AM )\times \fct \, ,\, \,
    \rt_{J}=\mathrm{id} \times \rt\, .  $$

     The set $\pi_{J} (K )$ splits as $K_\fct \sqcup K'$, where
     $K_\fct = \pi_{J} (K ) \cap \pi_{\bu ,\bv}^{-1}(\fct)$ and $K'$
     is its complementary set.  Every $x\in K'$ has $\pi_{\bu
     ,\bv}(x)$ in some $\{t_i\}\times \calc_i$ with $i\in
     I\setminus\{i_1,...,i_m\}$.  Since there exists $n\in N$ such
     that $x$ is at distance smaller than $\delta$ from $\pi_{J} (n)$,
     it follows that $\pi_{\bu ,\bv}(x)$ is at distance smaller than
     $\delta$ from $\fct$, hence at distance smaller that $\delta $
     from $(t_i, u )=\rt \left(\pi_{\bu ,\bv}(x)\right)$.  We conclude
     that $\rt (\pi_{\bu ,\bv} (K')) \subset \{t_i \mid i\in I\}
     \times \{u\}\cap \pi_\bu (K) \times \{u\}$, hence $\rt_{J}
     (K')\subset \pi_{J\setminus \{ \bv \}} (K)\times \{u\}$, which is
     of dimension at most $\xi (S)$ by the induction hypothesis.

     By definition $\rt_{J} (K_\fct) = K_\fct$.  The set $K_\fct$
     splits as $K_\bu \sqcup \bigsqcup_{j=1}^m K_j\, ,$ where $K_\bu =
     \pi_{J} (K ) \cap \pi_{\bu ,\bv}^{-1}(T_\bu \times \{u\})$ and
     $K_j = \pi_{J} (K ) \cap \pi_{\bu ,\bv}^{-1}(\{t_{i_j}\}\times
     \calc_{i_j})$.  Now $K_\bu \subset \pi_{J\setminus \{\bv \}} (K )
     \times \{u\}$, while $K_j \subset \pi_{J\setminus \{\bu \}} (K )
     \times \{ t_{i_j}\}$ for $j=1,..,m$, hence by the induction
     hypothesis they have dimension at most $\xi (S)$.  Consequently
     $K_\fct$ has dimension at most $\xi (S)$.

     We have obtained that the map $\rt_{J}$ restricted to $\pi_{J} (K
     )$ is a $2\delta$-map with image $K_\fct \cup \rt_{J} (K')$ of
     dimension at most $\xi (S)$.  It follows that $\pi_{J} (K )$ is
     of dimension at most $\xi (S)$.\end{proof}

\subsection{The median structure}\label{sec:median}

More can be said about the structure of $\AM$ endowed with $\td$. We recall that a \emph{median space}
is a metric space for which, given any triple of points, there exists a unique \emph{median point},
that is a point which is simultaneously between any two points in that triple. A point $x$ is said to
be \emph{between} two other points $a,b$ in a metric space $(X, \dist )$ if $\dist (a,x)+\dist
(x,b)=\dist (a,b)$. See \cite{CDH-T} for details.

\begin{theorem}\label{thmedian}
The asymptotic cone $\AM$ endowed with the metric $\td$ is a median
space.  Moreover hierarchy paths (i.e. ultralimits of hierarchy
paths) are geodesics in $(\AM , \td )$.
\end{theorem}

The second statement follows from Proposition \ref{hiergeod}.
Note that the first statement is equivalent to that of $\psi (\AM)$
being a median subspace of the median space $(\calt_0 ,\td)$.  The
proof is done in several steps.

\begin{lemma}\label{lem:duproj}
Let $\seq\nu $ in $\AM$ and $\seq\Delta =(\Delta_n)^\omega$, where
$\Delta_n$ is a multicurve.  Let $\seq\nu' $ be the projection of
$\seq\nu $ on $\QQ(\seq\Delta)$.  Then for every subsurface $\bu$ such
that $\bu \notpitchfork \seq\Delta$ (i.e., $\bu$ does not overlap
$\Delta$) the distance $\tdlu (\seq\nu , \seq\nu')=0$.
\end{lemma}

\proof The projection of
$\seq\nu $ on $\QQ(\seq\Delta)$ is defined as limit of projections described in Section \ref{sect:proj}. Since $\bu \notpitchfork \seq\Delta$ the subsurface
$\bu =(U_n)^\omega$ is contained in a component,
$\bv=(V_n)^\omega$, of $S\setminus \seq\Delta =(S\setminus \Delta_n)^\omega$.  The
marking $\nu_n'$, by construction, does not differ from the
intersection of $\nu_n$ with $V_n$, and since $U_n \subseteq V_n$ the
same is true for $U_n$, hence $\dist_{C(U_n)}(\nu_n , \nu_n')=O(1)$.
On the other hand, if $\tdlu (\seq\nu , \seq\nu') >0$ then by Lemma
\ref{restr} $\lim_\omega \dist_{C(U_n)} (\seq\nu_n ,
\seq\nu_n')=+\infty$, whence a contradiction.\endproof

\me

\begin{lemma}\label{lem:betw}
Let $\seq\nu =\seqrep \nu $ and $ \seq\rho =\seqrep \rho$ be two
points in $\AM$, let $\seq\Delta =(\Delta_n)^\omega$, where $\Delta_n$
is a multicurve, and let $\seq\nu' , \seq\rho' $ be the respective
projections of $\seq\nu , \seq\rho $ on $\QQ(\seq\Delta)$.  Assume there
exist $\bu_1 = (U_n^1)^\omega,...,\bu_k=(U_n^k)^\omega$ subsurfaces
such that $\Delta_n = \partial U_n^1 \cup ...\cup \partial U_n^k$, and
$\dist_{C(U_n^i)} (\nu_n , \rho_n ) >M$ \uass for every $i=1,...,k$,
where $M$ is the constant in Lemma \ref{MM2:LLL}.

Then for every $\fh_1$, $\fh_2$ and $\fh_3$ hierarchy paths, joining
$\seq\nu , \seq\nu'$ respectively, $\seq\nu' , \seq\rho' $ and
$\seq\rho' , \seq\rho$, the path $\fh_1\sqcup \fh_2 \sqcup \fh_3$ is a
geodesic in $(\AM , \td)$.
\end{lemma}

\proof Let $\bv \in \upss$ be an arbitrary subsurface.  According to
Lemma \ref{projtu}, $\psi_\bv (\fh_i),\, i=1,2,3,$ is a geodesic in
$T_\bv$.  We shall prove that $\psi_\bv (\fh_1\sqcup \fh_2 \sqcup
\fh_3 )$ is a geodesic in $T_\bv$.

There are two cases: either $\bv \notpitchfork \seq\Delta$ (i.e., $\bv$ does not overlap
$\Delta$) or
$\bv \pitchfork \seq\Delta$ (i.e., $\bv$ overlaps $\Delta$).  In the first case, by
Lemma \ref{lem:duproj} the projections $\psi_\bv (\fh_1)$ and
$\psi_\bv (\fh_3)$ are singletons, and there is nothing to prove.

Assume now that $\bv \pitchfork \seq\Delta$.  Then $\bv
\pitchfork \partial \bu_i$ for some $i\in \{1,...,k\}$.

We have that \uass $$\dist_{C(U_n )} (\nu_n' , \rho_n' )\leq
\dist_{C(U_n )} (\nu_n' , \Delta_n )+\dist_{C(U_n )} (\rho_n' ,
\Delta_n )=O(1)\, .  $$

Lemma \ref{restr} then implies
that $\tdlu (\seq\nu' , \seq\rho')=0$.
Hence, $\psi_\bv (\fh_2)$ reduces to a singleton $x$ which is the
projection
onto $T_\bv$ of both $\QQ(\seq\Delta)$ and $\QQ(\partial \bu_i)$.  It
remains to prove that $\psi_\bv (\fh_1)$ and $\psi_\bv (\fh_3)$ have
in common only $x$.  Assume on the contrary that they are two geodesic
with a common non-trivial sub-geodesic containing $x$.  Then the
geodesic in $T_\bv$ joining $\psi_\bv (\seq\nu)$ and $\psi_\bv
(\seq\rho )$ does not contain $x$.  On the other hand, by hypothesis
and Lemma \ref{MM2:LLL} any hierarchy path joining $\seq\nu$ and
$\seq\rho$ contains a point in $\QQ(\partial \bu_i)$.  Lemma
\ref{projtu} implies that the geodesic in $T_\bv$ joining $\psi_\bv
(\seq\nu)$ and $\psi_\bv (\seq\rho )$ contains $x$, yielding a
contradiction.

 Thus $\psi_\bv (\fh_1)\cap \psi_\bv (\fh_3)=\{ x \}$ and $\psi_\bv
 (\fh_1\sqcup \fh_2 \sqcup \fh_3 )$ is a geodesic in $T_\bv$ also in
 this case.

We proved that $\psi_\bv (\fh_1\sqcup \fh_2 \sqcup \fh_3 )$ is a
geodesic in $T_\bv$ for every $\bv \in \upss $.  This implies that
$\fh_1 \cap \fh_2 = \{ \seq\nu' \}$ and $\fh_2 \cap \fh_3 = \{
\seq\rho' \}$, and that $\fh_1 \cap \fh_3 = \emptyset$ if $\fh_2$ is
non-trivial, while if $\fh_2$ reduces to a singleton $\seq\nu'$,
$\fh_1 \cap \fh_3 = \{ \seq\nu' \}$.  Indeed if for instance $\fh_1
\cap \fh_2$ contained a point $\seq\mu \neq \seq\nu'$ then $\td
(\seq\mu , \seq\nu') >0$ whence $\tdlv (\seq\mu , \seq\nu') >0$ for
some subsurface $\bv$.  It would follow that $\psi_\bv (\fh_1),
\psi_\bv (\fh_2)$ have in common a non-trivial sub-geodesic,
contradicting the proven statement.

Also, if $\fh_1 \cap \fh_3$ contains a point $\seq\mu\neq \seq\nu'$
then for some subsurface $\bv$ such that $\bv \cap \seq\Delta \neq
\emptyset$, $\tdlv (\seq\mu , \seq\nu') >0$.  Since $\tdlv (\seq\nu' ,
\seq\rho')=0$ it follows that $\tdlv (\seq\mu , \seq\rho') >0$ and
that $\psi_\bv (\fh_1\sqcup \fh_2 \sqcup \fh_3 )$ is not a geodesic in
$T_\bv$.  A similar contradiction occurs if $\seq\mu\neq \seq\rho'$.
Therefore if $\seq\mu$ is a point in $\fh_1 \cap \fh_3$, then we must
have $\seq\mu = \seq\nu' =\seq\rho'$, in particular $\fh_2$ reduces to
a point, which is the only point that $\fh_1$ and $\fh_3$ have in
common.

Thus in all cases $\fh_1\sqcup \fh_2 \sqcup \fh_3$ is a topological
arc.  Since $\fh_i$ for  $i=1,2,3,$ each satisfy the hypotheses of Lemma
\ref{hiergeod} and for every $\bv \in \upss $, $\psi_\bv (\fh_1\sqcup
\fh_2 \sqcup \fh_3 )$ is a geodesic in $T_\bv$, it follows that
$\fh_1\sqcup \fh_2 \sqcup \fh_3$ also satisfies the hypotheses of
Lemma \ref{hiergeod}.  We may therefore conclude that $\fh_1\sqcup \fh_2
\sqcup \fh_3$ is a geodesic in $(\AM , \td)$.\endproof

\me

\begin{defn}\label{def:between}
A point $\seq\mu$ in $\AM$ is \emph{between} the points $\seq\nu ,
\seq\rho$ in $\AM$ if for every $\bu \in \upss$ the projection
$\psi_\bu (\seq\mu)$ is in the geodesic joining $\psi_\bu (\seq\nu)$
and $\psi_\bu (\seq\rho)$ in $T_\bu$ (possibly identical to one of its
endpoints).
\end{defn}

\me

\begin{lemma}\label{meddisj}
For every triple of points $\seq\nu , \seq\rho , \seq \sigma$ in
$\AM$, every choice of a pair $\seq\nu , \seq\rho$ in the triple and
every finite subset $F$ in $\upss$ of pairwise disjoint subsurfaces
there exists a point $\seq\mu$ between $\seq\nu , \seq\rho$ such that
$\psi_F (\seq\mu )$ is the median point of $\psi_F(\seq\nu),
\psi_F(\seq\rho), \psi_F(\seq\sigma )$ in $\prod_{\bu \in F} T_\bu$.
\end{lemma}

 \proof Let $F= \{ \bu_1,...,\bu_k\}$,
where $\bu_i =\left( U^i_n \right)^\omega$.  We argue by induction on $k$.  If $k=1$ then the statement
follows immediately from Lemma \ref{projtu}.  We assume that the statement is true for all $i<k$, where
$k\geq 2$, and we prove it for $k$.

We consider the multicurve $\Delta_n= \partial U^1_n \cup \cdots\cup
\partial U^k_n$.  We denote the set $\{1,2,...,k\}$ by $I$.  If for
some $i\in I$, $\td_{\bu_i } (\seq\nu , \seq\rho )=0$ then the median
point of $\psi_{\bu_i } (\seq\nu)$, $\psi_{\bu_i } (\seq\rho )$,
$\psi_{\bu_i } (\seq\sigma )$ is $\psi_{\bu_i } (\seq\nu)= \psi_{\bu_i
} (\seq\rho )$.  By the induction hypothesis there exists $\seq\mu$
between $\seq\nu , \seq\rho$ such that $\psi_{F\setminus \{i\}}
(\seq\mu )$ is the median point of $\psi_{F\setminus \{i\}}(\seq\nu),
\psi_{F\setminus \{i\}}(\seq\rho), \psi_{F\setminus \{i\}}(\seq\sigma
)$.  Since $\psi_{\bu_i } (\seq\mu)=\psi_{\bu_i } (\seq\nu)=
\psi_{\bu_i } (\seq\rho )$ it follows that the desired statement holds
not just for $F\setminus \{i \}$, but for all of $F$ as well.

Assume now that for all $i\in I$, $\td_{\bu_i } (\seq\nu , \seq\rho )
> 0$.  Lemma \ref{restr} implies that $\lim_\omega \dist_{C(U_n^i)}
(\nu_n ,\rho_n )=\infty$.

Let $\seq\nu' , \seq\rho' , \seq \sigma'$ be the respective
projections of $\seq\nu , \seq\rho , \seq \sigma$ onto
$\QQ(\seq\Delta)$, where $\seq\Delta = (\Delta_n)^\omega$.  According to
Lemma \ref{lem:duproj}, $\td_{\bu_i}(\seq\nu , \seq\nu ')=
\td_{\bu_i}(\seq\rho , \seq\rho')=\td_{\bu_i}(\seq\sigma ,
\seq\sigma')=0$ for every $i\in I$, whence
$\psi_F(\seq\nu)=\psi_F(\seq\nu'), \psi_F(\seq\rho)=\psi_F(\seq\rho'),
\psi_F(\seq\sigma )=\psi_F(\seq\sigma')$.  This and Lemma
\ref{lem:betw} imply that it suffices to prove the statement for
$\seq\nu' , \seq\rho' , \seq \sigma'$.  Thus, without loss of
generality we may assume that $\seq\nu , \seq\rho , \seq \sigma$ are
in $\QQ(\seq\Delta)$.  Also without loss of generality we may assume
that $\{ U^1_n, U^2_n,...,U^k_n \}$ are all the connected components
of $S\setminus \Delta_n$ and all the annuli with core curves in
$\Delta_n$.  If not, we may add the missing subsurfaces.

For every $i\in I$ we consider the projections $\nu^{i}_n$, $\rho^{i}_n$ of $\nu_n$ and, respectively,
$\rho_n$ on $\MM (U^i_n)$. Let $\g^i_n$ be a hierarchy path in $\MM (U^i_n)$ joining $\nu^{i}_n$,
$\rho^{i}_n$ and let $\seq \g^i = \seqrep{\g^i}$ be the limit hierarchy path in $\MM (\bu_i)$.
According to Lemma \ref{projtu}, for every $i\in I$ there exists $\mu^{i}_n$ on $\g^i_n$ such that
$\seq{\mu^{i}}=\seqrep{\mu^{i}}$ projects on $T_{\bu_i}$ on the median point of the projections of
$\seq\nu , \seq\rho , \seq \sigma$.
Let $\pgot^i_n$ and $\q^i_n$ be the subpaths of $\g^i_n$
preceding and respectively succeeding $\mu^{i}_n$ on $\g^i_n$, and let $\seq \pgot^i =
\seqrep{\pgot^i}$ and $\seq \q^i = \seqrep{\q^i}$ be the limit hierarchy paths in $\MM (\bu_i)$.

Let $\widetilde{\pgot}^1_n$ be a path in $\MM (S)$ which starts at
$\nu_n$ and then continues on a path
obtained by markings whose restriction to $U_{1}^{n}$ are
given by $\pgot^1_n$ and in the complement
of $U_{1}^{n}$ are given by the restriction of $\nu_{n}$.
Continue this path by concatenating a path, $\widetilde{\pgot}^2_n$,
obtained by starting from the terminal point of
$\widetilde{\pgot}^1_n$ and then continuing by markings which
are all the same in the
complement of $U^{2}_{n}$ while their restriction to $U^{2}_{n}$ are
given by $\pgot^2_n$. Similarly, we obtain
$\pgot^j_n$ is from $\pgot^{j-1}_n$ for any $j\leq k$. Note that for
any $1\leq j\leq k$ and any $i\neq j$,
the path $\pgot^j_n$ restricted to $U^i_n$ is constant.
Note that the starting point of
$\widetilde{\pgot}^1_n \sqcup \cdots \sqcup \widetilde{\pgot}^k_n$ is
$\nu_n$ and the terminal point is the marking $\mu_n$ with the
property that it projects on $\MM(U^i_n)$ in $\mu^i_n$ for every $i\in
I$.  Now consider $\widetilde{\q}^1_n$ the path with
starting point $\mu_n$ obtained following $\q^1_n$ (and keeping the
projections onto $U^i_n$ with $i\in I\setminus \{ 1\}$ unchanged),
then $\widetilde{\q}^2_n$,...  $\widetilde{\q}^k_n$ constructed such
that the starting point of $\widetilde{\q}^j_n$ is the terminal point
of $\widetilde{\q}^{j-1}_n$, and $\widetilde{\q}^j_n$ is obtained
following $\q^j_n$ (and keeping the projections onto $U^i_n$ with
$i\in I\setminus \{ j\}$ unchanged).

Let $\seq{\widetilde{\pgot}^j}=\seqrep{\widetilde{\pgot}^j}$ and
$\seq{\widetilde{\q}^j}=\seqrep{\widetilde{\q}^j}$.  We prove that for
every subsurface $\bv=(V_n)^\omega$ the path
$\seq{\fh}=\seq{\widetilde{\pgot}^1} \sqcup \cdots \sqcup
\seq{\widetilde{\pgot}^k}\sqcup \seq{\widetilde{\q}^1} \sqcup \cdots
\sqcup \seq{\widetilde{\q}^k}$ projects onto a geodesic in $T_\bv$.
For any $i\neq
j$, we have that $\widetilde{\pgot}^i_n\cup \widetilde{\q}^i_n$ and
$\widetilde{\pgot}^j_n\cup \widetilde{\q}^j_n$
have disjoint support. Hence for each $i\in I$ we have that the
restriction to $U^{i}_{n}$ of the entire path is the same as the
restriction to $U^{i}_{n}$ of
$\widetilde{\pgot}^i_n\cup \widetilde{\q}^i_n$. Since the latter
is by construction the hierarchy path
$\g^i_n$, if $\bv\subset \bu_{i}$ for some $i\in I$, then it
follows from Lemma~\ref{hiergeod} that $\seq{\fh}$ projects to a
geodesic in $T_{\bv}$.
If $\bv$ is disjoint from
$\bu_i$ then all the markings composing $\fh_n'$ have the same
intersection with $V_n$, whence the diameter of $\fh_n$ with respect
to $\dist_{C(V_n)}$ must be uniformly bounded.  This and
Lemma~\ref{restr} implies that $\psi_\bv (\seq{\fh'} )$ is a singleton.
Lastly,
if $\bv$ contains or overlaps $\bu_i$, then since all the markings in
$\fh_n'$ contain $\partial U^i_n$ the diameter of $\fh_n$ with respect
to $\dist_{C(V_n)}$ must be uniformly bounded, leading again to the
conclusion that $\psi_\bv (\seq{\fh'} )$ is a singleton.  Thus, the
only case when $\psi_\bv (\seq{\fh'} )$ is not a singleton is when
$\bv \subseteq \bu_i$.

Now let $\bv $ denote an arbitrary subsurface.  If it is not contained in any
$\bu_i$ then $\psi_\bv (\seq{\fh} )$ is a singleton.  The other
situation is when $\bv$ is contained in some $\bu_i$, hence disjoint
from all $\bu_j$ with $j\neq i$.  Then $\psi_\bv (\seq{\fh} ) =
\psi_\bv (\seq{\widetilde{\pgot}^i} \sqcup \seq{\widetilde{\q}^i} ) =
\psi_\bv^i (\seq \pgot^i \sqcup \seq \q^i) = \psi_\bv^i (\seq \g^i)$,
which is a geodesic.  Note that in the last two inequalities the map
$\psi_\bv^i$ is the natural projection of $\MM (\bu_i )$ onto $T_\bv$
which exists when $\bv \subseteq \bu_i$.

Thus we have shown that $\seq{\fh}$ projects onto a geodesic in
$T_\bv$ for every $\bv$, whence $\seq\mu$ is between $\seq\nu ,
\seq\rho$.  By construction, for every $i\in I$, $\psi_{\bu_i}
(\seq\mu )$ is the median point of $\psi_{\bu_i}(\seq\nu),
\psi_{\bu_i}(\seq\rho), \psi_{\bu_i}(\seq\sigma )$, equivalently
$\psi_F (\seq\mu )$ is the median point of $\psi_F(\seq\nu)$, $\psi_F(\seq\rho)$, $\psi_F(\seq\sigma )$ in $\prod_{\bu \in F}
T_\bu$.
\endproof

We now generalize the last lemma by removing the hypothesis that the
subsurfaces are disjoint.

\begin{lemma}\label{medgen}
For every triple of points $\seq\nu , \seq\rho , \seq \sigma$ in
$\AM$, every choice of a pair $\seq\nu , \seq\rho$ in the triple and
every finite subset $F$ in $\upss$ there exists a point $\seq\mu$ in
$\AM$ between $\seq\nu , \seq\rho$ such that $\psi_F (\seq\mu )$ is
the median point of $\psi_F(\seq\nu), \psi_F(\seq\rho),
\psi_F(\seq\sigma )$ in $\prod_{\bu \in F} T_\bu$.
\end{lemma}

\proof We prove the statement by induction on the cardinality of $F$.
When $\card F=1$ it follows from Lemma~\ref{projtu}.  Assume that it
is true whenever $\card F < k$ and consider $F$ of cardinality $k\geq 2$.
If the subsurfaces in $F$ are pairwise disjoint then we can apply
Lemma~\ref{meddisj}, hence we may assume that there exists a pair of
subsurfaces $\bu , \bv$ in $F$ which either overlap or are nested.

First, assume that $\bu ,\bv$ overlap.
Then $\psi\buv$ is equal to $\left(T_\bu \times \{u \} \right)
\cup \left( \{v\}\times T_\bv \right)$, by
Theorem~\ref{projection trichotomy}.  We write $\seq\nu\buv$ to denote
the image $\psi\buv (\seq\nu)$ and let
$\seq\nu_\bu$ and $\seq\nu_\bv$ denote
its coordinates (i.e. $\psi_\bu (\seq\nu)$ and $\psi_\bv (\seq\nu)$).
We use similar notations for $\seq\rho , \seq \sigma$.  If the median
point of $\seq\nu\buv , \seq\rho\buv , \seq \sigma\buv$ is not $(v,u)$
then it is either some point $(x,u)$ with $x\in T_\bu \setminus \{
v\}$, or $(v,y)$ with $y\in T_\bv \setminus \{ u\}$.  In the first
case, by the induction hypothesis, there exists a point
$\seq\mu_{1}$ between $\seq\nu , \seq\rho$ such
that $\psi_{F\setminus \{\bv \}} (\seq\mu_{1})$ is the median point of
$\psi_{F\setminus \{\bv \}}(\seq\nu), \psi_{F\setminus \{\bv
\}}(\seq\rho), \psi_{F\setminus \{\bv \}}(\seq\sigma )$.
In particular,
$\psi_\bu (\seq\mu)=x$, hence $\psi\buv (\seq\mu)$ is a point in
$ \left( T_\bu \times \{u \} \right) \cup \left( \{v\}\times T_\bv
\right)$ having the first coordinate $x$.  Since there exists only
one such point, $(x,u)$, it follows that $\psi_\bv (\seq\mu )=u$.
Thus, for every $\mathbf{Y}\in F$, the point
$\psi_{\mathbf{Y}} (\seq\mu)$ is the median point in
$T_{\mathbf{Y}}$ of $\psi_{\mathbf{Y}}
(\seq\nu )$, $\psi_{\mathbf{Y}} (\seq\rho )$ and $\psi_{\mathbf{Y}}
(\seq\sigma )$.  This is equivalent to the fact that $\psi_{F}
(\seq\mu )$ is the median point in $\prod_{\mathbf{Y}} T_{\mathbf{Y}}
$ of $\psi_{F} (\seq\nu )$, $\psi_{F} (\seq\rho )$ and $\psi_{F}
(\seq\sigma )$.
A similar argument works when the median point of $\seq\nu\buv ,
\seq\rho\buv , \seq \sigma\buv$ is a point $(v,y)$ with $y\in T_\bv
\setminus \{ u\}$.

Hence, we may assume that the median point of
$\seq\nu\buv , \seq\rho\buv , \seq \sigma\buv$ is $(v,u)$.  Let
$\seq\mu_1$ be a point between $\seq\nu , \seq\rho$ such that
$\psi_{F\setminus \{\bv \}} (\seq\mu_1 )$ is the median point of
$\psi_{F\setminus \{\bv \}}(\seq\nu), \psi_{F\setminus \{\bv
\}}(\seq\rho), \psi_{F\setminus \{\bv \}}(\seq\sigma )$, and let
$\seq\mu_2$ be a point between $\seq\nu , \seq\rho$ such that
$\psi_{F\setminus \{\bu \}} (\seq\mu )$ is the median point of
$\psi_{F\setminus \{\bu \}}(\seq\nu), \psi_{F\setminus \{\bu
\}}(\seq\rho), \psi_{F\setminus \{\bu \}}(\seq\sigma )$.  In
particular $\psi\buv (\seq\mu_1)= (v,y)$ with $y\in T_\bv $ and
$\psi\buv (\seq\mu_2)=(x,u)$ with $x\in T_\bu$.  Any hierarchy path
joining $\seq\mu_1$ and $\seq\mu_2$ is mapped by $\psi\buv$ onto a
path joining $(v,y)$ and $(x,u)$ in $ \left( T_\bu \times \{u \}
\right) \cup \left( \{v\}\times T_\bv \right) \, $.  Therefore it
contains a point $\seq\mu$ such that $\psi\buv(\seq\mu)$ is $(v,u)$.
According to Lemma~\ref{projtu} $\seq\mu $ is between $\seq\mu_1$ and
$\seq\mu_2$, hence it is between $\seq\nu$ and $\seq\rho$, moreover
for every $\mathbf{Y}\in F\setminus \{ \bu,\bv\}$, $\psi_{\mathbf{Y}}
(\seq\mu )= \psi_{\mathbf{Y}} (\seq\mu_1 )=\psi_{\mathbf{Y}}
(\seq\mu_2 )$, and it is the median point in $T_{\mathbf{Y}}$ of
$\psi_{\mathbf{Y}} (\seq\nu )$, $\psi_{\mathbf{Y}} (\seq\rho )$ and
$\psi_{\mathbf{Y}} (\seq\sigma )$.  This and the fact that $\psi\buv
(\seq\mu )=(v,u)$ is the median point of $\seq\nu\buv$, $\seq\rho\buv$
and $\seq\sigma\buv$ finish the argument in this case.

We now consider the case that $\bu\subsetneq \bv$.
Let $u$ be the point in $T_\bv$
which is the projection of $\QQ(\partial \bu )$ and let
$T_\bv \setminus\{u\}=\bigsqcup_{i\in I} \calc_i$ be the
decomposition into connected
components.  By Theorem \ref{projection trichotomy}, the image of
$\psi_{\bu , \bv }$ is $(T_\bu \times \{u\})\cup \bigsqcup_{i\in
I}(\{t_i\}\times \calc_i)$ where $t_i$ are points in $T_\bu$.  If
the median point of $\seq\nu\buv$, $\seq\rho\buv$ and $\seq\sigma\buv$
is not in the set $\{(t_i,u)\mid i\in I \}$, then we are done as in the
previous case using the induction hypothesis as well as the fact that
for such points there are no other points having the same first
coordinate or the same second coordinate.

Thus, we may assume that the median point of $\seq\nu\buv$,
$\seq\rho\buv$ and $\seq\sigma\buv$ is $(t_i,u)$ for some $i\in I$.
Let $\seq\mu_1$ be a point between $\seq\nu , \seq\rho$ such that
$\psi_{F\setminus \{\bv \}} (\seq\mu_1 )$ is the median point of
$\psi_{F\setminus \{\bv \}}(\seq\nu), \psi_{F\setminus \{\bv
\}}(\seq\rho), \psi_{F\setminus \{\bv \}}(\seq\sigma )$, and let
$\seq\mu_2$ be a point between $\seq\nu , \seq\rho$ such that
$\psi_{F\setminus \{\bu \}} (\seq\mu )$ is the median point of
$\psi_{F\setminus \{\bu \}}(\seq\nu), \psi_{F\setminus \{\bu
\}}(\seq\rho), \psi_{F\setminus \{\bu \}}(\seq\sigma )$.  In
particular $\psi\buv (\seq\mu_1)= (t_i,y)$ with $y\in C_i $ and
$\psi\buv (\seq\mu_2)=(x,u)$ with $x\in T_\bu$.  Any hierarchy path
joining $\seq\mu_1$ and $\seq\mu_2$ is mapped by $\psi\buv$ onto a
path joining $(t_i,y)$ and $(x,u)$ in $ \left( T_\bu \times \{u \}
\right) \cup \bigsqcup_{i\in I}(\{t_i\}\times \calc_i)$.  It contains
a point $\seq\mu$ such that $\psi\buv(\seq\mu)$ is $(t_i,u)$.
By Lemma~\ref{projtu}, the point  $\seq\mu$ is between $\seq\mu_1$ and
$\seq\mu_2$, hence in particular it is between $\seq\nu$ and $\seq\rho$.
Moreover,
for every $\mathbf{Y}\in F\setminus \{ \bu,\bv\}$, $\psi_{\mathbf{Y}}
(\seq\mu )= \psi_{\mathbf{Y}} (\seq\mu_1 )=\psi_{\mathbf{Y}}
(\seq\mu_2 )$ and hence $\seq\mu$ is the median point in $T_{\mathbf{Y}}$ of
$\psi_{\mathbf{Y}} (\seq\nu )$, $\psi_{\mathbf{Y}} (\seq\rho )$ and
$\psi_{\mathbf{Y}} (\seq\sigma )$.  This and the fact that $\psi\buv
(\seq \mu )=(t_i,u)$ is the median point of $\seq\nu\buv$,
$\seq\rho\buv$ and $\seq\sigma\buv$ finish the argument.
\endproof

\noindent \emph{Proof of Theorem \ref{thmedian}.} \quad Consider an
arbitrary triple of points $\seq\nu , \seq\rho , \seq \sigma$ in
$\AM$.  For every $\epsilon >0$ there exists a finite subset $F$ in
$\upss$ such that $\sum_{\bu \in \upss \setminus F}\tdlu (\seq a ,
\seq b) <\epsilon$ for every $\seq a , \seq b$ in $\{ \seq\nu ,
\seq\rho , \seq \sigma \}$.  By Lemma \ref{medgen} there exists
$\seq\mu$ in $\AM$ between $\seq\nu , \seq\rho$ such that $\psi_F
(\seq\mu )$ is the median point of $\psi_F(\seq\nu), \psi_F(\seq\rho),
\psi_F(\seq\sigma )$ in $\prod_{\bu \in F} T_\bu$.  The latter implies
that for every $\seq a , \seq b$ in $\{ \seq\nu , \seq\rho , \seq
\sigma \}$,
$$
\sum_{\bu \in F}\tdlu (\seq a , \seq b) = \sum_{\bu \in F}\tdlu (\seq
a , \seq \mu) + \sum_{\bu \in F}\tdlu (\seq \mu , \seq b)\, .
$$

Also, since $\seq\mu$ is between $\seq\nu , \seq\rho$ it follows that
$$
\sum_{\bu \in \upss \setminus F}\tdlu (\seq\nu , \seq\mu)
<\epsilon\mbox{ and } \sum_{\bu \in \upss \setminus F}\tdlu (\seq \mu
, \seq \rho) <\epsilon\, , $$ whence
$$
\sum_{\bu \in \upss \setminus F}\tdlu (\seq\mu , \seq\sigma )
<2\epsilon\, .
$$

It follows that for every $\seq a , \seq b$ in $\{ \seq\nu , \seq\rho
, \seq \sigma \}$,
$$
\sum_{\bu \in \upss }\tdlu (\seq a , \seq \mu) + \sum_{\bu \in \upss
}\tdlu (\seq \mu , \seq b)\leq \sum_{\bu \in \upss }\tdlu (\seq a ,
\seq b) +3\epsilon\, .
$$

That is, $\td (\seq a , \seq \mu) + \td (\seq \mu , \seq b) \leq \td
(\seq a , \seq b) +3\epsilon\, .$ This and \cite[Section 2.3]{CDH-T}
imply that $\psi (\seq\mu )$ is at distance at most $5\epsilon$ from
the median point of $\psi(\seq\nu ), \psi (\seq\rho ), \psi
(\seq \sigma )$ in $\calt_0$.

We have thus proved that for every $\epsilon >0$ there exists a point
$\psi (\seq\mu )$ in $\psi (\AM)$ at distance at most $5\epsilon$ from
the median point of $\psi(\seq\nu ), \psi (\seq\rho ), \psi (\seq
\sigma )$ in $\calt_0$.  Now the asymptotic cone $\AM$ is a complete
metric space with the metric $\dist_\AM$, hence the bi-Lipschitz
equivalent metric space $\psi (\AM )$ with the metric $\td$ is also
complete.  Since it is a subspace in the complete metric space
$\calt_0$, it follows that $\psi (\AM )$ is a closed subset in
$\calt_0$.  We may then conclude that $\psi (\AM )$ contains the
unique median point of $\psi(\seq\nu ), \psi (\seq\rho ), \psi (\seq
\sigma )$ in $\calt_0$.  \hspace*{\fill}$\Box$


\section{Actions on asymptotic cones of mapping class groups and
splitting}\label{sec:actions}

\subsection{Pieces of the asymptotic cone}\label{sec:pieces}

\begin{lemma}\label{pinchedquad}
    Let $\seqrep{\mu},\seqrep{\mu'},\seqrep{\nu},\seqrep{\nu'}$ be
    sequences of points in $\MM(S)$ for which
    $\ulim \dist_{\MM (S)}(\mu_{n},\nu_{n})=\infty$.
    For every $M> 2K(S)$, where $K(S)$ is the constant in
    Theorem~\ref{distanceformula} there exists a positive constant
    $C=C(M)<1$ so that if
    $$\dist_{\MM(S)}(\mu_{n},\mu'_{n})+\dist_{\MM(S)}(\nu_{n},\nu'_{n})\leq C \dist_{\MM (S)}(\mu_{n},\nu_{n})\, ,
    $$ then
    there exists a sequence of subsurfaces $Y_{n}\subseteq S$
    such that for $\omega$-a.e.~$n$ both
    $\dist_{\CC(Y_n)}(\mu_{n},\nu_{n})>M$ and
    $\dist_{\CC(Y_n)}(\mu'_{n},\nu'_{n})>M$.
\end{lemma}

\begin{proof} Assume that \uass the sets of subsurfaces
$$\mathcal{Y}_n = \{ Y_n \mid \dist_{C(Y_n)} (\mu_n, \nu_n )> 2M
\}\mbox{ and }\mathcal{Z}_n= \{ Z_n \mid \dist_{C(Z_n)} (\mu_n',
\nu_n' )> M \}$$ are disjoint.  Then for every $Y_n \in
\mathcal{Y}_n$, $\dist_{C(Y_n)} (\mu_n', \nu_n' )\leq M$, which by the
triangle inequality implies that $\dist_{C(Y_n)} (\mu_n, \mu_n' ) +
\dist_{C(Y_n)} (\nu_n, \nu_n' )\geq \dist_{C(Y_n)} (\mu_n, \nu_n )- M
\geq \frac{1}{2} \dist_{C(Y_n)} (\mu_n, \nu_n )> M $.  Hence either $\dist_{C(Y_n)} (\mu_n, \mu_n' )$ or $\dist_{C(Y_n)} (\nu_n,
\nu_n' )$ is larger than $M/2 > K(S)$.  Let $a,b$ be the constants
appearing in (\ref{fdistformula}) for $K=M/2$, and let $A,B$ be the
constants appearing in the same formula for $K'=2M$.  According to the
above we may then write
$$
\dist_{\MM(S)} (\mu_n, \mu_n' ) + \dist_{\MM(S)} (\nu_n, \nu_n' )
\geq_{a,b} \sum_{Y\in \mathcal{Y}_n }
    \Tsh K{\dist_{\CC(Y)}(\mu_n ,\mu_n' )} +
     $$
    $$\sum_{Y\in \mathcal{Y}_n }
    \Tsh K{\dist_{\CC(Y)}(\nu_n, \nu_n' )}
    \geq    \frac{1}{4} \sum_{Y\in \mathcal{Y}_n }
    \dist_{\CC(Y)}(\mu_n ,\nu_n )\geq_{A,B} \frac{1}{4}
    \dist_{\MM(S)}(\mu_n ,\nu_n )\, .
$$

The coefficient $\frac{1}{4}$ is accounted for by the case when one of
the two distances $\dist_{C(Y_n)} (\mu_n, \mu_n' )$ and
$\dist_{C(Y_n)} (\nu_n, \nu_n' )$ is larger than $K=M/2$ while the
other is not.

When $C$ is small enough we thus obtain a contradiction of the
hypothesis, hence \uass $\mathcal{Y}_n \cap \mathcal{Z}_n \neq \emptyset$
\end{proof}

\begin{defn} For any $g=(g_n)^\omega \in \UM$ let us denote by
$U(g)$ the set of points $\seq h\in\AM$ such that for some representative $(h_n)^\omega\in\UM$ of $\seq
h$,
$$\ulim\dist_{\CC(S)}(h_n, g_n) < \infty.$$ The set $U(g)$ is
called the {\em $g$-interior}. This set is non-empty since $g \in U(g)$.
\end{defn}

\begin{lemma}\label{interior} Let $P$ be a piece in $\AM=\coneM$.  Let
$\seq x, \seq y$ be distinct points in $P$.  Then there exists
$g=(g_n)^\omega \in \UM$ such that $U(g) \subseteq P$; moreover the
intersection of any hierarchy path $[\seq x, \seq y]$ with $U(g)$
contains $[\seq x, \seq y]\setminus \{\seq x, \seq y\}$.
\end{lemma}

\proof Consider arbitrary representatives $(x_n)^\omega ,(y_n)^\omega
$ of $\seq x$ and respectively $\seq y$, and let $\seq{[x,y]}$ be the
limit of a sequence of hierarchy paths $[x_n, y_n]$.  Since $\seq x,
\seq y\in P$, there exist sequences of points $\seq{x}(k)=\lio{x_n(k)},
\seq{y}(k)=\lio{y_n(k)}$, $x_n(k), y_n(k)\in [x_n, y_n]$ and a sequence
of numbers $C(k)>0$ such that for $\omega$-almost every~$n$ we have:
$$\dcs(x_n(k),y_n(k))<C(k)$$ and
\begin{equation}\label{eqkk}\begin{array}{l}\dmm(x_n(k),
x_n)<\frac{d_n\dist_\AM(\seq x,\seq y)}{k},\\
\dmm(y_n(k),y_n)<\frac{d_n\dist_\AM(\seq x,\seq y)}{k}.\end{array}
\end{equation}

Let $[x_n(k), y_n(k)]$ be the subpath of $[x_n, y_n]$ connecting
$x_n(k)$ and $y_n(k)$.  Let $g_n$ be the midpoint of $[x_n(n),
y_n(n)]$.  Then $\seqrep g \in \seq{[x,y]}$.  Let $g=(g_n)^\omega\in
\UM$.  Let us prove that $U(g)$ is contained in $P$.

Since $\seq x,\seq y\in P$, it is enough to show that any point $\seq
z=\seqrep z $ from $U(g)$ is in the same piece with $\seq x$ and in
the same piece with $\seq y$ (because distinct pieces cannot have two
points in common).

By the definition of $U(g)$, we can assume that $\omega$-a.s. $\dist_{\CC(S)}(z_n,
g_n)\le C_1$ for some constant $C_1$. For every $k>0$,
$\dist_{\CC(S)}(x_n(k),y_n(k))\le C(k)$, so
$$\dist_{\CC(S)}(x_n(k),z_n), \dist_{\CC(S)}(y_n(k),z_n)\le C(k)+C_1$$
$\omega$-a.s. By (\ref{eqkk}) and Theorem~\ref{classifypieces}, $\seq
x=\lio{x_n}, \seq z=\lio{z_n}, \seq y=\lio{ y_n}$ are in the same piece.

Note that $\lio{x_n(k)}$ and $\lio{y_n(k)}$ are in $U(g)$ for every $k$.
Now let $(x_n')^\omega, (y_n')^\omega$ be other representatives of
$\seq x$, $\seq y$, and let $x_n'(k), y_n'(k)$ be chosen as above on a
sequence of hierarchy paths $[x_n', y_n']$.  Let $g'=(g_n')^\omega$,
where $g_n'$ is the point in the middle of the hierarchy path $[x_n',
y_n']$.  We show that $U(g')=U(g)$.  Indeed, the sequence of
quadruples $x_n(k), y_n(k), y_n'(k), x_n'(k)$ satisfies the conditions
of Lemma \ref{pinchedquad} for large enough $k$.  Therefore the
subpaths $[x_n(k), y_n(k)]$ and $[x_n'(k), y_n'(k)]$ share a large
domain $\omega$-a.s. Since the entrance points of these subpaths in
this domain are at a uniformly bounded $\CC(S)$-distance, the same
holds for $g_n$, $g_n'$.  Hence $U(g')=U(g)$.  This completes the
proof of the lemma.  \endproof

Lemma \ref{interior} shows that for every two points $\seq x, \seq y$
in a piece $P$ of $\AM$, there exists an interior $U(g)$ depending
only on these points and contained in $P$.  We shall denote $U(g)$ by
$U(\seq x, \seq y)$.

\begin{lemma} \label{triangle} Let $\seq x, \seq y, \seq z$ be three
different points in a piece $P\subseteq \AM$.  Then $$U(\seq x, \seq
y)=U(\seq y, \seq z)=U(\seq x, \seq z).$$
\end{lemma}

\proof Let $(x_n)^\omega, (y_n)^\omega, (z_n)^\omega$ be
representatives of $\seq x, \seq y, \seq z$.  Choose hierarchy paths
$[x_n, y_n], [y_n, z_n], [x_n, z_n]$.  By Theorem
\ref{distanceformula}, the hierarchy path $[x_n, y_n]$ shares a large
domain with either $[x_n, z_n]$ or $[y_n, z_n]$ for all $n$
$\omega$-a.s. By Lemma \ref{interior}, then $U(\seq x, \seq y)$
coincides either with $U(\seq x, \seq z)$ or with $U(\seq y, \seq z)$.
Repeating the argument with $[\seq x, \seq y]$ replaced either by
$[\seq y, \seq z]$ or by $[\seq x,\seq z]$, we conclude that all three
interiors coincide.  \endproof

\begin{proposition}\label{unique} Every piece $P$ of the asymptotic cone
$\AM$ contains a unique interior $U(g)$, and $P$ is the closure of
$U(g)$.
\end{proposition}

\proof Let $U(g)$ and $U(g')$ be two interiors inside $P$.  Let $\seq
x, \seq y$ be two distinct points in $U(g)$, $\seq z, \seq t$ be two
distinct points in $U(g')$.  If $\seq y\neq \seq z$ we apply Lemma
\ref{triangle} to the triples $(\seq x, \seq y, \seq z)$ and $(\seq y,
\seq z, \seq t)$, and conclude that $U(g)=U(g')$.  If $\seq y = \seq
z$ we apply Lemma \ref{triangle} to the triple $(\seq x, \seq y, \seq
t)$.  The fact that the closure of $U(g)$ is $P$ follows from Lemma
\ref{interior}.  \endproof

\subsection{Actions and splittings}\label{sec:actionssplittings}

We recall a theorem proved by V. Guirardel in
\cite{Guirardel:trees}, that we will use in the sequel.

\begin{defn} \label{defgui}
 The \emph{height} of an arc in an $\R$--tree with respect to the
 action of a group $G$ on it is the maximal length of a decreasing
 chain of sub-arcs with distinct stabilizers.  If the height of an arc
 is zero then it follows that all sub-arcs of it have the same
 stabilizer.  In this case the arc is called \textit{stable}.

  The tree $T$ is \textit{of finite height} with respect to the action
  of some group $G$ if any arc of it can be covered by finitely many
  arcs with finite height.  If the action is minimal and $G$ is
  finitely generated then this condition is equivalent to the fact
  that there exists a finite collection of arcs $\mathcal{I}$ of
  finite height such that any arc is covered by finitely many
  translates of arcs in $\mathcal{I}$ \cite{Guirardel:trees}.
\end{defn}

\begin{theorem}[{Guirardel} \cite{Guirardel:trees}]\label{gui}
Let $\Lambda$ be a finitely generated group and let $T$ be a real tree
on which $\Lambda$ acts minimally and with finite height.  Suppose
that the stabilizer of any non-stable arc in $T$ is finitely
generated.

Then one of the following three situations occurs:

\begin{itemize}
    \item[(1)] $\Lambda$ splits over the stabilizer of a non-stable
    arc or over the stabilizer of a tripod;

    \item[(2)] $\Lambda$ splits over a virtually cyclic extension of
    the stabilizer of a stable arc;

\item[(3)] $T$ is a line and $\Lambda$ has a subgroup of index at most
2 that is the extension of the kernel of that action by a finitely
generated free abelian group.
\end{itemize}
\end{theorem}

  \medskip

In some cases stability and finite height follow from the algebraic
structure of stabilizers of arcs, as the next lemma shows.

\begin{lemma}(\cite{DrutuSapir:splitting})\label{stable} Let $G$ be a
finitely generated group acting on an $\R$--tree $T$ with finite of
size at most $D$ tripod stabilizers, and (finite of size at most
$D$)-by-abelian arc stabilizers, for some constant $D$.  Then
\begin{itemize}
    \item[(1)] an arc with stabilizer of size $>(D+1)!$ is stable;
    \item[(2)] every arc of $T$ is of finite height (and so the action
    is of finite height and stable).
\end{itemize}
\end{lemma}


We also recall the following two well known results due to Bestvina
(\cite{Bestvina:degener}, \cite{Bestvina:survoltrees}) and Paulin
\cite{Paulin:arbres}.

\begin{lemma}\label{Pau} Let $\Lambda$ and $G$ be two finitely
generated groups, let $A=A\iv$ be a finite set generating $\Lambda$
and let $\dist$ be a word metric on $G$.  Given
$\phi_n\co\Lambda\to G$ an infinite sequence of homomorphisms, one
can associate to it a sequence of positive integers defined by
\begin{equation}\label{dn}
d_n=\inf_{x\in G}\sup_{a\in A}\dist(\phi_n(a)x,x)\, .
\end{equation}

If $(\phi_n)$ are pairwise non-conjugate in $\Gamma$ then $\lim_{n\to
\infty}d_n=\infty $.
\end{lemma}

\begin{remark}\label{xa}
For every $n\in \N$, $d_n=\dist(\phi_n(a_n)x_n,x_n)$ for some $x_n\in
\Gamma$ and $a_n\in A$.
\end{remark}

Consider an arbitrary ultrafilter $\omega$.  According to Remark
\ref{xa}, there exists $a\in A$ and $x_n\in G$ such that
$d_n=\dist(\phi_n(a)x_n,x_n)$ $\omega$--a.s.

\begin{lemma} \label{Pau1}
Under the assumptions of Lemma \ref{Pau}, the group $\Lambda$ acts on
the asymptotic cone $\ck_\omega=\ko{G; (x_n), (d_n)}$ by isometries,
without a global fixed point, as follows:
\begin{equation}\label{action}
  g\cdot \lio{x_n}=\lio{\phi_n(g)x_n}\, .
\end{equation}

This defines a homomorphism $\phi_\omega$ from $\Lambda$ to the group
$x^\omega (\Pi_1 \Gamma/\omega) (x^\omega)\iv$ of isometries
of~$\ck_\omega$.
\end{lemma}

 Let $S$ be a surface of complexity $\xi(S)$.  When $\xi(S)\leq 1$ the
 mapping class group $\MCG (S)$ is hyperbolic and the well-known
 theory on homomorphisms into hyperbolic groups can be applied
(see for instance \cite{Bestvina:degener}, \cite{Paulin:arbres}, \cite{Bestvina:survoltrees} and references therein). Therefore we adopt the following convention for the rest of this
 section.

\begin{cvn}\label{cxi}
In what follows we assume that $\xi (S) \geq 2$.
\end{cvn}

\begin{proposition}\label{i}
Suppose that a finitely generated group $\Lambda =\la A\ra$ has
infinitely many homomorphisms $\phi_n$ into a mapping class group
$\MCG(S)$, which are pairwise non-conjugate in $\MCG(S)$.  Let

\begin{equation}\label{dnp}
d_n=\inf_{\mu \in \MM (S) }\sup_{a\in A}\dist(\phi_n(a)\mu ,\mu )\, ,
\end{equation}

and let $\mu_n$ be the point in $\MM (S)$ where the above infimum is
attained.

Then one of the following two situations occurs:
\begin{enumerate}
    \item\label{acs} either the sequence $(\phi_n)$ defines a
    non-trivial action of $\Lambda$ on an asymptotic cone of the
    complex of curves $\ko{\CC (S); (\gamma_n), (\ell_n)}$,
    \item\label{ams} or the action by isometries, without a global
    fixed point, of $\Lambda $ on the asymptotic cone $\ko{\MM (S) ;
    (\mu_n), (d_n)}$ defined as in Lemma \ref{Pau1} fixes a piece
    set-wise.
\end{enumerate}
\end{proposition}

\proof Let $\ell_n=\inf_{\gamma \in \CC(S)}\sup_{a\in A}\dcs
(\phi_n(a)\gamma ,\gamma )$.  As before, there exists $b_0\in A$ and
$\gamma_n\in \CC (S)$ such that $\ell_n=\dcs
(\phi_n(b_0)\gamma_n,\gamma_n)$ $\omega$--a.s.

If $\lim_\omega \ell_n = +\infty$ then the sequence $(\phi_n)$ defines
a non-trivial action of $\Lambda$ on $\ko{\CC (S); (\gamma_n),
(\ell_n)}$.

Assume now that there exists $M$ such that for every $b\in A$, $\dcs
(\phi_n(b)\gamma_n,\gamma_n)\leq M$ \uas.  This implies that for every
$g\in \Lambda$ there exists $M_g$ such that $\dcs
(\phi_n(g)\gamma_n,\gamma_n)\leq M_g$ \uas.

Consider $\mu_n'$ the projection of $\mu_n$ onto $\QQ(\gamma_n )$.  A
hierarchy path $[\mu_n ,\mu_n']$ shadows a tight geodesic $\g_n$
joining a curve in $\base (\mu_n )$ to a curve in $\base (\mu_n' )$,
the latter curve being at $C(S)$-distance 1 from $\gamma_n$.  If the
$\omega$-limit of the $C(S)$-distance from $\mu_n$ to $\mu_n'$ is
finite then in follows that for every $b\in A$, $\dcs
(\phi_n(b)\mu_n,\mu_n)=O(1)$ \uas.  Then the action of $\Lambda$ on
$\ko{\MM (S) ; (\mu_n), (d_n)}$ defined by the sequence $(\phi_n)$
preserves $U((\mu_n)^\omega)$, which is the interior of a piece, hence
it fixes a piece set-wise.  Therefore in what follows we assume that
the $\omega$-limit of the $C(S)$-distance from $\mu_n$ to $\mu_n'$ is
infinite.

Let $b$ be an arbitrary element in the set of generators $A$.
Consider a hierarchy path $[\mu_n ,\phi_n(b)\mu_n]$.  Consider the
Gromov product
$$
\tau_n(b) = \frac{1}{2} \left[\dcs (\mu_n,\mu_n') + \dcs (\mu_n,
\phi_n(b)\mu_n )-\dcs (\mu_n', \phi_n(b)\mu_n)\right]\, , $$ and
$\tau_n = \max_{b\in A}\tau_n(b)$.

The geometry of quadrangles in hyperbolic geodesic spaces combined
with the fact that $\dcs (\mu_n', \phi_n(b)\mu_n' )\leq M$ implies
that:
\begin{itemize}
    \item every element $\nu_n$ on $[\mu_n , \mu_n']$ which is at
    $C(S)$ distance at least $\tau_n(b)$ from $\mu_n$ is at
    $C(S)$-distance $O(1)$ from an element $\nu_n'$ on
    $[\phi_n(b)\mu_n , \phi_n(b)\mu_n']\, $; it follows that $\dcs
    (\nu_n', \phi_n(b)\mu_n')= \dcs (\nu_n, \mu_n')+O(1)$, therefore $\dcs
    (\nu_n', \phi_n(b)\nu_n)=O(1)$ and $\dcs (\nu_n,
    \phi_n(b)\nu_n)=O(1)$;

\medskip

 \item the element $\rho_n(b)$ which is at $C(S)$ distance $\tau_n(b)$
 from $\mu_n$ is at $C(S)$-distance $O(1)$ also from an element
 $\rho_n''(b)$ on $[\mu_n , \phi_n(b)\mu_n]$.
\end{itemize}

We have thus obtained that for every element $\nu_n$ on $[\mu_n ,
\mu_n']$ which is at $C(S)$ distance at least $\tau_n$ from $\mu_n$,
$\dcs (\nu_n, \phi_n(b)\nu_n)=O(1)$ for every $b\in B$ and
$\omega$-almost every $n$.  In particular this holds for the point
$\rho_n$ on $[\mu_n , \mu_n']$ which is at $C(S)$ distance $\tau_n$
from $\mu_n$.  Let $a\in A$ be such that $\tau_n(a)=\tau_n$ and
$\rho_n=\rho_n(a)$, and let $\rho_n''= \rho_n''(a)$ be the point on
$[\mu_n , \phi_n(a)\mu_n]$ at $C(S)$-distance $O(1)$ from $\rho_n(a)$.
It follows that $\dcs (\rho_n'', \phi_n(b)\rho_n'')=O(1)$ for every
$b\in B$ and $\omega$-almost every $n$.  Moreover, since $\rho''$ is a
point on $[\mu_n , \phi_n(a)\mu_n]$, its limit is a point in $\ko{\MM
(S) ; (\mu_n), (d_n)}$.  It follows that the action of $\Lambda$ on
$\ko{\MM (S) ; (\mu_n), (d_n)}$ defined by the sequence $(\phi_n)$
preserves $U((\rho_n'')^\omega)$, which is the interior of a piece,
hence it fixes a piece set-wise.  \endproof

\begin{lemma}\label{stab2cs}
Let $\seq \gamma$ and $\seq \gamma'$ be two distinct points in an
asymptotic cone of the complex of curves $\ko{\CC (S); (\gamma_n),
(d_n)}$.  The stabilizer $\stab (\seq \gamma , \seq \gamma')$ in
the ultrapower $\MCG (S)^\omega$ is the extension of a finite subgroup
of cardinality at most $N=N(S)$ by an abelian group.
\end{lemma}

\proof Let $\q_n$ be a geodesic joining $\gamma_n$ and $\gamma_n'$ and let $x_n, y_n$ be points
at distance $\epsilon d_n$ from  $\gamma_n$ and $\gamma_n'$ respectively, where $\epsilon >0$ is small enough.

Let $\seq g=\ultra g$ be an element in $\stab (\seq\gamma , \seq\gamma' )$.
Then $$\delta_n(\seq g) =\max (\dist_{\CC (S)} (\gamma_n , g_n \gamma_n)\, ,\, \dist
(\gamma_n' , g_n \gamma_n'))$$ satisfies $\delta_n(\seq g) = o(d_n)$.

Since $\CC (S)$ is a Gromov hyperbolic space
it follows that the sub-geodesic of $\q_n$ with endpoints $x_n, y_n$
is contained in a finite radius tubular neighborhood
of $g_n \q_n$.  Since $x_n$ is \uass at distance
$O(1)$ from a point $x_n'$ on $g_n q_n$, define $\ell_x (g_n)$ as
$(-1)^\epsilon \dcs (x_n , g_n x_n) $, where $\epsilon =0$ if $x_n'$
is nearer to $g_n \mu_n$ than $g_n x_n $ and $\epsilon =1$ otherwise.

Let $\ell_x \co \stab (\mu , \nu ) \to \Pi \R /\omega $ defined by
$\ell_x (\seq g)= \left(\ell_x (g_n) \right)^\omega$.  It is easy to
see that $\ell_x$ is a quasi-morphism, that is
\begin{equation}\label{qm}
\left| \ell_x (\seq g\seq h)- \ell_x (\seq g)-\ell_x (\seq h) \right|
\leq _{\omega} O(1)\, .
\end{equation}

It follows that $\left| \ell_x \left( [\seq g,\seq h] \right) \right|
\leq _{\omega} O(1) \, .$

The above and a similar argument for $y_n$ imply that for every
commutator, $\seq c = \seqrep c$, in the stabilizer of $\seq \mu$ and
$\seq\nu$, $\dcs (x_n , c_n x_n)$ and $\dcs (y_n , c_n y_n)$ are at
most $O(1)$.  Lemma \ref{D} together with Bowditch's
acylindricity result \cite[Theorem 1.3]{Bowditch:tightgeod}
imply that the set
of commutators of $\stab (\seq\mu, \seq\nu)$ has uniformly bounded
cardinality, say, $N$.  Then any finitely generated subgroup $G$ of
$\stab (\seq\mu, \seq\nu)$ has conjugacy classes of cardinality at
most $N$, i.e. $G$ is an $FC$-group \cite{Neumann:finiteconj}.  By
\cite{Neumann:finiteconj}, the set of all torsion elements of $G$ is
finite, and the derived subgroup of $G$ is finite of cardinality $\le
N(S)$ (by Lemma \ref{fs}).

\endproof

\begin{lemma}\label{stab3}
Let $\sa, \sb , \sc$ be the vertices of a non-trivial tripod in an
asymptotic cone of the complex of curves $\ko{\CC (S); (\gamma_n),
(d_n)}$.  The stabilizer $\stab (\sa , \sb , \sc )$ in
the ultrapower $\MCG (S)^\omega$ is a finite subgroup
of cardinality at most $N=N(S)$.
\end{lemma}

\proof Since $\CC(S)$ is $\delta$-hyperbolic, for every $a>0$ there exists
$b>0$ such that for any triple of points $x,y,z\in\CC(S)$ the intersection of the three
$a$-tubular neighborhoods of geodesics $[x,y]$, $[y,z]$, and
$[z,a]$ is a set $C_{a}(x,y,z)$ of diameter at most $b$.

Let $\seq g=\ultra g$ be an element in $\stab (\sa , \sb , \sc )$.
Then $$\delta_n(\seq g) =\max (\dist_{\CC (S)} (\alpha_n , g_n
\alpha_n)\, ,\, \dist (\beta_n, g_n \beta_{n})\, ,\, \dist
(\gamma_n , g_n \gamma_n))$$ satisfies $\delta_n(\seq g) = o(d_n)$.
While, the distance between each pair of points among $\alpha_{n},\beta_{n},$ and $\gamma_{n}$ is at least $ \lambda d_{n}$ for some $\lambda >0$. It
follows that if $(x_n,y_n)$ is any of the pairs $(\alpha_{n},\beta_{n})$,
$(\alpha_{n},\gamma_{n})$, $(\gamma_{n},\beta_{n})$, then away
from a $o(d_n)$-neighborhood of the endpoints the two
geodesics $[x_n,y_n]$ and $[g_{n}x_n,g_{n}y_n]$ are uniformly Hausdorff close. This in particular implies that away
from a $o(d_n)$-neighborhood of the endpoints, the $a$-tubular neighborhood of $[g_{n}x_n,g_{n}y_n]$ is contained in the $A$-tubular neighborhood of $[x_n, y_n]$ for some $A>a$. Since $\sa , \sb ,\sc$ are the vertices of a non-trivial tripod, for any $a>0$, $C_{a}(\alpha_{n},\beta_{n}, \gamma_{n})$ is \uass disjoint of  $o(d_n)$-neighborhoods of $\alpha_{n},\beta_{n}, \gamma_{n}$. The same holds for  $\seq g \sa , \seq g \sb , \seq g \sc$. It follows that $C_{a}(g_{n}\alpha_{n},g_{n}\beta_{n}, g_{n}\gamma_{n})$ is contained in $C_{A}(\alpha_{n},\beta_{n}, \gamma_{n})$, hence it is at Hausdorff distance at most $B>0$ from $C_{a}(\alpha_{n},\beta_{n}, \gamma_{n})$.

Thus we may find a point $\tau_n \in [\alpha_{n},\beta_{n}]$ such that $\dist (\tau_n , g_n \tau_n ) =O(1)$, while the distance from $\tau_n$ to $\{ \alpha_{n},\beta_{n} \}$
 is at least $2\epsilon d_n$. Let $\eta_n\in [\tau_n , \alpha_n ]$ be a point at distance $\epsilon d_n$ from $\tau_n$. Then  $g_n\eta_n\in [g_n\tau_n , g_n\alpha_n ]$ is a point at distance $\epsilon d_n$ from $g\tau_n$. On the other hand, since
 $\eta_n$ is at distance at least $\epsilon d_n$ from $\alpha_n$ it follows that there exists $\eta_n'\in [g_n\tau_{n},g_n\alpha_{n}]$ at distance $O(1)$ from $\eta_n$. It follows that $\eta_n'$ is at distance $\epsilon d_n +O(1)$ from $g_n \tau_n$, hence $\eta_n'$ is at distance $O(1)$ from $g_n \eta_n$. Thus we obtained that $g_n \eta_n$ is at distance $O(1)$ from $\eta_n$. This, the fact that $g_n \tau_n$ is at distance $O(1)$ from $\tau_n$ as well, and the fact that $\dist (\tau_n , \eta_n) = \epsilon d_n$, together with Bowditch's
acylindricity result \cite[Theorem 1.3]{Bowditch:tightgeod} and   Lemma \ref{D}
imply that the stabilizer $\stab (\sa , \sb , \sc )$ has uniformly bounded
cardinality. \endproof

\begin{corollary}\label{ii}
Suppose that a finitely generated group $\Lambda =\la A\ra$ has
infinitely many injective homomorphisms $\phi_n$ into a mapping class
group $\MCG(S)$, which are pairwise non-conjugate in $\MCG(S)$.

Then one of the following two situations occurs:
\begin{enumerate}
    \item\label{acs2} $\Lambda$ is virtually abelian, or it splits
    over a virtually abelian subgroup; \item\label{ams2} the action by
    isometries, without a global fixed point, of $\Lambda $ on the
    asymptotic cone $\ko{\MCG (S) ; (\mu_n), (d_n)}$ defined as in
    Lemma \ref{Pau1} fixes a piece set-wise.
\end{enumerate}
\end{corollary}

\proof It suffices to prove that case (\ref{acs}) from Proposition
\ref{i} implies (\ref{acs2}) from Corollary \ref{i}.  According to
case (\ref{acs}) from Proposition \ref{i} the group $\Lambda$ acts
non-trivially on a real tree, by Lemma~\ref{stab2cs} we know that
the stabilizers of non-trivial arcs are virtually abelian,
 and by Lemma~\ref{stab3} we know that
the stabilizers of non-trivial tripods are finite.

On the other hand, since $\Lambda$ injects into $\MCG (S)$, it
follows immediately from results of
Birman--Lubotzky--McCarthy \cite{BirmanLubotzkyMcCarthy}
that virtually abelian subgroups in
$\Lambda$ satisfy the ascending chain condition, and are always
finitely generated.  By Theorem \ref{gui} we thus have that
$\Lambda$ is either virtually abelian or it splits over
a virtually solvable subgroup. One of the main theorems of
\cite{BirmanLubotzkyMcCarthy} is that any virtual solvable subgroup of
the mapping class group is virtually abelian, finishing the argument.
\endproof

In fact the proof of Corollary \ref{ii} allows one to remove the injectivity assumption by replacing $\Lambda$ by its natural image in $\mcgb$ as follows:

\begin{corollary}\label{iit-1} Suppose that a finitely generated group $\Lambda =\la A\ra$ has
infinitely many homomorphisms $\phi_n$ into a mapping class
group $\MCG(S)$, which are pairwise non-conjugate in $\MCG(S)$. Let $N$ be the intersections of kernels of these homomorphisms

Then one of the following two situations occurs:
\begin{enumerate}
    \item\label{acs2.1} $\Lambda/N$ is virtually abelian, or it splits
    over a virtually abelian subgroup;

    \item\label{ams2.1} the action by
    isometries, without a global fixed point, of $\Lambda/N$ on the
    asymptotic cone $\ko{\MCG (S) ; (\mu_n), (d_n)}$ defined as in
    Lemma \ref{Pau1} fixes a piece set-wise.
\end{enumerate}

\end{corollary}

Corollary \ref{iit-1} immediately implies the following statement because every quotient of a group with property (T) has property (T):

\begin{corollary}\label{iit}
Suppose that a finitely generated group $\Lambda =\la A\ra$ with
property~(T) has infinitely many injective homomorphisms $\phi_n$ into
a mapping class group $\MCG(S)$, which are pairwise non-conjugate in
$\MCG(S)$.

Then the action by isometries, without a global fixed point, of
$\Lambda $ on the asymptotic cone $\ko{\MCG (S) ; (\mu_n), (d_n)}$
defined as in Lemma \ref{Pau1} fixes a piece set-wise.
\end{corollary}


\section{Subgroups with property (T)}\label{sect:T}

Property (T) can be characterized in terms of isometric actions on median spaces:

\begin{theorem}(\cite{CDH-T}, Theorem 1.2)\label{eqmw}
A locally compact, second countable group has property (T) if and only if
any continuous action by isometries of that group on any metric median space has bounded orbits.
\end{theorem}

In our argument, the direct implication of Theorem \ref{eqmw} plays the most important part. This implication for
discrete countable groups first appeared in \cite{Nica:preprint} (see also
\cite{NibloReeves:groupsactingCATO}, \cite{NibloRoller:cubesandT} for
implicit proofs, and \cite{Roller:median} for median algebras). The same
direct implication, for locally compact second countable groups, follows
directly by
combining \cite[Theorem 20, Chapter 5]{delaHarpeValette:proprieteT} with
\cite[Theorem V.2.4]{Verheul:book}.

\medskip

Our main result in this section is the following.

\begin{theorem}\label{thmT}
Let $\Lambda$ be a finitely generated group  and let $S$ be a surface.

If there exists an infinite collection $\Phi$ of homomorphisms $\phi\co \Lambda \to \MCG (S)$ pairwise non-conjugate in
$\MCG(S)$, then $\Lambda$ acts on an asymptotic cone of $\MCG(S)$
with unbounded orbits.
\end{theorem}

\begin{corollary}\label{cor:thmT}
If $\Lambda$ is a finitely generated group with Kazhdan's property (T) and $S$ is a surface then the set of homomorphisms of $\Lambda$ into $\MCG(S)$ up to conjugation in $\MCG(S)$ is finite.
\end{corollary}

\begin{remarks}
\begin{enumerate}
  \item Theorems \ref{thmT} and \ref{thmedian} imply the following: for a finitely generated group $\Lambda$ such that any action on a median space has bounded orbits, the space of homomorphisms from $\Lambda$ to $\MCG(S)$ is finite modulo conjugation in $\MCG(S)$. Due to Theorem \ref{eqmw}, the above hypothesis on $\Lambda$ is equivalent to property (T).
  \item Corollary \ref{cor:thmT} suggests that $\MCG(S)$ contains few subgroups with property~(T).
\end{enumerate}
\end{remarks}

In order to prove Theorem \ref{thmT} we need two easy lemmas and a proposition. Before formulating them we introduce some notation and terminology.

Since the group $\MCG(S)$ acts co-compactly on $\MM(S)$ there exists a compact subset $K$
of $\MM (S)$ containing the basepoint $\nu_0$ fixed in Section \ref{sec:conedistformula} and such that $\MCG(S) K= \MM
(S)$.

Consider an asymptotic cone $\AM=\coneMmu$. According to the above there exists $x_n\in \MCG(S)$ such that $x_n K$ contains $\mu_n^0$. We denote by $x^\omega$ the element $(x_n)^\omega$ in the ultrapower of $\MCG (S)$. According to the Remark \ref{grascone}, (2), the subgroup $x^\omega \left(\Pi_1 \MCG (S)/\omega \right) \left( x^\omega \right)\iv$ of the ultrapower of $\MCG (S)$ acts transitively by isometries on $\AM$.  The action is isometric both with
respect to the metric $\dist_\AM$ and with respect to the metric $\td$.

\begin{notation}\label{bdedpart}
We denote for simplicity the subgroup $x^\omega \left(\Pi_1 \MCG (S)/\omega \right) \left( x^\omega \right)\iv$ by $\MCG (S)_e^\omega$.
\end{notation}

We say that an element
$g=(g_n)^\omega$ in $\mcgb$ has a given property
(e.g. it is pseudo-Anosov, pure, reducible, etc.) if and only if \uass
$g_n$ has that property (it is pseudo-Anosov, resp.\ pure, resp.\ reducible, etc.).

 \begin{lemma}\label{fixedpair}
 Let $g$ be an element in $\mcgb$ fixing two distinct points $\seq\mu ,\, \seq\nu$ in $\AM$.
 Then for every subsurface $\bu \in \fu ( \seq\mu ,\, \seq\nu)$
 there exists $k\in \N , \, k\geq 1 ,$ such that $g^k \bu = \bu$.

In the particular case when $g$ is pure, $k$ may be taken equal to 1.
 \end{lemma}

\proof If  $\tdlu (\seq\mu ,\, \seq\nu)=d >0$ then for every $i\in \N , \, i\geq 1 ,$
$\td_{g^i \bu} (\seq\mu ,\, \seq\nu)=d$. Then there exists $k\geq 1$, $k$ smaller than
$\left[\frac{\td (\seq\mu ,\, \seq\nu)}{d} \right]+1$ such that $g^k \bu = \bu$.

The latter part of the statement follows from the fact that if $g$ is pure any power of it fixes exactly the same subsurfaces as $g$ itself.\endproof

\medskip

\begin{lemma}\label{2ptsinpiece}
 Let $\seq\mu , \seq\nu$ be two distinct points in the same piece. Then $\bs \not\in \fu (\seq\mu , \seq\nu)$.
\end{lemma}

\proof Assume on the contrary that $\td_{\bs }(\seq\mu , \seq\nu )=4\epsilon >0$. By Proposition \ref{unique} there exist $\seq\mu , \seq\nu$ in $U(P)$ such that $\td (\seq\mu , \seq\mu') < \epsilon$ and $\td (\seq\nu , \seq\nu ') < \epsilon$. Then $\td_{\bs }(\seq\mu' , \seq\nu' )>2\epsilon >0$ whence $\lio{\dist_{C(S)} (\mu_n' , \nu_n' ) }=+\infty$, contradicting the fact that $\seq\mu , \seq\nu$ are in $U(P)$.\endproof

We now state the result that constitutes the main ingredient in the
proof of Theorem \ref{thmT}.  Its proof will be postponed until after
the proof of Theorem \ref{thmT} is completed.

\begin{proposition}\label{nredgen}
Let $g_1=(g^1_n)^\omega$,...,$g_m=(g_n^m)^\omega$ be pure elements in $\mcgb$, such that $\la g_1,...,g_m \ra$ is composed only of pure elements and its orbits in $\AM$ are bounded. Then $g_1,...,g_m$ fix a point in $\AM$.
\end{proposition}

\noindent \textit{Proof of Theorem \ref{thmT}.} \quad  Assume there exists an infinite collection $\Phi=\{\phi_1,\phi_2,...\}$ of
pairwise non-conjugate homomorphisms $\phi_n\colon \Lambda\to\MCG(S)$. Lemma \ref{Pau} implies that
given a finite generating set $A$ of $\Lambda$, $\lim_{n\to \infty}d_n=\infty $, where

\begin{equation}\label{dnmcg}
d_n=\inf_{\mu \in \MM (S) }\sup_{a\in A}\dist(\phi_n(a)\mu ,\mu )\, .
\end{equation}

Since $\MM (S)$ is a simplicial complex there exists a vertex
$\mu_n^0\in \MM (S)$ such that $d_n=\sup_{a\in
A}\dist(\phi_n(a)\mu_n^0 ,\mu_n^0 )\, .$  Consider an arbitrary ultrafilter $\omega$ and the asymptotic cone $\AM=\coneMmu$. We use the notation that appears before Lemma \ref{fixedpair}.

The infinite sequence of homomorphisms
$(\phi_n)$ defines a homomorphism
$$
\fo \co \Lambda \to \mcgb \, ,\, \fo (g)= (\phi_n (g))^\omega\, .
$$

This homomorphism defines an isometric action of $\Lambda$ on $\AM$ without global fixed point.

\medskip

\Notat \quad When there is no possibility
of confusion, we write $g\seq\mu $ instead of $\fo (g)\seq\mu$, for $g\in \Lambda$ and $\seq\mu$
in $\AM$.

\medskip

We prove by induction on the complexity of $S$ that if $\Lambda$ has bounded orbits in $\AM$ then $\Lambda$ has a global fixed point. When $\xi (S) \leq 1$ the asymptotic cone $\AM$ is a complete real tree and the previous statement is known to hold. Assume that we proved the result for surfaces with complexity at most $k$, and assume that $\xi (S )=k+1$.

If $\Lambda$ does not fix set-wise a piece in the (most refined) tree-graded structure of $\AM$, then by Lemma \ref{gpisompiece} $\Lambda$ has a global fixed point. Thus we may assume that $\Lambda$ fixes set-wise a piece $P$ in
$\AM$. Then $\Lambda$ fixes the interior $U(P)$ of $P$ as well by Proposition~\ref{unique}.

The point $\seq{\mu}^0 = \seqrep{\mu^0}$ must be in the piece $P$. If not, the projection $\seq\nu$ of
$\seq{\mu}^0$ on $P$ would be moved by less that $1$ by all $a\in A$. Indeed, if $a\seq\nu \neq
\seq\nu$ then the concatenation of geodesics $[\seq{\mu}^0 , \seq\nu] \sqcup [\seq\nu , a\seq\nu]\sqcup
[a\seq\nu , a\seq{\mu}^0]$ is a geodesic according to \cite{DrutuSapir:TreeGraded}.

According to Lemma \ref{fs} there exists a normal subgroup $\Lambda_p$
in $\Lambda$ of index at most $N=N(S)$ such that for every $g\in
\Lambda_p$, $\fo (g)$ is pure and fixes set-wise the boundary
components of $S$.

By Proposition \ref{nredgen}, $\Lambda_p$ fixes a point $\seq\alpha$ in $\AM$. Since it also fixes set-wise the piece $P$, it fixes the unique projection of $\seq\alpha$ to $P$. Denote this projection by $\seq\mu$.

Assume that $\Lambda$ acts on $\AM$ without fixed point. It follows that there exists $g\in \Lambda$ such that
  $g\seq\mu\neq \seq\mu$. Then $\Lambda_p = g\Lambda_pg\iv$ also fixes $g\seq\mu$. Lemma \ref{fixedpair}
   implies that $\Lambda_p$ fixes a subsurface $\bu \in \fu (\seq \mu, g\seq\mu)$. Since $\seq \mu, g\seq\mu$ are in the piece $P$, it follows that $\bu$ is a proper subsurface of $\mathbf{S}$, by Lemma \ref{2ptsinpiece}. Thus $\Lambda_p$ must fix a multicurve $\partial \bu$.

  Let $\seq\Delta$ be a maximal multicurve fixed by $\Lambda_p$. Assume there exists $g\in \Lambda$ such that
  $g\seq\Delta\neq \seq\Delta$.  Then $\Lambda_p = g\Lambda_pg\iv$
  also fixes $g\seq\Delta$, contradicting the maximality of
  $\seq\Delta$.  We then conclude that all $\Lambda$ fixes
  $\seq\Delta$.  It follows that the image of $\fo$ is in $\Stab
  (\seq\Delta)$, hence \uass $\phi_n(\Lambda )\subset \Stab(\Delta_n
  )$.  Up to taking a subsequence and conjugating we may assume that
  $\phi_n(\Lambda )\subset \Stab (\Delta )$ for some fixed multicurve
  $\Delta$.  Let $U_1,...,U_m$ be the subsurfaces and annuli
  determined by $\Delta$.
  Hence, we can see $\phi_n$ as isomorphisms with target
  $\Stab(\Delta)$ which are pairwise non-conjugate by hypothesis
  and thus define an isometric action with bounded orbits on the asymptotic
  cone of $\Stab(\Delta)$. Recall that $\Stab(\Delta)$ is
  quasi-isometric to
  $\MCG (U_1)\times \cdots \times \MCG(U_m)$ by
  Remark~\ref{rmk:qqproduct}.
  Hence, the sequence $\phi_n$  defines an isometric action with bounded orbits on the spaces $\AM (U_j)$.
  By hypothesis at least one of these actions is
  without a global fixed point, while the inductive hypothesis implies
  that each of these actions has a global fixed point, a
  contradiction.  \hspace*{\fill}$\Box$

\me

To prove Proposition \ref{nredgen}, we first provide
a series of intermediate results.

\begin{lemma}\label{red}
Let  $g=(g_n)^\omega\in \mcgb$ be a reducible element in $\AM$, and let $\seq\Delta =(\Delta_n)^\omega$ be a multicurve
 such that if $U^1_n,...,U^m_n$ are the connected components  of
$S\setminus \Delta_n$ and the annuli with core curve in $\Delta_n$ then \uass $g_n$ is a
pseudo-Anosov on $U^1_n,...,U^k_n$ and the identity map on $U^{k+1}_n,..., U^m_n$, and $\Delta_n=
\partial U^1_n\cup...\cup \partial U^k_n$ (the latter condition may be achieved by deleting the boundary between two components
 onto which $g_n$ acts as identity).

 Then the limit set  $Q(\seq\Delta )$ appears in the asymptotic cone (i.e. the distance from the basepoint $\mu_n^0$ to $Q(\Delta_n)$ is $O(d_n)$), in particular $g$ fixes the piece containing  $Q(\seq\Delta )$.

 If $g$ fixes a piece $P$ then $U(P)$ contains  $Q(\seq\Delta )$.
\end{lemma}

\proof Consider a point $\seq\mu=\seqrep \mu $ in $\AM$. Let $D_n = \dist_{\MM (S)} (\mu_n , \Delta_n )$.
 Assume that $\lim_\omega \frac{D_n}{d_n} =+\infty $. Let $\nu_n$ be a projection of $\mu_n$ onto
 $Q(\Delta_n)$. Note that for every $i=1,2,...,k$, $\dist_{C(U^i_n)} (\mu_n, g_n\mu_n ) =
 \dist_{C(U^i_n)} (\nu_n, g_n \nu_n )+O(1)$. Therefore when replacing $g$ by some large enough power of it
  we may ensure that $\dist_{C(U^i_n)} (\mu_n, g_n \mu_n ) >M$, where $M$ is the constant from
  Lemma \ref{MM2:LLL}, while we still have that $\dist_{\MM (S)} (\mu_n , g_n\mu_n )\leq C d_n$.
  In the cone $\mathrm{Con}^\omega(\MM (S); (\mu_n), (D_n))$ we have that $\seq\mu=g\seq\mu$
  projects onto $Q(\seq\Delta )$ into $\seq\nu=g\seq\nu$, which is at distance $1$. This contradicts
  Lemma \ref{lem:betw}. It follows that $\lim_\omega \frac{D_n}{d_n} < +\infty $.

  Assume now that $g$ fixes a piece $P$ and assume that the point $\seq\mu$ considered above is in $U(P)$. Since the
  previous argument implies that a hierarchy path joining $\seq\mu$ and $g^k\seq\mu$ for some large
  enough $k$ intersects $Q(\partial \bu_i)$, where $\bu_i=(U^i_n)^\omega$ and $i=1,2,...,k$, and $Q(\seq\Delta )\subset Q(\partial
  \bu_i)$, it follows that $Q(\seq\Delta )\subset U(P)$.\endproof

\medskip

\Notat \quad Given two points  $\seq\mu ,\, \seq\nu$ in $\AM$ we denote by $\fu ( \seq\mu ,\, \seq\nu)$ the set of subsurfaces $\bu \subseteq \bs$  such that $\tdlu (\seq\mu ,\, \seq\nu) >0$. Note that $\fu ( \seq\mu ,\, \seq\nu)$ is non-empty if and only if  $\seq\mu\neq \seq\nu$.

\medskip

 \begin{lemma}\label{uqd}
 Let $\seq\Delta$ be a multicurve.
\begin{enumerate}
   \item\label{inq} If $\seq\mu , \seq\nu$ are two points in $Q(\seq\Delta)$ then any $\bu \in \fu (\seq\mu , \seq\nu)$ has the property that $\bu \notpitchfork \seq\Delta$.
   \item \label{orthq} If $\seq\mu$ is a point outside $Q(\seq\Delta)$ and $ \seq\mu'$ is the projection of $ \seq\mu$ onto $Q(\seq\Delta)$ then any $\bu \in \fu (\seq\mu , \seq\mu')$ has the property that $\bu \pitchfork \seq\Delta$.
   \item\label{outq}  Let $\seq\mu$ and $ \seq\mu'$ be as in (\ref{orthq}). For every $\seq\nu \in Q(\seq\Delta)$ we have that $\fu (\seq\mu , \seq\nu) = \fu (\seq\mu , \seq\mu') \sqcup \fu (\seq\mu' , \seq\nu)$.
   \item\label{geodq} Let $\seq\mu , \seq\nu$ be two points in $Q(\seq\Delta)$. Any geodesic in $(\AM , \td )$ joining $\seq\mu , \seq\nu$ is entirely contained in $Q(\seq\Delta)$.
 \end{enumerate}
 \end{lemma}

 \proof (\ref{inq}) Indeed if $\bu = (U_n)^\omega \in \fu (\seq\mu , \seq\nu)$ then $\lim_\omega (\dist_{C(U_n)} (\mu_n , \nu_n))=\infty$, according to Lemma \ref{restr}. On the other hand if $\bu \pitchfork \seq\Delta$ then $\dist_{C(U_n)} (\mu_n , \nu_n)=O(1)$, as the bases of both $\mu_n$ and $\nu_n$ contain $\Delta_n$.

 (\ref{orthq}) follows immediately from Lemma \ref{lem:duproj}.

 (\ref{outq}) According to (\ref{inq}) and (\ref{orthq}), $\fu
 (\seq\mu , \seq\mu') \cap \fu (\seq\mu' , \seq\nu)=\emptyset$. The
 triangle inequality implies that for every $\bu \in \fu (\seq\mu , \seq\mu')$ either $\tdlu (\seq\mu' , \seq\nu )>0$ or  $\tdlu (\seq\mu , \seq\nu )>0$. But since the former cannot occur it follows that $\bu \in \fu (\seq\mu , \seq\nu)$. Likewise we prove that $\fu (\seq\mu' , \seq\nu) \subset  \fu (\seq\mu , \seq\nu)$. The inclusion $\fu (\seq\mu , \seq\nu) \subset \fu (\seq\mu , \seq\mu') \sqcup \fu (\seq\mu' , \seq\nu)$ follows from the triangle inequality.

 (\ref{geodq}) follows from the fact that for any point $\seq\alpha$ on a $\td$-geodesic joining $\seq\mu , \seq\nu$, $\fu (\seq\mu , \seq\nu)= \fu (\seq\mu , \seq\alpha ) \cup \fu (\seq\alpha , \seq\nu)$, as well as from (\ref{inq}) and  (\ref{outq}). \endproof

 \begin{lemma}\label{fixedred}
 Let $g\in \mcgb$ and $\seq\Delta$ be as in Lemma \ref{red}.
 If $\seq\mu$ is fixed by $g$ then $\seq\mu \in Q(\seq\Delta)$.
 \end{lemma}

 \proof Assume on the contrary that $\seq\mu \not\in Q(\seq\Delta)$, and let $\seq\nu$ be its projection onto
 $Q(\seq\Delta)$. Then $g\seq\nu$ is the projection of $g\seq\mu$ onto $Q(\seq\Delta)$.
 Corollary \ref{distsubsurf} implies that $g\seq\nu = \seq\nu$. By replacing $g$
  with some power of it we may assume that the hypotheses of Lemma \ref{lem:betw} hold.
 On the other hand, the conclusion of Lemma \ref{lem:betw} does not
 hold since the geodesic between $\mu$ and $\nu$ and the geodesic
 between $g\mu$ and $g\nu$ coincide. This contradiction proves the
 lemma.
 \endproof

 \begin{lemma}\label{fixpower}
 Let $g\in \mcgb$ be a pure element such that $\la g \ra$ has bounded orbits in $\AM$, and let $\seq\mu$ be a point such that $g\seq\mu \neq \seq\mu$.
 Then for every $k\in \Z \setminus \{ 0\}$,  $g^k\seq\mu \neq \seq\mu$.
 \end{lemma}

 \proof \emph{Case 1.}\quad Assume that $g$ is a pseudo-Anosov element.

\emph{Case 1.a} Assume moreover that $\seq\mu$ is in a piece $P$ stabilized by $g$. Let $\bu $ be a subsurface in $\fu (\seq\mu , g\seq\mu)$. As $\seq\mu , g\seq\mu$ are both in $P$ it follows by Lemma \ref{2ptsinpiece} that $\bu \subsetneq \bs $.

Assume that the subsurfaces $g^{-i_1}
\bu ,$ $....,$ $g^{-i_k} \bu$ are also in $\fu (\seq\mu , g\seq\mu )$, where $i_1 < \cdots <i_k$. Let $3\epsilon
>0$ be the minimum of $\td_{g^{-i} \bu} (\seq\mu , g\seq\mu
), i=0,i_1,...,i_k$. Since $P$ is the closure of its interior $U(P)$
(Proposition~\ref{unique})
there exists $\seq\nu \in U(P)$ such
that $\td (\seq\mu , \seq\nu) < \epsilon $. It follows that $\td_{g^{-i} \bu} (\seq\nu , g\seq\nu )
\geq \epsilon$ for $i=0,i_1,...,i_k$. Then by Lemma \ref{restr}, $\lim_\omega \dist_{C(g^{-i} U_n)}
(\nu_n , g_n\nu_n )= \infty $. Let $\fh =\seqrep \fh $ be a hierarchy path joining $\seq\nu $ and
$g\seq\nu$. The above implies that \uass $\fh_n$ intersects $Q(g^{-i_j} \partial U_n )$, hence there
exists a vertex $v^j_n$ on the tight geodesic $\ft_n$ shadowed by $\fh_n$ such that $g^{-i_j}
U_n\subseteq S\setminus v^j_n$. In particular $\dcs (g^{-i_j}
\partial U_n, v^j_n )\leq 1$. Since $\seq\nu \in U(P)$ and $g$ stabilizes $U(P)$ it follows that $g\seq\nu \in
U(P)$, whence $\dcs (\nu_n , g_n\nu_n ) \leq D=D(g)$ \uas. In particular the length of the tight
geodesic $\ft_n$ is at most $D+2$ \uas.

According to \cite[Theorem 1.4]{Bowditch:tightgeod}, there exists $m=m(S)$ such that \uass
$g_n^m$ preserves a bi-infinite geodesic $\g_n$ in $C(S)$.
To denote $g^m$  we write $g_{1}$ for the sequence with terms $g_{1,n}$.

For every curve $\gamma$ let $ \gamma'$ be a projection of it on $\g_n$.  A standard hyperbolic
geometry argument implies that for every $i\geq 1$ $$\dcs (\gamma , g_{1,n}^{-i} \gamma )\geq \dcs (\gamma' , g_{1,n}^{-i} \gamma' ) + O(1)\geq i+O(1)\, . $$

The same estimate holds for $\gamma$ replaced by $\partial U_n$. Now assume that the maximal power $i_k
= m q +r$, where $0\leq r <m$. Then $\dcs (g_n^{-i_k}\partial U_n, g_n^{-mq}\partial U_n)=
\dcs (g_n^{-r}\partial U_n, \partial U_n)\leq 2(D+2)+ \dcs (g_n^{-r}\nu_n, \nu_n)\leq
2(D+2) +rD=D_1$. It follows that $\dcs (g_n^{-i_k}\partial U_n,
\partial U_n)\geq \dcs ( \partial U_n,
g_n^{-mq}\partial U_n) -D_1 \geq q + O(1) -D_1$.

On the other hand $\dcs (g_n^{-i_k}\partial U_n,
\partial U_n)\leq 2+ \dcs (v^k_n , v^0_n )\leq D+4$, whence $q\leq D+D_1+4+O(1)=D_2$ and $i_k\leq
m(D_2+1)$. Thus the sequence $i_1,...,i_k$ is bounded, and it has a maximal element. It follows that
there exists a subsurface $\bu $ in $\fu (\seq\mu , g\seq\mu )$ such that for every $k>0$,
$\td_{g^{-k} \bu} (\seq\nu , g\seq\nu )=\td_{\bu} (g^{k}\seq\nu , g^{k+1}\seq\nu )=0$. The
triangle inequality in $T_\bu$ implies that $\dlu (\seq\mu , g\seq\mu )= \dlu (\seq\mu , g^k\seq\mu )>0$ for every
$k\geq 1$. It follows that no power $g^k$ fixes $\seq\mu$.

\medskip

\emph{Case 1.b}\quad Assume now that $\seq\mu$ is not contained in any piece fixed by $g$. By Lemma
\ref{isompiece} $g$ fixes either the middle cut-piece $P$ or the middle cut-point $m$ of $\seq\mu,
g\seq\mu$.

Assume that $\seq\mu , g\seq\mu$ have a middle cut-piece $P$, and let $\seq\nu , \seq\nu'$ be the
endpoints of the intersection with $P$ of any arc joining $\seq\mu , g\seq\mu$ . Then $g \seq\nu =
\seq\nu'$ hence $g\seq\nu \neq \seq \nu$. By Case 1.a it then follows that for every $k\neq 0$,
$g^k\seq\nu \neq \seq \nu$, and since $[\seq\mu , \seq\nu ] \sqcup [\seq\nu , g^k\seq\nu] \sqcup
[g^k\seq\nu , g^k\seq\mu]$ is a geodesic, it follows that $g^k\seq\mu \neq \seq \mu$.

We now assume that $\seq\mu , g\seq\mu$ have a middle cut-point $m$, fixed by $g$. Assume that $\fu
(\seq\mu , g\seq\mu)$ contains a strict subsurface of $\bs$. Then the same thing holds for $\fu
(\seq\mu , m)$. Let $\bu \subsetneq \bs$ be an element in $\fu (\seq\mu , m)$.

If $g^k \seq\mu = \seq\mu$ for some $k \neq 0$, since $g^k m = m$ it follows that $g^{kn} (\bu )=\bu$ for some $n \neq 0$, by Lemma \ref{fixedpair}. But this is impossible, since $g$ is a pseudo-Anosov.

Thus, we may assume that $\fu (\seq\mu , g\seq\mu) = \{ \bs \}$, i.e. that $\seq\mu , g\seq\mu$ are in the same transversal tree.

Let $\g_n$ be a bi-infinite geodesic in $C(S)$ such that $g_n\g_n$ is at Hausdorff distance $O(1)$ from $\g_n$. Let $\gamma_n$ be the projection of $\pcs (\mu_n)$ onto $\g_n$. A hierarchy path $\seq\fh =\seqrep \fh$ joining $\mu_n$ and $g_n\mu_n$ contains two points $\nu_n , \nu_n'$ such that:
\begin{itemize}
  \item the sub-path with endpoints $\mu_n , \nu_n$ is at $C(S)$-distance $O(1)$ from any $C(S)$-geodesic joining $\pcs (\mu_n)$ and $\gamma_n$;
  \item the sub-path with endpoints $g\mu_n , \nu_n'$ is at   $C(S)$-distance $O(1)$ from any $C(S)$-geodesic joining $\pcs (g\mu_n)$ and $g\gamma_n$;
  \item if $\dcs (\nu_n , \nu_n') $ is large enough then the sub-path
  with endpoints $\nu_n , \nu_n'$ is  at   $C(S)$-distance $O(1)$ from $\g_n$;

  \item $\dcs (\nu_n' , g\nu_n) $ is $O(1)$.
\end{itemize}

Let $\seq\nu = \seqrep \nu$ and $\seq\nu' = \seqrep{\nu'}$. The last property above implies that $\tdls (\seq\nu' , g\seq\nu )=0$. Assume that $\tdls (\seq\nu , \seq\nu') >0$ hence $\tdls (\seq\nu , g\seq\nu ) >0$. Let $\fh'$ be a hierarchy sub-path with endpoints $\seq\nu , g\seq\nu$. Its projection onto $T_\bs$ and the projection of $g\fh'$ onto $T_\bs$ have in common only their endpoint. Otherwise there would exist $\seq\alpha $ on  $\fh'\cap g\fh'$ with $\tdls (\seq\alpha , g\seq\mu ) >0$, and such that $\Cutp \{ \seq\alpha , g\seq\mu \}$ is in the intersection of $\Cutp (\fh')$ with $\Cutp (g \fh')$. Consider $\seq \beta \in \Cutp \{ \seq\alpha , g\seq\mu \}$ at equal $\tdls$-distance from $\seq\alpha , g\seq\mu$.
 Take $\alpha_n , \beta_n$ on $\fh'_n$ and  $\alpha_n' , \beta_n'$ on $g\fh'_n$ such that $\seq\alpha = \seqrep{\alpha}= \seqrep{\alpha'}$ and  $\seq\beta = \seqrep{\beta}= \seqrep{\beta'}$. Since $\alpha_n , \alpha_n'$ and $\beta_n , \beta_n'$ are at distance $o(d_ n)$ it follows that $\fh_n'$ between $\alpha_n , \beta_n$ and $g\fh_n'$ between $\alpha_n' , \beta_n'$ share a large domain $U_n$. Let $\sigma_n$ and $\sigma_n'$ be the corresponding points on the two hierarchy sub-paths contained in $Q(\partial U_n)$. The projections of $\fh'_n$ and $g\fh'_n$ onto $C(S)$, both tight geodesics, would contain the points $\pcs (\sigma_n)$ and $\pcs (\sigma_n')$ at $\dcs$-distance $O(1)$ while $\lim_\omega \dcs (\sigma_n , g\nu_n)=\infty$ and $\lim_\omega \dcs (\sigma_n' , g\nu_n)=\infty$. This contradicts the fact that the projection of $\fh'_n \sqcup g\fh'_n$ is at $\dcs$-distance $O(1)$ from the geodesic $\g_n$.

We may thus conclude that the projections of $\fh'$ and $g\fh'$ on $T_\bs$ intersect only in their endpoints. From this fact one can easily deduce by induction that $\la g\ra $ has unbounded orbits in $T_\bs$, hence in $\free$.

Assume now that $\tdls (\seq\nu , \seq\nu') =0$ (hence $\seq\nu =m$) and that $\tdls (\seq\mu , \seq\nu) >0$. Let $\seq\alpha$ be the point on the hierarchy path joining $\seq\mu , \seq\nu$ at equal $\tdls$-distance from its extremities and let $\seq{\fh''}=\seqrep{\fh''}$ be the sub-path of endpoints $\seq\mu , \seq\alpha$. All the domains of $\fh''_n$ have $C(S)$ distance to $\g_n$ going to infinity, likewise for the $C(S)$ distance to any geodesic joining $\pcs (g^k\mu_n)$ and $\pcs (g^k\nu_n)$ with $k\neq 0$. It follows that $\dist (\seq\mu , g^k\seq\mu) \geq \dist (\seq\mu , \seq\alpha )>0$.

\medskip

\emph{Case 2.}\quad Assume now that $g$ is a reducible element, and let $\seq\Delta =(\Delta_n)^\omega$ be a multicurve
as in Lemma \ref{red}. According to the same lemma, $Q(\seq\Delta )\subset U(P)$.

If $\seq\mu \not\in Q(\seq\Delta )$ then $g^k \seq\mu \neq \seq\mu$ by Lemma \ref{fixedred}.
Assume therefore that  $\seq\mu \in Q(\seq\Delta )$. The set $Q(\seq\Delta )$ can be identified to $\prod_{i=1}^m \MM
  (\bu_i)$ and $\seq\mu$ can be therefore identified to $(\seq\mu_1,...,\seq\mu_m)$.
  If for every $i\in \{1,2,...,k\}$ the component of $g$ acting on $\bu_i$ would fix $\seq\mu_i$ in $\MM
  (\bu_i)$ then $g$ would fix $\seq\mu$. This would contradict the hypothesis on $g$.
   Thus for some $i\in \{1,2,...,k\}$
   the corresponding component
  of $g$ acts on $\MM (\bu_i)$ as a pseudo-Anosov and does not fix $\seq\mu_i$. According to the first case
  for every $k\in \Z \setminus \{ 0 \}$ the component of $g^k$ acting on $\bu_i$ does  not fix $\seq\mu_i$ either,
   hence $g^k$ does not fix  $\seq\mu$. \endproof

\begin{lemma}\label{split}
Let $g\in \mcgb$ be a pure element, and let $\seq\mu =\seqrep \mu$ be
a point such that $g\seq\mu \neq \seq\mu$.  If $g$ is reducible take
$\seq\Delta=(\Delta_n)^\omega$, and $U_n^1,....,U_n^m$ as in Lemma
\ref{red}, while if $g$ is pseudo-Anosov take $\seq\Delta = \emptyset$
and $\{ U_n^1,....,U_n^m \} =\{ S\} $, and by convention $Q(\Delta_n)
= \MM (S)$.  Assume that $g$ is such that for any $\nu_n \in
Q(\Delta_n)$, $\dist_{C(U^i_n)} (\nu_n, g_n\nu_n) >D$ \uass for every
$i\in \{1,...,k\}$, where $D$ is a fixed constant, depending only on
$\xi(S)$ (this may be achieved for instance by replacing $g$ with
a large enough power of it).

Then $\fu=\fu (\seq\mu , g\seq\mu )$ splits as $\fu_0 \sqcup \fu_1
\sqcup g\fu_1 \sqcup \mathfrak{P}$, where
\begin{itemize}
  \item $\fu_0$ is the set of elements $\bu\in \fu$ such that no
  $g^k\bu$ with $k\in \Z \setminus \{0 \}$ is in $\fu$,
  \item $\mathfrak{P}$ is the
  intersection of $\fu$ with $\{\bu^1, ...,\bu^k \}$, where $\bu^j =
  (U^j_n)^\omega$,
  \item $\fu_1$ is the set of elements $\bu\in \fu\setminus \mathfrak{P}$
  such that
  $g^k\bu\in \fu$ only for $k=0,1$ (hence $g\fu_1$ is the set of
  elements $\bu\in \fu\setminus \mathfrak{P}$
  such that $g^k\bu\in \fu$ only for $k=0,-1$).
\end{itemize}

Moreover, if either $\fu_0\neq \emptyset$ or $\tdlu (\seq\mu ,
g\seq\mu )\neq \td_{g\bu }(\seq\mu , g\seq\mu )$ for some $\bu \in
\fu_1$ then the $\la g\ra$-orbit of $\seq\mu$ is unbounded.
\end{lemma}

\proof \emph{Case 1.} \quad Assume that $g$ is a pseudo-Anosov with $C(S)$-translation length $D$,
where $D$ is a large enough constant. There exists a bi-infinite axis $\pgot_n$ such that $g_n\pgot_n$ is at Hausdorff distance $O(1)$ from $\pgot_n$. Consider $\fh =\seqrep \fh$ a hierarchy path joining $\seq\mu$ and $g\seq\mu$, such
that $\fh_n$ shadows a tight geodesic $\ft_n$. Choose two points $\gamma_n , \gamma_n'$ on $\pgot_n$ that are nearest
to $\pcs (\mu_n )$, and $\pcs (g\mu_n )$ respectively. Note that $\dcs (\gamma_n' , g_n \gamma_n)=O(1)$.

Standard arguments concerning the way hyperbolic elements act on hyperbolic metric
spaces imply that the geodesic $\ft_n$ is in a tubular neighborhood with radius $O(1)$ of the union of
$C(S)$-geodesics $[\pcs (\mu_n ), \gamma_n] \sqcup [\gamma_n , \gamma_n' ] \sqcup
[\gamma_n' , g_n\pcs (\mu_n )]$. Moreover any point on $\ft_n$ has any nearest point projection on
$\pgot_n$ at distance $O(1)$ from $[\gamma_n , \gamma_n' ] \subset \pgot_n$.

Now let $\bu =(U_n)^\omega$ be a subsurface in $\fu (\seq\mu , g\seq\mu )$, $\bu \subsetneq \bs $. Assume that for
some $i\in \Z$, $\tdlu (g^i\seq\mu , g^{i+1}\seq\mu )>0$. This implies that $\lim_\omega \dist_{C(U_n)} (g_n^j\mu_n , g_n^{j+1}\mu_n)=+\infty $ for $j\in
\{0,i\}$,  according to Lemma \ref{restr}. In particular, by Lemma \ref{MM2:LLL}, $\partial U_n$ is at $C(S)$-distance $\leq 1$ from a
vertex $u_n\in \ft_n$ and $g_n^{-i}\partial U_n$ is at $C(S)$-distance $\leq 1$ from a vertex $v_n\in
\ft_n$. It follows from the above that  $\partial U_n$ and  $g_n^{-i}\partial U_n$  have any nearest point projection on
$\pgot_n$ at distance $O(1)$ from $[\gamma_n , \gamma_n' ] \subset \pgot_n$. Let $x_n$ be a nearest point projection on
$\pgot_n$ of $\partial U_n$. Then $g_n^{-i} x_n$ is a nearest point projection on
$\pgot_n$ of $g_n^{-i} \partial U_n$. As both $x_n$ and $g_n^{-i}x_n$ are at distance $O(1)$ from $[\gamma_n , \gamma_n' ]$,
 they are at distance at most $D+O(1)$ from each other. On the other hand $\dcs (x_n , g_n^{-i} x_n )= |i|D+O(1)$.
 For $D$ large enough this implies that $i\in \{-1,0,1\}$. Moreover for $i=-1$, $\partial U_n$ projects on $\pgot_n$ at $C(S)$-distance $O(1)$ from
$\gamma_n$ while
$g_n\partial U_n$ projects on $\pgot_n$ at $C(S)$-distance $O(1)$ from $g_n\gamma_n$.
This in particular implies
that, for $D$ large enough, either $\tdlu (g\seq\mu , g^{2}\seq\mu )>0$ or  $\tdlu (g^{-1}\seq\mu , \seq\mu )>0$ but not both.

Let $\fu_0 = \fu\setminus (g\fu \cup g\iv \fu )$. Let $\fu_1 = \left( \fu \cap g\iv \fu \right)\setminus \{ \bs \}$ and $\fu_2 = \left( \fu \cap g
\fu \right) \setminus \{ \bs \}$. Clearly $\fu = \fu_0 \cup \fu_1 \cup \fu_2 \cup \mathfrak{P}$, where $\mathfrak{P}$ is either $\emptyset$ or $\{ \bs \}$. Since $g\iv \fu \cap g \fu$ is either empty or $\{ \bs \}$, $\fu_0 , \fu_1, \fu_2 , \mathfrak{P}$ are pairwise disjoint, and $\fu_2 = g\fu_1$.

\smallskip

Assume that $\fu_0$ is non-empty, and let $\bu$ be an element in $\fu_0$. Then
$d=\tdlu (\seq\mu , g\seq\mu )>0$ and $\tdlu (g^i\seq\mu , g^{i+1}\seq\mu )=0$ for every
$i\in \Z \setminus \{ 0\}$. Indeed if there existed $i\in \Z \setminus \{ 0\}$ such that
$\tdlu (g^i\seq\mu , g^{i+1}\seq\mu )>0$ then, by the choice of $D$ large enough, either $i= -1 $ or $i=1 $,
therefore either $\bu \in g\fu_1$ or $\bu \in \fu_1$, both contradicting the fact that $\bu \in \fu_0$.
The triangle inequality then implies that for every $i\leq 0 < j$, $\tdlu (g^i\seq\mu , g^{j}\seq\mu )=d$.
 Moreover for every $i\leq k\leq j$,  by applying $g^{-k}$ to the previous equality we deduce that
 $\td_{g^k \bu }(g^i\seq\mu , g^{j}\seq\mu )=d$. Thus for every $i\leq 0 < j$ the distance
 $\td (g^i\seq\mu , g^{j}\seq\mu )$ is at least
 $\sum_{i\leq k\leq j} \td_{g^k \bu }(g^i\seq\mu , g^{j}\seq\mu ) = (j-i)d$. This implies that the
 $\la g\ra$-orbit of $\seq\mu$ is unbounded.

\smallskip

Assume that $\tdlu (\seq\mu , g\seq\mu )\neq \td_{g\bu }(\seq\mu , g\seq\mu )$ for some $\bu \in \fu_1$. Then the distance $\tdlu (g\iv \seq\mu , g\seq\mu )$ is at least $ |\tdlu (\seq\mu , g\seq\mu )- \td_{g\bu }(\seq\mu , g\seq\mu )|=d>0$. Moreover since $\tdlu (g^k\seq\mu , g^{k+1}\seq\mu )=0$ for every $k\geq 1$ and $k\leq -2$, it follows that $\tdlu (g^{-k}\seq\mu , g^m \seq\mu )= \tdlu (g\iv \seq\mu , g\seq\mu ) >d$ for every $k, m\geq 1$. We then obtain that for every $\bv = g^j\bu$ with $j\in \{-k+1,....,m-1\}$, $\tdlv (g^{-k}\seq\mu , g^m \seq\mu ) >d$. Since $\bu \subsetneq \bs$ and $g$ is a pseudo-Anosov, it follows that if $i\neq j$ then $g^i\bu\neq g^j\bu$. Then $\td (g^{-k}\seq\mu , g^m \seq\mu ) \geq \sum_{j=-k+1}^{m-1} \td_{g^j\bu} (g^{-k}\seq\mu , g^m \seq\mu )\geq (k+m-1)d$. Hence the $\la g\ra$-orbit of $\seq\mu$ is unbounded.


\medskip

\emph{Case 2.} \quad Assume that $g$ is reducible.
Let $\seq\nu$ be the projection of $\seq\mu$ onto $Q(\seq\Delta)$. Consequently $g\seq\nu$ is the projection of $g\seq\mu$ onto $Q(\seq\Delta)$. Lemma \ref{uqd} implies that $\fu (\seq\mu , g\seq\mu)= \fu (\seq\mu , \seq\nu) \cup \fu (\seq\nu , g\seq\nu) \cup \fu (g\seq\nu , g\seq\mu)$.

Consider an element $\bu\in \fu$, $\bu \not\in \{\bu_1,...\bu_m\}$, and assume that for
some $i \in \Z \setminus \{ 0\}$,  $g^i \bu \in \fu$. We prove that $i\in \{-1,0,1\}$.

Assume that $\bu \in \fu (\seq\nu , g\seq\nu)$.  Then, since $\lim_\omega \dist_{C(U_n)}(\mu_n , g\mu_n )=+\infty $, it follows that $\bu \not\pitchfork \seq\Delta$ and $\bu$ is contained in $\bu^j$ for some $j\in \{1,...,k\}$. Either $\bu = \bu^j \in \mathfrak{P}$ or $\bu \subsetneq \bu^j$. In the latter case, an argument as in Case 1 implies that
for $D$ large enough $i\in \{-1,0,1\}$.

Assume that $\bu\in \fu (\seq\mu , \seq\nu)$. Then $\bu \pitchfork \seq\Delta$, since $\seq\mu$ and $\seq\nu$ do not differ inside the subsurfaces $\bu^j\, ,\, j=1,...,m$. Since $\seq\Delta = \bigcup_{j=1}^k \partial \bu^j$ it follows that for some $j\in \{1,...,k\}$,
$\bu \pitchfork \partial \bu^j $.

We have that  $\td_{g^i \bu} (g^i\seq\mu , g^i\seq\nu )>0$, hence a hierarchy path joining $g_n^i\mu_n$ and  $g_n^i\nu_n$
contains a point $\beta_n$ in $Q(g_n^i\partial U_n)$.

The hypothesis that $\td_{g^i \bu} (\seq\mu , g\seq\mu )>0$ implies that either
 $\td_{g^i \bu} (\seq\mu , \seq\nu)>0$ or $\td_{g^i \bu} (g\seq\nu , g\seq\mu)>0$.
  Assume that  $\td_{g^i \bu} (\seq\mu , \seq\nu)>0$. Then a hierarchy path joining
  $\mu_n$ and  $\nu_n$ also contains a point $\beta_n'$ in $Q(g_n^i\partial U_n)$.

 For the element $j\in \{1,...,k\}$ such that $\bu \pitchfork \partial \bu^j $, $\dist_{C(U_n^j)}(\beta_n , \beta_n')=O(1)$ since both $\beta_n$ and $\beta_n'$ contain the multicurve
 $\partial U_n$. By properties of projections, $\dist_{C(U_n^j)}(\mu_n, \nu_n)=O(1)$ and
 $\dist_{C(U_n^j)}(g_n^i\mu_n, g_n^i\nu_n)=O(1)$, which implies that
 $\dist_{C(U_n^j)}(\beta_n, g_n^i\nu_n)=O(1)$ and $\dist_{C(U_n^j)}(\beta_n', \nu_n)=O(1)$. It follows that the distance $\dist_{C(U_n^j)}(g_n^i\nu_n, \nu_n)$ has order $O(1)$. On the other hand
 $\dist_{C(U_n^j)}(g_n^i\nu_n, \nu_n)> |i|D$. For $D$ large enough this implies that $i=0$.

 Assume that $\td_{g^i \bu} (g\seq\nu , g\seq\mu)>0$. This and the fact that
  $\td_{g \bu} (g\seq\nu , g\seq\mu)>0$ imply as in the previous argument,  with $\seq\mu, \seq\nu $ and $\bu$ replaced by $g\seq\mu, g\seq\nu $ and $g\bu$, that $i=1$.

 The case when  $\tdlu (g\seq\nu , g\seq\mu)>0$ is dealt with similarly. In this case it follows that,
 if $\td_{g^i \bu} (\seq\mu , \seq\nu)>0$ then $i=-1$, and if $\td_{g^i \bu} (g\seq\nu , g\seq\mu)>0$ then $i=0$.

 We have thus proved that for every $\bu\in \fu$, $\bu \not\in \{\bu_1,...\bu_m\}$, if for
some $i \in \Z \setminus \{ 0\}$,  $g^i \bu \in \fu$ then $i\in \{-1,0,1\}$. We take $\mathfrak{P} = \fu \cap \{\bu_1,...\bu_m\}$ and $\fu' = \fu \setminus \mathfrak{P}$. We define $\fu_0 = \fu' \setminus (g\fu' \cup g\iv \fu' )$. Let $\fu_1 = \fu' \cap g\iv \fu' $ and $\fu_2 = \fu' \cap g
\fu'$. Clearly $\fu = \fu_0 \cup \fu_1 \cup \fu_2 \cup \mathfrak{P}$. Since $g\iv \fu' \cap g \fu'$ is empty, $\fu_0 , \fu_1, \fu_2 , \mathfrak{P}$ are pairwise disjoint, and $\fu_2 = g\fu_1$.

If $\fu_0\neq \emptyset$ then a proof as in Case 1 yields that the $\la g\ra$-orbit of $\seq\mu$ is unbounded.

Assume that $\tdlu (\seq\mu , g\seq\mu )\neq \td_{g\bu }(\seq\mu , g\seq\mu )$ for some $\bu \in \fu_1$.
It follows from the previous argument that $\bu \in \fu (\seq\nu , g\seq\nu )$, hence $g\bu$ is in the same set. Without loss of generality we may therefore replace $\seq\mu$ by $\seq\nu$ and assume that $\seq\mu \in Q(\Delta)$. In particular $\tdlu (\seq\mu , g\seq\mu )$ is composed only of subsurfaces that do not intersect $\Delta$. We proceed as in Case 1 and prove that the $\la g\ra$-orbit of $\seq\mu$ is unbounded. \endproof


\begin{lemma}\label{middle}
Let $g=(g_n)^\omega\in \mcgb$ be a pseudo-Anosov fixing a piece $P$,
such that $\la g\ra$ has bounded orbits in $\AM$.  Assume that \uass
the translation length of $g_n$ on $C(S)$ is larger than a
uniformly chosen constant depending
only on $\xi(S)$.  Then for any point $\seq\mu$ in $P$
and for any hierarchy path $\seq\fh$ connecting $\seq\mu$ and its
translate $g \seq\mu$, the isometry $g$ fixes the midpoint of
$\seq\fh$.
\end{lemma}

\proof Let $\seq\mu$ be an arbitrary point in $P$ and $\fh =\seqrep
\fh$ a hierarchy path joining $\seq\mu$ and $g\seq\mu$, such that
$\fh_n$ shadows a tight geodesic $\ft_n$.  We may assume that
$g\seq\mu \neq \seq\mu$, and consider the splitting defined in Lemma
\ref{split}, $\fu=\fu (\seq\mu , g\seq\mu )=\fu_0 \sqcup \fu_1 \sqcup
g\fu_1$.  Note that since $\seq\mu$ and $g\seq\mu$ are both in the
same piece $P$, $\fu (\seq\mu , g\seq\mu )$ cannot contain $\bs$, by
Lemma \ref{2ptsinpiece}.  As $\la g\ra $ has bounded orbits, we may
assume that $\fu_0$ is empty, and that $\fu = \fu_1 \cup g\fu_1$.  We
denote $g\fu_1$ also by $\fu_2$.  For every $\bu \in \fu$ choose a
sequence $(U_n)$ representing it, and define $\fu (n),$ $\fu_1 (n),$
$\fu_2(n)$ as the set of $U_n$ corresponding to $\bu$ in $\fu , \fu_1,
\fu_2$ respectively.

Let $\alpha_n$ be the last point on the hierarchy path $\fh_n$
belonging to $Q(\partial U_n)$ for some $U_n \in \fu_1 (n)$.  Let
$\seq\alpha = \seqrep \alpha$.  Assume that $g\seq\alpha \neq
\seq\alpha$.  For every subsurface $\bv=(V_n)^\omega \in \upss$ such
that $\tdlv (\seq\alpha , g\seq\alpha )>0$ it follows by the triangle
inequality that either $\tdlv (\seq\alpha , g\seq\mu )>0$ or $\tdlv
(g\seq\mu , g\seq\alpha )>0$.  In the first case $\bv \in \fu$.  If $\bv \in
\fu_1$ then $\bv = (U_n)^\omega$ for one of the chosen sequences
$(U_n)$ representing an element in $\fu_1$, whence $\lim_\omega
\dist_{C(U_n)} (\alpha_n , g_n\mu_n )= \infty$ and the hierarchy
sub-path of $\fh_n$ between $\alpha_n$ and $g_n\mu_n $ has a large
intersection with $Q(\partial U_n)$.  This contradicts the choice of
$\alpha_n$.  Thus in this case we must have that $\bv \in g\fu_1$.

We now consider the second case, where
$\tdlv (g\seq\mu , g\seq\alpha )>0$. Since this condition is
equivalent to $\td_{g\iv \bv } (\seq\mu , \seq\alpha )>0$ it
follows that $g\iv \bv \in \fu$.  Moreover \uass the hierarchy
sub-path of $\fh_n$ between $\mu_n$ and $\alpha_n$ has a large
intersection with $Q(g_n\iv \partial V_n)$.

Define $\pgot_n$ and the points $\gamma_n , \gamma_n'$ on $\pgot_n$ as in Case 1 of the proof of Lemma \ref{split}. The argument in that proof shows that for every $\bu=(U_n)^\omega \in \fu_1$, \uass
$\partial U_n$ has any nearest point projection on $\pgot_n$ at $C(S)$-distance $O(1)$ from $\gamma_n$
while $g_n\partial U_n$ has any nearest point projection on $\pgot_n$ at $C(S)$-distance $O(1)$
from $g_n\gamma_n$. In particular $\alpha_n$ has any nearest point projection on $\pgot_n$ at $C(S)$-distance $O(1)$
from $\gamma_n$ whence $g_n\iv \partial V_n$  has any nearest point projection on $\pgot_n$ at
$C(S)$-distance $O(1)$ from $\gamma_n$ too. For sufficiently large
translation length (i.e. the constant in the hypothesis of the
lemma), this implies that $g_n\iv \partial V_n$  cannot have a nearest point projection on $\pgot_n$ at
$C(S)$-distance $O(1)$ from $g\gamma_n$. Thus \uass $g_n\iv  V_n \not\in \fu_2 (n)$, therefore $g\iv \bv \not\in \fu_2$. It
follows that $g\iv \bv \in \fu_1$, whence $\bv \in g\fu_1$.

We have thus obtained that $\tdlv (\seq\alpha , g\seq\alpha ) >0$ implies that $\bv \in g
\fu_1$, therefore for every $k\in \Z$, $\tdlv (g^k\seq\alpha , g^{k+1}\seq\alpha ) >0$ implies that
$\bv \in g^{k+1} \fu_1$. Since the collections of subsurfaces $g^{i} \fu_1$ and
$g^{j} \fu_1$ are disjoint for $i\neq j$ it follows that $\td (g^{-i}\seq\alpha , g^j \seq\alpha
)=\sum_{k=-i}^{j-1} \sum_{\bv \in g^{k+1}\fu_1} \tdlv (g^{-i}\seq\alpha , g^j \seq\alpha ) =
\sum_{k=-i}^{j-1} \sum_{\bv \in g^{k+1}\fu_1 } \tdlv (g^k\seq\alpha , g^{k+1}\seq\alpha ) =
(j+i-1) \td ( \seq\alpha , g \seq\alpha )$. This implies that the $\la g \ra$-orbit of $\seq\alpha$ is unbounded,
 contradicting our hypothesis.

 Therefore $\seq\alpha = g \seq\alpha $. From this, the fact that $g$ acts as an isometry on $(\AM , \td)$, and
 the fact that hierarchy paths are geodesics in $(\AM , \td)$, it follows that $\seq\alpha$ is
 the midpoint of $\seq\fh$.\endproof

 \medskip

\begin{lemma}\label{middlepA}
Let $g=(g_n)^\omega\in \mcgb$ be a pseudo-Anosov such that $\la g\ra $
has bounded orbits in $\AM$.  Assume that \uass the translation length
of $g_n$ on $C(S)$ is larger than a uniformly chosen constant depending
only on $\xi(S)$. Then for any point
 $\seq\mu$ and for any hierarchy path $\seq\fh$ connecting $\seq\mu$ and
 $g \seq\mu$, the isometry $g$ fixes the midpoint of  $\seq\fh$.
\end{lemma}

\proof Let $\seq\mu$ be an arbitrary point in $\AM$ and assume
$g\seq\mu \neq \seq\mu$.  Lemma \ref{isompiece}, (\ref{fixpoint}),
implies that $g$ fixes either the middle cut-point or the middle cut-piece of $\seq\mu , g\seq\mu$.  In the former case we are done.  In
the latter case consider $P$ the middle cut-piece, $\seq\nu$ and
$\seq\nu'$ the entrance and respectively exit points of $\seq\fh$ from
$P$.  Then $\seq\nu' = g\seq\nu \neq \seq\nu$ and we may apply Lemma
\ref{middle} to $g$ and $\seq\nu$ to finish the argument.
\endproof

\medskip

\begin{lemma}\label{fixpA}
Let $g=(g_n)^\omega\in \mcgb$ be a pseudo-Anosov. The set of fixed points of $g$ is either empty or it is a convex subset of a transversal tree of $\AM$.
\end{lemma}

\proof Assume there exists a point $\seq\mu\in \AM$ fixed by $g$. Let $\seq\nu$ be another point fixed by $g$. Since $g$ is an isometry permuting pieces, this and property $(T_2')$ implies that $g$ fixes every point in $\Cutp \{\seq\mu , \seq\nu \}$. If a geodesic (any geodesic) joining $\seq\mu$ and $\seq\nu$ has a  non-trivial intersection $[\seq\alpha,\seq\beta ]$ with a piece then $\seq\alpha,\seq\beta$ are also fixed by $g$. By Lemma \ref{2ptsinpiece}, $\fu (\seq\alpha,\seq\beta )$ contains a proper subsurface $\bu \subsetneq \bs$, and by Lemma \ref{fixedpair}, $g\bu =\bu$, which is impossible.

It follows that any geodesic joining $\seq\mu$ and $\seq\nu$ intersects all pieces in points. This means that the set of points fixed by $g$ is contained in the transversal tree $T_{\seq\mu }$ (as defined in Definition \ref{ttrees}). It is clearly a convex subset of $T_{\seq\mu }$.\endproof

\medskip

\begin{lemma}\label{middlered}
Let $g\in \mcgb$ be a reducible element such that $\la g \ra$ has
bounded orbits in $\AM$, and let $\seq\Delta =(\Delta_n)^\omega$ and
$\bu_1=(U^1_n)^\omega,...,\bu_m=(U^m_n)^\omega$ be the multicurve and
the subsurfaces associated to $g$ as in Lemma \ref{red}.  Assume that
for any $i\in \{1,2,...,k \}$ and for any $\seq\nu \in Q(\seq\Delta )$ the distance $\dist_{C(U_i)} (\seq\nu ,
g\seq\nu )$ is larger than some sufficiently large constant, $D$,
depending only on $\xi(S)$.

Then for any point $\seq\mu$ there exists a geodesic in $(\AM, \td )$
connecting $\seq\mu$ and its translate $g \seq\mu$, such that the
isometry $g$ fixes its midpoint.
\end{lemma}

\proof Let $\seq\mu$ be an arbitrary point in $\AM$.  By means of
Lemma \ref{isompiece}, (\ref{fixpoint}), we may reduce the argument to
the case when $\seq\mu$ is contained in a piece $P$ fixed set-wise by
$g$.  Lemma \ref{red} implies that $U(P)$ contains $Q(\seq\Delta)$.
Let $\seq\nu$ be the projection of $\seq\mu$ onto $Q(\seq\Delta )$.
According to Lemma \ref{lem:betw}, if $D$ is large enough then given
$\fh_1,\fh_2,\fh_3$ hierarchy paths connecting respectively $\seq\mu ,
\seq\nu$, $\seq\nu , g\seq\nu$, and $g\seq\nu , g\seq\mu$,
$\fh_1\sqcup \fh_2 \sqcup \fh_3$ is a geodesic in $(\AM ,\td)$
connecting $\seq\mu$ and $g\seq\mu$.

If $\seq\nu=g\seq\nu$ then we are done.  If not, we apply Lemma
\ref{middlepA} to $g$ restricted to each $\bu^j$ and we find a point
$\seq\alpha$ between $\seq\nu$ and $g\seq\nu$ fixed by $g$.  Since
both $\seq\nu$ and $g\seq\nu$ are between $\seq\mu$ and $g\seq\mu$ it
follows that $\seq\alpha$ is between $\seq\mu$ and $g\seq\mu$, hence
on a geodesic in $(\AM ,\td)$ connecting them.\endproof

\begin{lemma}\label{fixred}
Let $g\in \mcgb$ be a reducible element, and let $\seq\Delta =(\Delta_n)^\omega$ and $\bu_1=(U^1_n)^\omega,...,\bu_m=(U^m_n)^\omega$ be the
multicurve and the subsurfaces associated to $g$ as in Lemma
\ref{red}.

If the set $\Fix(g)$ of points fixed by $g$ contains a point $\seq\mu$
then, when identifying $Q(\seq\Delta )$ with $\MM (\bu_1) \times
\cdots \times \MM (\bu_m)$ and correspondingly $\sm$ with a point
$(\sm_1,...,\sm_m)$, $\Fix (g)$ identifies with $C_1\times \cdots
\times C_k \times \MM (\bu_{k+1})\times \cdots \times \MM (\bu_m)$,
where $C_i$ is a convex subset contained in the transversal tree
$T_{\seq\mu_i}$.
\end{lemma}

\proof This follows immediately from the fact that $\Fix (g) = \Fix (g(1))\times \cdots \times \Fix (g(m))$, where $g(i)$ is the restriction of $g$ to the subsurface $\bu_i$, and from Lemma \ref{fixpA}.\endproof

\begin{lemma}\label{splitandmid}
Let $g$ be a pure element such that $\la g\ra$ has bounded orbits in $\AM$.  Let $\sm$
be a point in $\AM$ such that $g\sm \neq \sm$ and let $\seq m$ be a
midpoint of a $\td$-geodesic joining $\sm$ and $g\sm$, $\seq m$ fixed by
$g$.  Then, in the splitting of $\fu (\sm , g\sm )$ given by Lemma
\ref{split}, the set $\fu_1$ coincides with $\fu (\sm , \seq m) \setminus
\fp$.
\end{lemma}

\proof As $\la g\ra$ has bounded orbits, we have that $\fu_0=\emptyset$, according to the last part of the statement of Lemma \ref{split}.

Since $\seq m$ is on a geodesic joining $\sm$ and $g\sm$, $\fu (\sm , g\sm ) = \fu (\sm , \seq m ) \cup \fu (\seq m , g\sm )$. From the definition of $\fu_1$ it follows that  $\fu (\sm , \seq m) \setminus \fp$ is contained in $\fu_1$. Also, if an element $\bu\in \fu_1$ would be contained in $\fu (\seq m , g\sm )$ then it would follow that $g\iv \bu $ is also in $\fu$, a contradiction.\endproof

\medskip

\Notat \quad In what follows, for any reducible element $t\in \mcgb$ we denote by $\seq\Delta_t$ the multicurve associated to $t$ as in Lemma \ref{red}.

\begin{lemma}\label{minfix}
\begin{enumerate}
  \item\label{minfixpt} Let $g$ be a pure element with $\Fix (g)$ non-empty. For every $\seq x \in \AM$ there exists a unique point $\seq y \in \Fix(g)$ such that $\td (\seq x , \seq y )= \td (\seq x , \Fix(g) )$.

  \item\label{minfix2} Let $g$ and $h$ be two pure elements not fixing a common multicurve. If $\Fix (g)$ and $\Fix (h)$ are non-empty then there exists a unique pair of points $\sm \in \Fix (g)$ and $\sn \in \Fix (h)$ such that $\td(\sm , \sn )=\td (\Fix (g), \Fix (h))$.

Moreover, for every $\sa \in \Fix (g)$, $\td (\sa , \sn )= \td (\sa , \Fix(h) )$, and $\sn$ is the unique point with this property; likewise for every $\sb \in \Fix (h)$, $\td (\sb , \sm )= \td (\sb , \Fix(g) )$ and $\sm$ is the unique point with this property.
\end{enumerate}
\end{lemma}

\proof We identify $\AM$ with a subset of the product of trees
$\prod_{\bu \in \upss} T_\bu$.  Let $g$ be a pure element with $\Fix
(g)$ non-empty.  By Lemma \ref{fixedpair},
for any $\bu$ such that $g(\bu )\neq \bu $ we have that
the projection of $\Fix(g)$
onto $T_\bu$ is a point which we denote $\sm_\bu$.  If $\bu$ is such that $\bu \pitchfork \delg$ then
the projection of $\Fix(g)$, and indeed of $Q(\delg )$ onto $T_\bu$
also reduces to a point, by Lemma \ref{uqd}, (\ref{inq}).  The only
surfaces $\bu $ such that $g(\bu )=\bu$ and $\bu \notpitchfork \delg$
are $\bu_1, ..., \bu_k$ and $\by \subseteq \bu_j$ with $j\in
\{k+1,...,m\}$, where $\bu_1,...,\bu_m$ are the subsurfaces determined
on $\bs$ by $\delg$, $g$ restricted to $\bu_1, ..., \bu_k$ is a
pseudo-Anosov, $g$ restricted to $\bu_{k+1}, ..., \bu_m$ is identity.
By Lemma \ref{fixred}, the projection of $\Fix (g)$ onto $T_{\bu_i}$
is a convex tree $C_{\bu_i}$, when $i=1,...,k$, and the projection of
$\Fix (g)$ onto $T_{\by}$ with $\by \subseteq \bu_j$ and $j\in
\{k+1,...,m\}$ is $T_\by$.

\medskip

(\ref{minfixpt})\quad The point $\seq x$ in $\AM$ is identified to the
element $( {\seq x}_\bu)_\bu$ in the product of trees $\prod_{\bu \in
\upss} T_\bu$.

For every $i\in \{1,...,k\}$ we choose the unique point ${\seq
y}_{\bu_i} $ in the tree $C_{\bu_i}$ realizing the distance from
${\seq x}_{\bu_i}$ to that tree.  The point ${\seq y}_{\bu_i} $ lifts
to a unique point ${\seq y}_i$ in the transversal sub-tree $C_i$.

Let $i \in \{ k+1,..., m\}$ and let ${\seq y}_{\bu_i} =({\seq x}_\by
)_\by $ be the projection of $( {\seq x}_\bu)_\bu$ onto $\prod_{\by
\subseteq \bu_i} T_\by$.  Now the projection of $\AM$ onto $\prod_{\by
\subseteq \bu_i} T_\by$ coincides with the embedded image of $\MM
(\bu_i )$, since for every $\seq x \in \AM$ its projection in $T_\by$
coincides with the projection of $\pi_{\MM (\bu_i)} (\seq x)$.
Therefore there exists a unique element ${\seq y}_i \in \MM (\bu_i )$
such that its image in $\prod_{\by \subseteq \bu_i} T_\by$ is ${\seq
y}_{\bu_i}$.  Note that the point ${\seq y}_i$ can also be found as
the projection of $\seq x$ onto $\MM (\bu_i)$.

Let $\seq z$ be an arbitrary point in $\Fix (g)$.  For every
subsurface $\bu$ the point $\seq z$ has the property that $\tdlu (\seq
z , \seq x ) \geq \tdlu (\seq y , \seq x)$.  Moreover if $\seq z \neq
\seq y$ then there exist at least one subsurface $\bv$ with
$g(\bv)=\bv$ and $\bv \notpitchfork \delg $ such that ${\seq z}_\bv
\neq {\seq y}_\bv$.  By the choice of ${\seq y}_\bv$ it follows that
$\tdlv ({\seq z}_\bv , {\seq x}_\bv) > \tdlv ({\seq y}_\bv , {\seq
x}_\bv)$.  Therefore $\td ({\seq z}, {\seq x}) \geq \td ({\seq y} ,
{\seq x} )$, and the inequality is strict if ${\seq z}\neq {\seq y}$.

\medskip

(\ref{minfix2})\quad Let $\bv_1,...,\bv_s$ be the subsurfaces determined on $\bs$ by $\delh$, such that $h$ restricted to $\bv_1, ..., \bv_l$ is a pseudo-Anosov, $h$ restricted to  $\bv_{l+1}, ..., \bv_s$ is identity. The projection of $\Fix (h)$ onto $T_{\bv_i}$ is a convex tree $C_{\bv_i}$, when $i=1,...,l$, the projection of $\Fix (g)$ onto $T_{\bz}$ with $\bz \subseteq \bv_j$ and $j\in \{l+1,...,s\}$ is $T_\bz$, and for any other subsurface $\bu$ the projection of $\Fix(h)$ is one point $\sn_\bu$.

For every $i\in \{1,...,k\}$ $\Fix (h)$ projects onto a point $\sn_{\bu_i}$ by the hypothesis that $g,h$ do not fix a common multicurve (hence a common subsurface). Consider $\sm_{\bu_i }$ the nearest to $\sn_{\bu_i}$ point in the convex tree $C_{\bu_i}$. This point lifts to a unique point $\sm_i$ in the transversal sub-tree $C_i$. Let $i \in \{ k+1,..., m\}$. On $\prod_{\by \subseteq \bu_i} T_\by$ $\Fix(h)$ projects onto a unique point, since it has a unique projection in each $T_\by$.  As pointed out already in the proof of (\ref{minfixpt}), the projection of $\AM$ onto $\prod_{\by \subseteq \bu_i} T_\by$ coincides with the embedded image of $\MM (\bu_i )$. Therefore there exists a unique element $\sm_i \in \MM (\bu_i )$ such that its image in $\prod_{\by \subseteq \bu_i} T_\by$ is $(\sn_\by )_\by$. Note that the point $\sm_i$ can also be found as the unique point which is the projection of $\Fix (h)$ onto $\MM (\bu_i)$ for $i=k+1,...,m$. We consider the point $\sm =(\sm_1,...., \sm_m)\in C_1\times \cdots \times C_k \times \MM (\bu_{k+1}) \times \cdots \times \MM (\bu_m)$. Let $\sa$ be an arbitrary point in $\Fix (g)$ and let $\sb$ be an arbitrary point in $\Fix (h)$. For every subsurface $\bu$ the point $\sm$ has the property that $\tdlu (\sm , \sb ) \leq \tdlu (\sa , \sb)$. Moreover if $\sa \neq \sm$ then there exist at least one subsurface $\bv$ with $g(\bv)=\bv$ and $\bv \notpitchfork \delg $ such that $\sa_\bv \neq \sm_\bv$. By the choice of $\sm_\bv$ it follows that $\tdlv (\sm_\bv , \sb_\bv) < \tdlv (\sa_\bv , \sb_\bv)$. Therefore $\td (\sm , \sb ) \leq \td (\sa , \sb )$, and the inequality is strict if $\sa\neq \sm$.

We construct similarly a point $\sn \in \Fix (h)$. Then $\td (\sm , \sn )\leq \td (\sm , \sb )\leq \td (\sa , \sb)$ for any $\sa \in \Fix (g)$ and $\sb \in \Fix (h)$. Moreover the first inequality is strict if $\sb \neq \sn$, and the second inequality is strict if $\sa \neq \sm $.\endproof

\begin{lemma}\label{lem:power}
Let $g\in \mcgb$ be a pure element satisfying the hypotheses from Lemma \ref{split}, and moreover assume that $\la g\ra$ has bounded orbits, whence $\Fix (g) \neq \emptyset$, by Lemmas \ref{middlepA} and \ref{middlered}. Let $\sm$ be an element such that $g\sm \neq \sm$ and let $\sn$ be the unique projection of $\sm$ onto $\Fix (g)$ defined in Lemma \ref{minfix}, (\ref{minfixpt}).

Then for every $k\in \Z \setminus \{0\}$, $\sn$ is on a geodesic joining $\sm$ and $g^k \sm$.
\end{lemma}

\proof By  Lemmas \ref{middlepA} and \ref{middlered} there exists $\seq m$ middle of a geodesic joining $\sm$ and $g^k \sm$ such that ${\seq m} \in \Fix (g^k)$. By Lemma \ref{fixpower}, $\Fix (g^k) = \Fix (g)$. Assume that $\seq m \neq \sn$. Then by Lemma \ref{minfix}, (\ref{minfixpt}), $\td (\sm , \sn ) < \td (\sm , \seq m)$. Then $\td (\sm , g^k \sm ) \leq \td (\sm , \sn ) + \td (\sn , g^k \sm )= 2\td (\sm , \sn ) < 2\td (\sm , \seq m)=\td (\sm , g^k \sm )$, which is impossible.\endproof

\begin{lemma}\label{transv}
Let $g=(g_n)^\omega$ and  $h=(h_n)^\omega$ be two pure reducible elements in $\mcgb$, such that they do not both fix a multicurve. If  a proper subsurface $\bu $ has the property that $h(\bu ) =\bu$ then
\begin{enumerate}
  \item\label{2inters} $g^m \bu \pitchfork \delh$ for $|m|\geq N=N(g)$;
  \item\label{2fixed}  the equality $h (g^k (\bu )) = g^k (\bu ) $ can hold only for finitely many $k\in \Z$.
\end{enumerate}
\end{lemma}

\proof (\ref{2inters}) Assume by contradiction that $g^m \bu \notpitchfork \delh$ for $|m|$ large. Since $h(\bu )=\bu$ it follows that $\bu$ must overlap a component $\bv$ of $\bs \setminus \seq\Delta_g$ on which $g$ is a pseudo-Anosov (otherwise $g\bu =\bu $). If $\delh$ would also intersect $\bv$ then the projections of $\Delta_{h,n}$ and of $\partial U_n$ onto the curve complex $C(V_n)$ would be at distance $O(1)$. On the other hand, since $\dist_{C(V_n)} (g^m \partial U_n , \partial U_n ) \geq |m| +O(1)$ it follows that for $|m|$ large enough $\dist_{C(V_n)} (g^m \partial U_n, \Delta_{h,n} ) >3$, that is $g^{m}\partial\bu$ would intersect $\delh$, a contradiction. Thus $\delh$ does not intersect $\bv$. It follows that $\bu$ does not have all boundary components from $\delh$, thus the only possibility for $h(\bu )=\bu$ to be achieved is that $\bu$ is a finite union of subsurfaces determined by $\delh$ and subsurfaces contained in a component of $\bs \setminus \delh$ on which $h$ is identity. Since $\bv$ intersects $\bu$ and not $\delh$, $\bv$ intersects only a subsurface $\bu_1 \subseteq \bu $ restricted to which $h$ is identity; moreover $\bv$ is in the same component of $\bs \setminus \delh$ as $\bu_1$. Therefore $h\bv=\bv$, and we also had that $g\bv=\bv$, a contradiction.

\medskip

(\ref{2fixed}) Assume that  $h (g^k (\bu )) = g^k (\bu ) $ holds for infinitely many $k\in \Z$. Without loss of generality we may assume that all $k$ are positive integers and that for all $k$, $g^k \bu \pitchfork \delh$. Up to taking a subsequence of $k$ we may assume that there exist $\bu_1,...,\bu_m$ subsurfaces determined by $\delh$ and $1\leq r\leq m$ such that $h$ restricted to $\bu_1,...,\bu_r$ is either a pseudo-Anosov or identity, $h$ restricted to $\bu_{r+1},...,\bu_m$ is identity, and $g^k(\bu )= \bu_1\cup ...\cup \bu_r\cup \bv_{r+1}(k)\cup ...\cup \bv_m(k)$, where $\bv_j(k) \subsetneq \bu_j$ for $j=r+1,...,m$. The boundary of $g^k(\bu )$ decomposes as $\partial'S \sqcup \delh'\sqcup \partial_k$, where $\partial'S$ is the part of $\partial g^k(\bu )$ contained in $\partial S$, $\delh'$ is the part contained in $\delh$, and $\partial_k$ is the remaining part (coming from the subsurfaces $\bv_j(k)$). Up to taking a subsequence and pre-composing with some $g^{-k_0}$ we may assume that $\bu= \bu_1\cup ...\cup \bu_r\cup \bv_{r+1}(0)\cup ...\cup \bv_m(0)$ and that $g^k$ do not permute the boundary components. It follows that $\delh'= \emptyset$, hence $\partial_k \neq \emptyset$. Take a boundary curve $\gamma \in \partial_0$. Then $\gamma \in \partial \bv_j (0)$ for some $j\in \{ r+1 , ..., m\}$, and for every $k$, $g^k\gamma \in \partial \bv_j (k) \subset \bu_j $, in particular $g^k\gamma \notpitchfork \delh$. An argument as in (\ref{2inters}) yields a contradiction.\endproof

\medskip

\begin{lemma}\label{2redgen}
Let $g=(g_n)^\omega$ and  $h=(h_n)^\omega$ be two pure elements in $\mcgb$, such that $\la g,h \ra$ is composed only of pure elements and its orbits in $\AM$ are bounded. Then $g$ and $h$ fix a point.
\end{lemma}

\proof \textbf{(1)} \quad Assume that $g$ and $h$ do not fix a common multicurve. We argue by contradiction and assume that $g$ and $h$ do not fix a point and we shall deduce from this that  $\la g,h \ra$ has unbounded orbits.

Since $g$ and $h$ do not fix a point, by Lemma \ref{minfix},
(\ref{minfix2}), $\Fix (g)$ and $\Fix (h)$ do not intersect, therefore
the $\td$-distance between them is $d>0$.  Let $\sm \in \Fix (g)$ and
$\sn \in \Fix (h)$ be the unique pair of points realizing this
distance $d$, according to Lemma \ref{minfix}.  Possibly by replacing
$g$ and $h$ by some powers we may assume that $g, h$ and all their
powers have the property that each pseudo-Anosov components has
sufficiently large translation
lengths in their respective curve complexes.

\medskip

\textbf{(1.a)} \quad We prove that for every $\sa\in \AM$ and every $\epsilon >0$ there exists $k$ such that $g^k(\sa )$ projects onto $\Fix (h)$ at distance at most $\epsilon$ from $\sn$.

Let $\sm_1$ be the unique projection of $\sa $ on $\Fix (g)$, as defined in Lemma \ref{minfix}, (\ref{minfixpt}).
According to Lemma \ref{lem:power}, $\sm_1$ is on a geodesic joining $\sa$ and $p\sa$ for every $p\in \la g\ra \setminus \{ id \}$. Let $\sn_1$ be the unique point on $\Fix (h)$ that is nearest to $p(\sa)$, whose existence is ensured by Lemma \ref{minfix}, (\ref{minfixpt}).

By Lemma \ref{split} $\fu (\sa , p(\sa) ) = \fu_1^p  \sqcup p\fu_1^p \sqcup \fp$. Moreover, by Lemma \ref{splitandmid}, $\fu_1^p = \fu (\sa , \sm_1) \setminus \fp$, therefore $\fu_1^p$ is independent of the power $p$. Therefore we shall henceforth denote it simply by $\fu_1$.

Let $\bu$ be a subsurface in $\fu (\sn , \sn_1)$. If $\bu$ is a pseudo-Anosov component of $h$ then the projection of $\Fix (h)$ onto $T_{\bu}$ is a subtree $C_{\bu}$,  the whole set $\Fix (g)$ projects onto a point $\sm_\bu$, and $\sn_{\bu}$ is the projection of $\sm_\bu$ onto $C_\bu$, $(\sn_1)_\bu$ is the projection of $(p(\sa))_\bu$ onto  $C_{\bu}$, and $\sn_{\bu}\, ,\, (\sn_1)_\bu$ are distinct. It follows that the geodesic joining $(\sm_1)_{\bu}$  and  $(p(\sa))_\bu$ covers the geodesic joining $\sn_{\bu}$ and $(\sn_1)_\bu$, whence $\tdlu (\sm_1 , p(\sa )) \geq \tdlu (\sn , \sn_1)$.

If $\bu$ is a subsurface of an identity component of $h$ then the projection of $\Fix (h)$ onto $T_{\bu}$ is the whole tree $T_{\bu}$, $\Fix (g)$ projects onto a unique point $\sm_{\bu}=\sn_\bu$ and $(\sn_1)_\bu=(p(\sa))_\bu$. It follows that the geodesic joining $(\sm_1)_{\bu}$  and  $(p(\sa))_\bu$ is the same as the geodesic joining $\sn_{\bu}$ and $(\sn_1)_\bu$, whence $\tdlu (\sm_1 , p(\sa )) = \tdlu (\sn , \sn_1)$.

Thus in both cases $\tdlu (\sm_1 , p(\sa )) \geq \tdlu (\sn , \sn_1)>0$, in particular $\bu \in \fu (\sm_1 , p(\sa ))$. Since $g$ and $h$ do not fix a common subsurface, $\fu (\sn , \sn_1) \cap \fp =\emptyset$, therefore $\fu (\sn , \sn_1) \subset \fu (\sm_1 , p(\sa )) \setminus \fp = p \fu_1 $. The last equality holds by Lemma \ref{splitandmid}.

Now consider $\bv_1,...,\bv_r$ subsurfaces in $\fu_1$ such that the sum
$$
\sum_{j=1}^r \left( \td_{\bv_j} (\sa , p(\sa) )+ \td_{g\bv_j} (\sa , p(\sa) )\right) +
\sum_{\bu \in \fp } \td_{\bu} (\sa , p(\sa) ) )
$$ is at least $\td (\sa , p(\sa) ) -\epsilon\, .$

According to Lemma \ref{transv}, (\ref{2fixed}), by taking $p$ a large enough power of $g$ we may ensure that $h(p(\bv_j))\neq p(\bv_j)$ for every $j=1,...,r$. Then
$$\td (\sn , \sn_1) = \sum_{\bu \in \fu (\sn , \sn_1) } \tdlu (\sn , \sn_1) \leq \sum_{\bu \in \fu (\sn , \sn_1) } \tdlu (\sm_1 , p(\sa )) \leq
$$
$$\sum_{\bu \in p \fu_1 , \bu \neq p \bv_j } \tdlu (\sm_1 , p(\sa ))= \sum_{\bu \in p \fu_1, \bu \neq p\bv_j } \tdlu (\sa , p(\sa ))\leq \epsilon \, . $$

\medskip

\textbf{(1.b)} \quad In a similar way we prove that for every $\sb\in \AM$ and every $\delta >0$ there exists $m$ such that $h^m(\sb )$ projects onto $\Fix (g)$ at distance at most $\delta$ from $\sm$.

\medskip

\textbf{(1.c)} \quad We now prove by induction on $k$ that for every $\epsilon >0$ there exists a word $w$ in $g$ and $h$ such that:
\begin{itemize}
  \item $\td (\sn , w \sn )$ is in the interval $[2kd - \epsilon , 2kd]$;
  \item $\td (\sm , w \sn )$ is in the interval $[(2k-1)d - \epsilon , (2k-1)d]$;
  \item $w\sn $ projects onto $\Fix (h)$ at distance at most $\epsilon$ from $\sn$.
\end{itemize}

This will show that the $\sn$-orbit of $\la g,h \ra$ is unbounded, contradicting the hypothesis.

Take $k=1$. Then (1.a) applied to $\sn $ and $\epsilon$ implies that there exists a power $p$ of $g$ such that $p\sn$ projects onto $\Fix (h)$ at distance at most $\epsilon$ from $\sn$. Note that  by Lemma \ref{lem:power}, $\sm$ is the middle of a geodesic joining $\sn , p\sn$, hence $\td (\sn , p\sn )= 2d$ and $\td (\sm , p\sn )= d$.

Assume that the statement is true for $k$, and consider $\epsilon >0$ arbitrary. The induction hypothesis applied to $\epsilon_1= \frac{\epsilon}{16}$ produces a word $w$ in $g$ and $h$. Property (1.b) applied to $\sb = w \sn$ implies that there exists a power $h^m$ such that $h^m w \sn $ projects onto $\Fix (g)$ at distance at most $\delta=\frac{\epsilon}{4}$.

The distance $\td (h^m w \sn , \sn )$ is equal to $\td ( w \sn , \sn )$, hence it is in $[2kd - \epsilon_1 , 2kd]$. The distance $\td (h^m w \sn , \sm )$ is at most $\td (h^m w \sn , \sn ) +d = (2k+1)d$. Also $\td (h^m w \sn , \sm )\geq \td (h^m w\sn , w\sn )- \td (w\sn , \sm ) \geq 2 (\td (w\sn , \sn )- \epsilon_1) - (2k-1)d \geq 2(2kd -2 \epsilon_1)- (2k-1)d = (2k+1)d-4\epsilon_1 \geq (2k+1)d-\epsilon$.

We apply (1.a) to $\sa = h^m w\sn$ and $\epsilon$ and obtain that for some $k$, $g^kh^m w\sn$ projects onto $\Fix (h)$ at distance at most $\epsilon$ from $\sn$. Take $w'=g^kh^m w$. We have $\td (\sm , w' \sn )= \td (\sm , h^m w \sn )$, and the latter is in $[(2k+1)d-\epsilon , (2k+1)d]$.

The distance $\td (\sn , w' \sn )$ is at most $\td (\sm , w' \sn )+d$, hence at most $(2k+2)d$. Also $\td (\sn , w' \sn )$ is at least $\td (w'\sn , h^m w\sn )- \td (h^m w\sn  , \sn)\geq 2(\td (h^m w\sn , \sm )-\delta )-2kd \geq 2((2k+1)d - 4\epsilon_1 - \delta )-2kd = (2k+2)d - 8 \epsilon_1-2\delta = (2k+2)d - \epsilon $.

\medskip

\textbf{(2)} \quad Let $\seq\Delta$ be a multicurve fixed by both $g$ and $h$, and let $\bu_1,...\bu_n$ be the subsurfaces determined by $\seq\Delta$. The restrictions of  $g$ and $h$ to each $\bu_i$, $g(i)$ and $h(i)$, do not fix any multicurve. By (1), $g(i)$ and $h(i)$ fix a point $\seq\nu_i$ in $\MM (\bu_i)$. It then follows that $g$ and $h$ fix the point $(\seq\nu_1,...\seq\nu_n)\in \MM (\bu_1)\times \cdots \times \MM (\bu_n) =Q(\seq\Delta)$. \endproof

\medskip

We are now ready to prove Proposition \ref{nredgen}.

\medskip

\noindent \textit{Proof of Proposition \ref{nredgen}.}\quad  According to Lemma \ref{gpisompiece} it suffices to prove the following statement: if $g_1$,...,$g_m$ are pure elements in $\mcgb$, such that $\la g_1,...,g_m \ra$ is composed only of pure elements, its orbits in $\AM$ are bounded and it fixes set-wise a piece $P$ then $g_1,...,g_m$ fix a point in $P$. We prove this statement by induction on $k$. For $k=1$ and $k=2$ it follows from Lemma \ref{2redgen}. Note that if an isometry of a tree-graded space fixes a point $x$ and a piece $P$ then it fixes the projection of $x$ on $P$.

Assume by induction that the statement is true for $k$ elements, and
consider $g_1,...,g_{k+1}$ pure elements in $\mcgb$, such that $\la
g_1,...,g_{k+1} \ra$ is composed only of pure elements, its orbits in
$\AM$ are bounded and it fixes set-wise a piece $P$.
\medskip

\textbf{(1)}\quad Assume that $g_1,...,g_{k+1}$ do not fix a common
multicurve.  By the induction hypothesis $g_1,...,g_{k-2}, g_{k-1},
g_k$ fixes a point $\seq\alpha \in P$, $g_1,...,g_{k-2}, g_{k-1},
g_{k+1}$ fixes a point $\seq\beta \in P$ and $g_1,...,g_{k-2}, g_{k},
g_{k+1}$ fixes a point $\seq\gamma \in P$.  If $\seq\alpha , \seq\beta
, \seq\gamma$ are not pairwise distinct then we are done.  Assume
therefore that $\seq\alpha , \seq\beta , \seq\gamma$ are pairwise
distinct, and let $\seq\mu$ be their unique median point.  Since
pieces are convex in tree-graded spaces, $\seq\mu \in P$.  For $i\in
\{1,..., k-2\}$, $g_i$ fixes each of the points $\seq\alpha ,
\seq\beta , \seq\gamma$, hence it fixes their median point $\seq\mu$.

Assume that $g_{k-1}\seq\mu \neq \seq\mu $. Then $\fu (\seq\mu , g_{k-1}\seq\mu ) \subset \fu (\seq\mu , \seq\alpha ) \cup \fu (\seq\alpha , g_{k-1} \seq\mu) = \fu (\seq\mu , \seq\alpha ) \cup g_{k-1} \fu (\seq\mu , \seq\alpha ) $. Now $\fu (\seq\mu , \seq\alpha ) \subset \fu (\seq\beta , \seq\alpha )$, and since $g_{k-1}$ fixes both $\seq\beta $ and $ \seq\alpha$ it fixes every subsurface $\bu \in \fu (\seq\beta , \seq\alpha )$, by Lemma \ref{fixedpair}. In particular $g_{k-1} \fu (\seq\mu , \seq\alpha ) = \fu (\seq\mu , \seq\alpha )$. Hence $\fu (\seq\mu , g_{k-1}\seq\mu ) \subset \fu (\seq\mu , \seq\alpha )$. A similar argument implies that $\fu (\seq\mu , g_{k-1}\seq\mu ) \subset \fu (\seq\mu , \seq\beta ) $. Take $\bv \in (\seq\mu , g_{k-1}\seq\mu )$. Then $\bv \in \fu (\seq\mu , \seq\alpha )$. In particular $\bv \in \fu (\seq\beta , \seq\alpha )$, hence each $g_i$ with $i=1,2,..., k-1$ fixes $\bv$, since it fixes the points $\seq\beta , \seq\alpha$. Also $\bv \in \fu (\seq\gamma , \seq\alpha )$, whence $g_k \bv = \bv$. Finally, as $\bv \in \fu (\seq\mu , \seq\beta ) \subset \fu (\seq\gamma, \seq\beta )$ it follows that $g_{k+1} \bv =\bv$. This contradicts the hypothesis that $g_1,...,g_{k+1}$ do not fix a common multicurve. Note that $\bv \subsetneq \bs$ by Lemma \ref{2ptsinpiece}, since $\sa , \sb$ are in the same piece and $\bv \in \fu (\sa , \sb)$.

We conclude that $g_{k-1}\seq\mu = \seq\mu $. Similar arguments imply that $g_{k}\seq\mu = \seq\mu $ and $g_{k+1}\seq\mu = \seq\mu $.

\medskip

\textbf{(2)}\quad Assume that $g_1,...,g_{k+1}$ fix a common multicurve.  Let $\seq\Delta$ be this multicurve, and let $\bu_1,...\bu_m$ be the subsurfaces determined by $\seq\Delta$. According to Lemma \ref{red}, $Q(\seq\Delta) \subset U(P)$.

The restrictions of $g_1,...,g_{k+1}$ to each $\bu_i$, $g_1(i),...,g_{k+1}(i)$, do not fix any multicurve. By Lemma \ref{gpisompiece} either $g_1(i),...,g_{k+1}(i)$ fix a point $\seq\nu_i$ in $\MM (\bu_i)$ or they fix set-wise a piece $P_i$  in $\MM (\bu_i)$. In the latter case, by (1) we may conclude that $g_1(i),...,g_{k+1}(i)$ fix a point $\seq\nu_i \in P_i$.

It then follows that $g_1,...,g_{k+1}$ fix the point $(\seq\nu_1,...\seq\nu_n)\in \MM (\bu_1)\times \cdots \times \MM (\bu_m) =Q(\seq\Delta)\subset U(P)$.\hspace*{\fill}$\Box$



\providecommand{\bysame}{\leavevmode\hbox to3em{\hrulefill}\thinspace}
\providecommand{\MR}{\relax\ifhmode\unskip\space\fi MR }
\providecommand{\MRhref}[2]{%
  \href{http://www.ams.org/mathscinet-getitem?mr=#1}{#2}
}
\providecommand{\href}[2]{#2}

\end{document}